\documentclass[]{report}%

\usepackage{setspace}
\usepackage{mathtools}
\usepackage{amsmath}
\usepackage{amsthm}
\usepackage{amssymb}
\usepackage{amsbsy}
\usepackage{amsfonts}
\usepackage{amstext}
\usepackage{amscd}
\usepackage{nccmath}
\usepackage{bbm}
\usepackage{graphicx}
\usepackage{color}
\usepackage{enumerate}
\usepackage[utf8]{inputenc}
\usepackage{epsfig}%
\usepackage{hyperref}
\usepackage{float}

\usepackage{faktor}
\usepackage{quiver}
\usepackage{appendix}

\usepackage{tocloft}
\allowdisplaybreaks

\newtheorem{theorem}{Theorem}[section]
\newtheorem{proposition}[theorem]{Proposition}
\newtheorem{corollary}[theorem]{Corollary}
\newtheorem{lemma}[theorem]{Lemma}

\theoremstyle{definition}
\newtheorem{definition}[theorem]{Definition}

\newtheorem{example}[theorem]{Example}
\newtheorem{remark}[theorem]{Remark}

\newtheorem{conjecture}[theorem]{Conjecture}

\renewcommand{\Re}{\operatorname{Re}} 
\renewcommand{\Im}{\operatorname{Im}}

\newcommand{\dist}{\operatorname{dist}}
\newcommand{\norm}[1]{\lVert #1 \rVert}

\DeclareMathOperator{\diam}{diam}
\DeclareMathOperator{\interior}{int}

\begin{document} 

\begin{titlepage}
\begin{center}
\large{
On the Intersections of Homogeneous Self-similar\\ Sets with their Translates in $\mathbb{R}^{n}$ and \\ a Formulation of Multiplicative Invariance in $\mathbb{Z}^{n}$}
\end{center}
\vspace{1cm}
\begin{center}
by
\end{center}
\vspace{1cm}
\begin{center}
Neil Austin MacVicar
\end{center}
\vspace{1cm}
\begin{center}
A thesis submitted to the \\
Department of Mathematics and Statistics \\
in conformity with the requirements for \\
the degree of Doctor of Philosophy
\end{center}
\vspace{1cm}
\begin{center}
Queen's University\\
Kingston, Ontario, Canada \\
July 2025
\end{center}
\vspace{1cm}
\begin{center}
Copyright \textcopyright\space Neil Austin MacVicar, 2025
\end{center}
\end{titlepage}

\addcontentsline{toc}{chapter}{Abstract}
\chapter*{Abstract}

This thesis generalizes the study of $C\cap(C+\alpha)$ where $C$ is the middle third Cantor set to self-affine sets in $\mathbb{R}^{n}$. We present sufficient and necessary conditions for when the translation $\alpha$ produces a self-affine intersection for a particular class of self-affine sets. In the case where the attractor is self-similar, we improve results concerning the function from $\alpha$ to the dimension of the intersection. This lends itself to a case study of the complex number system $(-n+i, \{0, 1, \ldots, n^{2}\})$, when $n$ is an integer greater than or equal to $2$. Lastly, we present a definition of multiplicative invariance for subsets of $\mathbb{Z}^{n}$ and establish a connection, known in the one-dimensional case, between them and invariants sets of the $n$-dimensional torus. 

\cleardoublepage
\addcontentsline{toc}{chapter}{Acknowledgements}
\chapter*{Acknowledgements}

First, I thank my committee for their time and energy. I give special thanks to my advisors Professor Jamie Mingo and Professor Francesco Cellarosi for their expertise, wisdom, patience, and encouragement, without which I would not have gotten very far. I also want to thank Professor J\"{o}rg Thuswaldner at the University of Leoben for being generous with his time, both over email and in person, when I had questions about finding neighbours of integral digit tiles. I thank my parents and sister for their support despite not knowing what I spend my time thinking about or why anyone would willingly pursue a project like this. I thank my peers who have been the source of many great conversations and experiences, both mathematical and otherwise: Nic, Julia, Luke, Sonja, Sasha, Quinn, Skye, Lauryn, and Chris. In particular, I must acknowledge the special support that I received from Matt, Becca, and Tariq. As for my friends outside of the department, I'd be bankrupt if Joy and Melissa made me payback the food and gasoline they spent on me during my time in Kingston. Similar things can be said of the inhabitants and frequent visitors of ``the dog house" on Division street: Dakota, Samson, Scott, Kyle, Johanna, and Mairead. They all did their best to keep me sane through the writing process and this document would never have seen completion without them. 

\begin{singlespace}
\tableofcontents
\end{singlespace}
\newpage
\begin{singlespace}
{%
\let\oldnumberline\numberline%
\renewcommand{\numberline}{\figurename~\oldnumberline}%
\listoffigures%
}
\addcontentsline{toc}{chapter}{List of Figures}
\end{singlespace}

\cleardoublepage\pagenumbering{arabic} 


\chapter{Introduction}\label{chp:Intro}

If geometry is the study of shapes, then fractal geometry might be called the study of shapes that exhibit complexity at all scalings. Consider your favourite polygon. If we zoom in on a point on the boundary of the polygon, we eventually lose sight of the polygon. A sufficiently small neighbourhood of a point on an edge will only capture a line segment. If the point is a vertex, we may reproduce the polygon if we know it is regular, but otherwise all we can see is a vertex. 

These are not the kind of shapes fractal geometry is concerned with. Fractal shapes, or simply fractals, are shapes that are ``rough" in a way that their complexity is not hidden after zooming in on a point. In some cases, we even maintain sight of the original fractal. Consider the subset of the plane in Figure~\ref{fig:SierpinskiTriangle}. 

\begin{figure}
\centering
\includegraphics[scale=0.45]{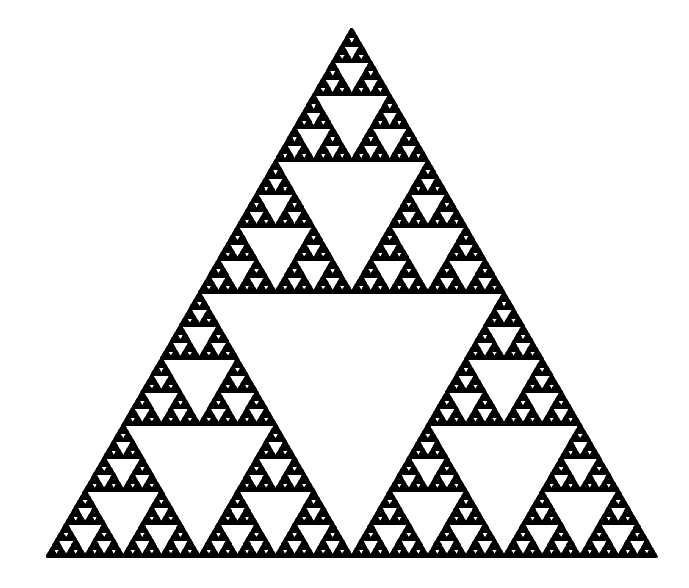}
\caption{The Sierpinski Triangle}
\label{fig:SierpinskiTriangle}
\end{figure}

Figure~\ref{fig:SierpinskiTriangle} is the Sierpinski triangle. Its construction begins with an equilateral triangle. The first step is the removal of an equilateral triangle that is a quarter of the area of the original triangle in such a way that three equilateral triangles, of size equal to what was removed, remain. At each stage, the procedure is performed on all of the smallest triangles that remain. If $a$ is the area of the original triangle, then the area remaining after $k\geq 1$ iterations is $a - \sum\limits_{j=1}^{k}3^{j-1}/4^{j}a = a(3/4)^{k}$. This tends to zero as $k$ tends to infinity. The area of the Sierpinski triangle must be zero since it is a subset of the shape at the $k$th-stage. On the other hand, we can keep track of the total length of the edges of the shape at the $k$th-stage. If $s$ is the length of a side of the original triangle, this is the quantity $3s + \sum\limits_{j=1}^{\infty}(3/2)^{j}s = 3s + 3s((3/2)^{k} - 1)$. This becomes arbitrarily large as $k$ tends to infinity. In this sense, the Sierpinski triangle has infinite length. 

These features make it seem that the Sierpinski triangle is between dimension one and dimension two. The word dimension here refers to the imprecise colloquial usage of calling shapes one dimensional if they only exhibit length, two dimensional if they boast length and width, and three dimensional if they have length, width, and depth. This concept of dimension is capturing the kind of measure that produces meaningful information that can be used to compare it with other like objects. That is to say, a line segment is not two dimensional because it has zero area and is best compared with other line segments. We do not think of a line of length 5 units as larger than a square of area 1 units-squared. In fact, we tend to think of the square as larger because it has area.

Fractal dimensions are a classification of size that assigns shapes a value on the continuum of the real line by extending the classification of dimension by volume-measure. Hausdorff dimension, specifically, accomplishes this by using a continuum of volume-measures as opposed to the common discrete spectrum of length, area, and three-dimensional volume. It is a distance-based concept that can be defined on any metric space. 

We provide a precise description of Hausdorff dimension in Section~\ref{sec:dimTheory} (see the first three definitions). For now, understand that the Hausdorff dimension is the critical point $r$ at the which the volume-measures, indexed by the nonnegative real numbers, assign a shape (a set of points) the value of positive infinity at all indices less than $r$ and the value zero at all indices greater than $r$. This is an extension of the idea that the index $3$ is larger than the dimension of a square because the square has zero $3$-dimensional volume and the index $1$ is too small because a square can be thought of as having infinite length. The latter can be interpreted as a consequence of the interior of the square containing uncountably many disjoint line segments. 

In summary, Hausdorff dimension is the parameter at which our measuring tool is just right. Hausdorff dimension reassuringly assigns non-fractals such as line segments, squares, and cubes dimension one, two, and three respectively, but it can be difficult to determine its value when assessing a ``true fractal" whose Hausdorff dimension is a non-integer value. The Hausdorff dimension of the Sierpinski Triangle is about $1.58$. Somewhere between $1$ and $2$. The $1.58$ish-dimensional volume of the Sierpinski triangle, when the side length is one, is somewhere between $0.67$ and $0.82$, but the exact value is unknown \cite{J07}. 

The Hausdorff dimension of the Sierpinski triangle is computable because it is among a collection of sets that are rough enough to be considered fractals, but not $\emph{too}$ rough. Consider a tree in a forest. We may verbally distinguish between the trunk and its branches, but we could imagine all of the branches to be their own trunks from which other trunks spring. From this perspective, a tree is made up of scaled-down versions of itself. This concept is called self-similarity. 

\begin{definition}\label{def:streamSelfSim} 
A nonempty compact set $S\subset\mathbb{R}^{n}$ is \textit{self-similar} if there exists a finite collection of functions $\{f_{i}\}_{i=1}^{N}$, $f_{i}:\mathbb{R}^{n}\rightarrow\mathbb{R}^{n}$, such that
\begin{itemize}
\item[(i)] $\bigcup\limits_{i=1}^{N}f_{i}(S) = S$,
\item[(ii)] for each $i=1, 2, \ldots, N$, there exists $c_{i}\in[0, 1)$ such that 
$\norm{f_{i}(x) - f_{i}(y)} = c_{i}\norm{x-y}$ for all $x, y\in\mathbb{R}^{n}$. 
\end{itemize}
\end{definition}

\begin{remark}
This definition is formulated with respect to Euclidean space, where the norm above is the standard norm. In Chapter~\ref{chp:IFS}, we present sets of this kind in greater generality and discuss how they are generated by collections of functions with the properties in Definition~\ref{def:streamSelfSim}. 
\end{remark}

Self-similarity is a desirable fractal property. There are non-integer dimensions other than Hausdorff dimension, such as packing dimension and box-counting dimension, whose values do not always agree with the Hausdorff dimension. They do for self-similar sets. Furthermore, under certain conditions (Theorem~\ref{thm:SepSimDim}) the dimension can be expressed algebraically in terms of the $c_{i}$. In many ways, self-similar sets are foundational to fractal geometry. It is worth asking how much of that structure is maintained after transformations and set-wise operations. 

Let us consider the following construction often given in a first course in real analysis. We start by taking the closed unit interval, $[0, 1]$, and remove the middle third to produce the union of $[0, 1/3]$ and $[2/3, 1]$. Let $C_{1}$ denote that union. We then remove the middle thirds of the components of $C_{1}$ to obtain the union of $[0, 1/9]$, $[2/9, 1/3]$, $[2/3, 7/9]$, and $[8/9, 1]$. We label this union $C_{2}$. Iterating this procedure produces a sequence of unions of intervals $C_{n}$ for which $C_{n+1} \subset C_{n}$. Let $C$ denote the countable intersection of the $C_{n}$. The set $C$ is called the middle third Cantor set. It is not empty. In fact, it has uncountably many points, yet also has zero ``length" (precisely, $1$-dimensional Lebesgue measure zero). The set $C$ is a self-similar set and has Hausdorff dimension $\log{2}/\log{3}$.  

In 1995, Davis and Hu and examined the structure of the intersections of the middle third Cantor set with translations of itself \cite{DH95}. The paper treats questions about the cardinality and Hausdorff dimension of the intersections. Soon after Davis and Hu's paper was published, Moreira published a paper about intersections of Cantor sets that appear in the study of homoclinic bifurcations of dynamical systems \cite{M96}, and eventually collaborated with Yoccoz on the subject \cite{MY01}. This seemingly brought more attention to the intersection of Cantor sets because over the next decade a series of papers about the set $T \cap (T+\alpha)$ where $T$ is a self-similar subset of the unit interval were published. This line of papers has a specific theme: Given a self-similar set $T\subset  [0, 1]$, for which translations $\alpha$ is $T \cap (T+\alpha)$ self-similar? We chronicle some of the key developments.  

Sufficient and necessary conditions on $\alpha$ were given by Deng, He, and Wen in 2008 when $T$ is the midde third Cantor set \cite{DHW08}. That same year, Zou, Lu, and Li generalized the statement from ``middle third" to middle $(1-2\beta)$ where $\beta\in(1/3, 1/2)$ \cite{ZLL08}. In 2011, Li, Yao, and Zhang generalized to multiple gaps which are uniformly spaced from each other \cite{ZLY11}. In 2014, both Pedersen and Philips \cite{PP14} and Kong \cite{K14} generalized the treatment using other partitions of the interval. In \cite{K14}, Kong treated what he called generalized Cantor sets. These are sets containing expansions of the form $\sum\limits_{j=1}^{\infty}d_{j}\beta^{j}$, where $\beta\in(0, 1)$ and $(d_{j})_{j=1}^{\infty}$ is a sequence of integers. We demonstrate the connection between these sets and $T \cap (T+\alpha)$ in Chapter~\ref{chp:highDimSEP}. Other variations on this theme from that time are treated in \cite{KLD11} and \cite{ZLY11} involving $\beta$-expansions and nonhomogeneous Cantor sets (gaps of different sizes) respectively. 

For now, we want to draw attention to the fact that all of these papers concern self-similar sets on the real line. This raises two questions. Firstly, are there higher dimensional versions of their findings? Secondly, can we modify existing strategies to prove them?

We take inspiration from sets that are self-similar with respect to the collection of functions $f_{d}(x) = (-n+i)^{-1}(x+d)$, $d\in\{0, 1, \ldots, n^{2}\}$, for some fixed positive integer $n$. Division by the Gaussian integer $-n+i$ where $n$ is a positive integer is a linear transformation from $\mathbb{C}$ to itself and thus can be encoded by the inverse of a two by two matrix. This matrix, call it $A$, has eigenvalues $-n\pm i$, which have modulus greater than one. The matrix $A$ also has integer entries and has an inverse satisfying $\norm{A^{-1}x} = \norm{x}/\sqrt{n^{2}+1}$. 

In this thesis, we establish sufficient and necessary conditions for the self-similarity of the intersections $T \cap (T+\alpha)$ where the self-similar set $T$ is defined by functions of the form $A^{-1}(x+d)$, $d \in D\subset\mathbb{R}^{n}$, where $A$ is an $n$ by $n$ matrix with real entries and eigenvalues all lying outside the unit disc. This is done with conditions on the set $D$. We also extend a result of Davis and Hu in \cite{DH95}. 

Let $C(\alpha)$ denote the intersection of $C$ and its translate $C+\alpha$. It is shown in \cite{DH95} that the level sets of the function that maps a translation $\alpha$ to the Hausdorff dimension of $C(\alpha)$ are dense when the function is restricted to translations that produce nonempty intersections. In 2021, Pedersen and Shaw showed the same property holds for box-counting dimension when the self-similar sets are defined by functions of the form $f_{d}(x) = (-n+i)^{-1}(x+d)$ on the complex plane where $d\in D\subset\{0, 1, \ldots, \left\lfloor n^{2}/2 \right\rfloor\}$ satisfy $|d-d^{'}|\geq n+1$ for all distinct pairs $d, d^{'}\in D$ \cite{PS21}. In this thesis, we improve upon the latter result by incorporating information about when sequences of the form $(f_{d_{1}}\circ\ldots\circ f_{d_{k}}(0))_{k=1}^{\infty}$ have the same limit, where $f_{d}(x) = (-n+i)^{-1}(x+d)$. This allows us to strengthen the $n$-dependent condition $|d-d^{'}| > n$ to the $n$-independent condition $|d-d^{'}| > 3$ for $n\geq 5$ and extend the set $D$ to a subset of $\{0, 1, \ldots, n^{2}\}$. 

The majority of this thesis, extending from Chapter~\ref{chp:IFS} through Chapter~\ref{chp:neighbours}, concerns itself with these ideas. In Chapter~\ref{chp:discMultInv} and part of Chapter~\ref{chp:mirai}, we pivot to the study of multiplicatively invariant sets. These sets also have equal Hausdorff and box-counting dimension and overlap with the class of self-similar sets.  

There is an established theme of dynamically invariant sets having equal Hausdorff and box-counting dimension (see any of \cite{B96}, \cite{GP97}, \cite{P97}). The origin of this theme can be found in a pair of papers at the intersection of dynamical systems and fractal geometry written by Furstenberg (\cite{F67}, \cite{F70}). Therein, Furstenberg proved results and made conjectures about the fractal properties of multiplicatively invariant subsets of $\mathbb{T} := \faktor{\mathbb{R}}{\mathbb{Z}}$. Multiplicatively invariant subsets of $\mathbb{T}$ are the closed subsets that are invariant under the function $f_{r}(x) = rx$ where $r$ is a positive integer greater than or equal to $2$. For a specific value $r$, this is called $\times r$-invariance.

Furstenberg showed that the Hausdorff dimension and box-counting dimension of a $\times r$-invariant set are equal, this dimension is equal to the topological entropy of an underlying subset, and, perhaps most strikingly, that the only infinite subset of $\mathbb{T}$ that is simultaneously $\times 2$ and $\times 3$ invariant is $\mathbb{T}$ itself. This prompted Furtstenberg to conjecture that the only atomless probability measure on $\mathbb{T}$ that is invariant under both multiplication by $2$ and $3$ is the Lesbesgue measure. As of this writing, this is an open problem. 

In the time between then and now, there has been considerable development of the theory of multiplicatively invariant subsets of $\mathbb{T}$. Furstenberg's sumset conjecture, which offers sufficient conditions under which the Hausdorff and box-counting dimensions of sumsets of multiplicative invariant subsets split into the sum of the dimensions of those subsets, was proven by Hochman and Shmerkin in \cite{HS12}. An improvement was made independently by Shmerkin in \cite{S19} and Wu in \cite{W19} using different methods.

A discrete version was published in 2024 for subsets of the nonnegative integers \cite{GMR24} by Glasscock, Moreira, and Richter. This necessarily requires a reasonable definition of multiplicative invariance in the nonnegative integers. In Chapter~\ref{chp:discMultInv} we provide a generalization for subsets of $\mathbb{Z}^{n}$ and demonstrate the initial step of connecting the discrete sets to the comparatively well-studied subsets of $\mathbb{T}^{n}$ \cite{B83}. 

\section{Organization of the Thesis} 

\begin{itemize}
\item[(i)] Chapter~\ref{chp:litReview} reviews the relevant literature and explains the decision-making behind the work that inspired this thesis.
\item[(ii)] Chapter~\ref{chp:IFS} provides precise definitions of Hausdorff, box-counting dimension, attractors of iterated functions systems and their types, and the tools we use to study sets of the form $T\cap(T+\alpha)$
\item[(iii)]  Chapter~\ref{chp:highDimSEP} generalizes the ``self-similarity if and only if SEP" theme for intersections $T \cap (T + \alpha)$ where $T$ is a self-affine defined by functions of the form $f(x) = A^{-1}(x+d)$ where $A$ is an expanding matrix. 
\item[(iv)] Chapter~\ref{chp:limFormula} generalizes the functions that map translations $\alpha$ to the Hausdorff dimension and box-counting dimension of $T\cap(T+\alpha)$ to the case when the inverse of the defining matrix $A$ is a similarity. 
\item[(v)] Chapter~\ref{chp:neighbours} presents a framework for understanding when representations of $\alpha$, relative to the matrix $A$, are unique. The two dimensional set of expansions $\sum\limits_{j=1}^{\infty}(-n+i)^{-1}d_{j}$ where $n$ is a positive integer and $d_{j} \in \{0, 1, \ldots, n^{2}\}$ is used as a case study.  
\item[(vi)]  Chapter~\ref{chp:discMultInv} discusses a definition for multiplicative invariance in $\mathbb{Z}^{n}$. 
\item[(vii)] Chapter~\ref{chp:mirai} presents ideas for future work regarding the previous chapters. 
\end{itemize}

\chapter{Literature Review}\label{chp:litReview}

Many of the relevant papers that inspired this thesis were mentioned in the previous section. Although others are mentioned here, we focus on discussing the finer points of the literature that underpin the contents of this thesis. We stress the themes of the literature and not the precise definitions therein. The modern standard of rigorous mathematics is employed from Chapter~\ref{chp:IFS} through Chapter~\ref{chp:discMultInv}. 

The middle third Cantor set is the unique nonempty compact set satisfying $f_{0}(C) \cup f_{2}(C) = C$ where $f_{0} = x/3$ and $f_{2}(x) = (x+2)/3$. It can also be expressed as the collection of elements of the unit interval that have a ternary expansion that only use the digits 0 and 2. In set notation this is
\begin{equation}
\bigg\{\sum_{j=1}^{\infty}d_{j}3^{-j}: d_{j}\in\{0, 2\}\bigg\}.
\end{equation}

If a real number $\alpha$ has the property that $C\cap(C+\alpha)$ is nonempty, then it must be that $\alpha$ is an element of the $C-C := \{x-y: x, y\in C\}$. This means $\alpha$ has an expansion of the form $\sum\limits_{j=1}^{\infty}\alpha_{j}3^{-j}$ where $\alpha_{j}$ is either $0, 2$, or $-2$. The necessary and sufficient condition given in \cite{DHW08} on the translations $\alpha$ for which $C\cap(C+\alpha)$ is a property of the sequence $(2-|\alpha_{j}|)_{j=1}^{\infty}$. 

\begin{definition} 
A sequence $(a_{j})_{j=1}^{\infty}$ of integers is \textit{strongly eventually periodic} (SEP) if there exists a finite sequence $(b_{\ell})_{\ell = 1}^{p}$ and a nonnegative sequence $(c_{\ell})_{\ell = 1}^{p}$, where $p$ is a positive integer, such that
\begin{equation}
(a_{j})_{j\geq1} = (b_{\ell})\overline{(b_{\ell} + c_{\ell})_{\ell = 1}^{p}}, 
\end{equation}
where $\overline{(d_{\ell})_{\ell = 1}^{p}}$ denotes the infinite repetition of the finite sequence $(d_{\ell})_{\ell = 1}^{p}$.   
\end{definition}

We can now state the main theorem in \cite{DHW08}. 

\begin{theorem}
Suppose $C$ is the middle third Cantor set and $\alpha\in(0, 1)$ has the expansion $\sum\limits_{j=1}^{\infty}\alpha_{j}3^{-j}$ where $\alpha_{j}\in\{0, \pm2\}$ for each $j$. Suppose further that $\alpha$ is not an element of $E = \bigcup\limits_{k=1}^{\infty}\{\sum\limits_{j=1}^{k}\alpha_{j}3^{-j} \pm 3^{-(k+1)}:\alpha_{j}\in\{0, \pm 2\}\}$.
The set $C\cap(C+\alpha)$ is self-similar if and only if the sequence of integers $(2 - |\alpha_{j}|)_{j=1}^{\infty}$ is SEP. 
\end{theorem}

The set $E$ is the collection of elements in $C-C$ that have more than one base-$3$ expansion using digits $0, \pm 2$. In \cite{DHW08}, they argue that when $\alpha$ is an element of $E$, then $C\cap(C+\alpha)$ is not self-similar and none of the expansions yields SEP sequences. When $\alpha$ is in the complement of $E$, its expansion is unique. This property is crucial because it allows the intersection to be expressed in the form $\{\sum\limits_{j=1}^{\infty}d_{j}3^{-j}: d_{j} \in \{0, 2\} \cap (\{0, 2\} + \alpha_{j})\}$. We show in Chapter~\ref{chp:highDimSEP} that for sets of this form, the direction ``SEP implies self-similarity" follows from a simple argument. 

In this particular case, which we examine in Section~\ref{sec:sepSimSp}, the translation of $C\cap(C+\alpha)$ by its minimum value yields the set $\{\sum\limits_{j=1}^{\infty}d_{j}3^{-j}: d_{j}\in\{0, 2\}, d_{j} \leq 2 - |\alpha_{j}|\}$. This transformation of the set does not limit the generality of their argument because translates of self-similar sets are still self-similar. In addition to revealing a relationship between the intersection and the sequence $(2 - |\alpha_{j}|)_{j=1}^{\infty}$, the translated set contains zero. In particular, zero is its leftmost element. If the set is self-similar then using positivity it can be shown that a function of the form $f(x) = cx$ is an element of the defining IFS. This fact is used in the direction ``self-similarity implies SEP" in their proof and in many of the proofs of the generalizations of this theorem. 

We present results for the case when the self-similar set can be described as the set of expansions $\sum\limits_{j=1}^{\infty}A^{-j}d_{j}$ where $A$ is a matrix with eigenvalues all having modulus greater than one and $d_{j} \in D \subset \mathbb{R}^{n}$ for each $j$. When $n$ is greater than one, a compact subset of $\mathbb{R}^{n}$ does not have a canonically defined leftmost element. If the set of digits contains only two real elements, as it does for the middle third Cantor set, we can specify a ``leftmost" element by taking the vector whose $j$ digit is the smaller of the two in $\{d, e\} \cap (\{d, e\} + \alpha_{j})$. For more general digit sets, this is not an option. Instead we subtract whatever can yield us the existence of a map of the form $g(x) = B^{-1}x$ where $B$ is now a matrix. 

The use of more than two digits and contractions by a ratio other than a third in order to generate the self-similar sets has been studied in one dimension. Li, Yao, and Zhang treated the case when the digits formed an arithmetic progression \cite{LYZ11}. That is, a set of the form $\{t, t+m, t+2m, t+3m, \ldots, t+(k-1)m\}$ where $m$ and $k$ are positive integers and $t$ is constant. Further generalizations to digit sets that were not ``uniform" were given by Pedersen and Philips in \cite{PP14} and Kong in \cite{K14}. Both required a generalization of the SEP condition to sequences of sets. 

\begin{definition} 
A sequence $(A_{j})_{j=1}^{\infty}$ of nonempty finite subsets of $\mathbb{R}^{n}$ is called \textit{strongly eventually periodic} (SEP) if there exist two finite sequences of subsets of $\mathbb{R}^{n}$, $(B_{\ell})_{\ell = 1}^{p}$ and $(C_{\ell})_{\ell = 1}^{p}$, where $p$ is a positive integer, such that 
\begin{equation}
(A_{j})_{j=1}^{\infty} = (B_{\ell})_{\ell=1}^{p}\overline{(B_{\ell} + C_{\ell})_{\ell = 1}^{p}}, 
\end{equation}
where $B + C = \{b + c : b\in B, c\in C\}$ and $\overline{(D_{\ell})_{\ell = 1}^{p}}$ denotes the infinite repetition of the finite sequence of sets $(D_{\ell})_{\ell = 1}^{p}$. 
\end{definition}

There is a connection between the SEP property of sequences of sets and that of sequences of integers. Suppose $(D_{j})_{j=1}^{\infty}$ is a sequence of sets satisfying the SEP condition and that each $D_{j}$ is finite. Then the sequence of integers $(|D_{j}| - 1)_{j=1}^{\infty}$ is SEP. The converse does not hold in general. This is clear when we observe that the cardinality of the sets provides no information about their elements and thus any eventually periodic series of cardinalities could be realized by an aperiodic sequence of sets. Under the added assumption that each set is an arithmetic progression with a common step size for all sets, then the sequence of sets $(D_{j} - \gamma_{j})_{j=1}^{\infty}$ where $\gamma_{j} = \min D_{j}$ is SEP.

The proof of the first statement is the easier of the two. Let $(D_{j})_{j=1}^{\infty} = (A_{\ell})^{p}\overline{(A_{\ell} + B_{\ell})_{\ell = 1}^{p}}$. For $\ell = 1, 2, \ldots, p$, we have $|D_{\ell}| - 1 = |A_{\ell}| - 1$. Then, for those same $\ell$ and $k\geq1$, $|D_{\ell + kp}| = (|A_{\ell}| - 1) + (|A_{\ell} + B_{\ell}| - |A_{\ell}| + 1)$. 

Conversely, suppose $(|D_{j}| - 1) = (a_{\ell})_{\ell = 1}^{p}\overline{(a_{\ell} + b_{\ell})_{\ell = 1}^{p}}$. For all $j$, the $D_{j}$ is of the form $\{\gamma_{j}, \gamma_{j} + d, \gamma_{j} + 2d, \ldots, \gamma_{j} + m_{j}d\}$. When $\ell = 1, 2, \ldots, p$, the value of $m_{\ell}$ is $a_{\ell}$. When $j = \ell + kp$ where $\ell = 1, 2, \ldots, p$ and $k\geq1$, the value of $m_{\ell}$ is $a_{\ell} + b_{\ell}$. It follows that $D_{\ell + kp} - \gamma_{\ell + kp} = \{0, d, \ldots, a_{\ell}d\} + \{0, d, \ldots, b_{\ell}d\}$. The second set in this equation is not uniquely determined. For example, if $A_{\ell} = \{0, 1, 2, 3\}$ and $D_{\ell + kp} = \{0, 1, 2, 3, 4, 5\}$, we could choose $B_{\ell} = \{0, 1, 2\}$ or $\{0, 2\}$.

In both \cite{PP14} and \cite{K14}, the uniqueness of the expansion of the translate with coefficients in $D-D$ is also used to write the intersection $T\cap(T+\alpha)$ as a set of expansions $\sum\limits_{j=1}^{\infty}d_{j}m^{-j}$ where $d\in D\cap(D+\alpha_{j})$. We proceed in the same way. The SEP condition is a component-wise condition. If we cannot express the intersection as expansions where the digits are contained in subsets of $D$ then there is seemingly no foundation for the SEP condition. Although, we do discuss intersections using translates which boast multiple expansions in Section~\ref{sec:nonUniqueReps}.

Our interest in the case when expansions of the form $\sum\limits_{j=1}^{\infty}A^{-j}e_{j}$ where $e_{j}\in(D-D)$ are unique leads us to develop general principles to ensure uniqueness. We conduct an in-depth study of expansions which use $-n+i$ as a base as opposed to decimals or binary. This base came into popularity when computer scientists could simulate fractals and were experimenting with different ways of representing numbers on computers. The case when $n=1$ is featured in the second volume of Knuth's The Art of Computer Programming (\cite{K81}, pg.190). Katai and Szabo demonstrated the following facts about this base in \cite{KS75}. 

\begin{theorem}[I. Katai, J. Szabo, \cite{KS75}, theorem 1]  
Given a Gaussian integer $b$, every Gaussian integer $g$ can be uniquely written as
\begin{equation}\label{eq:integerRepChp2} 
g = \lambda_{0} + \lambda_{1}b + \ldots + \lambda_{k}b^{k}. 
\end{equation}
with $\lambda_{j} \in \{0, 1, 2, \ldots, |b|^{2}-1\}$ if and only if $\Re(b) < 0$ and $\Im(b) = \pm 1$. 
\end{theorem}

In other words, the base must be of the form $b = -n\pm i$ where $n$ is a positive integer and the set of digits is $\{0, 1, \ldots, n^{2}\}$.

\begin{corollary}[I. Katai, J. Szabo, \cite{KS75}, theorem 2] 
Suppose $n$ is a positive integer and set $b = -n+i$. Every complex number has an expansion of the form 
\begin{equation}\label{eq:compRep}
\zeta + \sum_{j=1}^{\infty}d_{j}b^{-j}
\end{equation}
where $d_{j}\in\{0, 1, \ldots, n^{2}\}$ for each $j$ and $\zeta\in\mathbb{Z}[i]$.
\end{corollary}

We owe much of our understanding of this number system to W. J. Gilbert. He published papers on when distinct representations capture the same complex number, provided an algorithm for finding the representations elements in $\mathbb{Q}[i]$, and computed the Hausdorff dimension of the boundary of the set containing all the points that have an expansion where $\zeta$ is the origin in (\ref{eq:compRep})  (\cite{G82}, \cite{G96}, \cite{G86}). That set, of the complex numbers of the form $\sum\limits_{j=1}^{\infty}d_{j}b^{-j}$ where $d_{j}\in\{0, 1, \ldots, n^{2}\}$, is an example of an integral digit tile. That is, the set is self-similar with respect to the collection of functions of the form $f_{d}(x) = A^{-1}(x + d)$ where $A$ is an integer matrix, $d$ is an element of $\mathbb{Z}^{2}$, the set has nonempty interior, the set of its translates by elements of $\mathbb{Z}^{n}$ covers $\mathbb{R}^{2}$, and the intersection of any finite number of those translates has $2$-dimensional Lebesgue measure zero.  We direct any interested readers to the survey \cite{AT04} by Akiyama and Thuswaldner which discusses the topology of the tiles, and the paper \cite{AL11} by Akiyama and Loridant describes a way of parameterizing their boundary. The collection of rules given in \cite{G82} are a specific instance of the graphs in \cite{ST03}. 

There is a more general tiling theory pioneered by Bandt \cite{B91}, Lagarias and Wang \cite{LW96}, and Strichartz and Wang \cite{SW99} that is related to our setting.  Many of the examples of self-similar sets in this thesis are either of this kind or are subsets of sets of this kind. On the whole, we do not make much use of this theory. We do rely on the algorithm in \cite{ST03} to complete one result and make a few references to theorems that apply to particular examples. It is also worth mentioning that the dimension of the boundary is determined in \cite{DKV00} by Duvall, Keesling, and Vince. Their methods rely on a subgraph of the graphs in \cite{ST03} and the ones we present in this thesis. 

Let us now discuss the functions that takes a translation $\alpha$ and maps it to a pre-specified fractal dimension of $C\cap(C+\alpha)$. The papers of David and Hu \cite{DH95} and Nekka and Li \cite{NL01} both prove the following theorem. 

\begin{theorem}\label{thm:DH}
Suppose $C$ is the middle third Cantor set. For any $\lambda\in[0, 1]$, the set $\{t:\dim_{H}(C\cap(C+t)) = \lambda\dim_{H}C\}$ is dense in $[0, 1]$. 
\end{theorem}

Here $\dim_{H}$ denotes the Hausdorff dimension. In \cite{PS21}, Pedersen and Shaw consider self-similar sets $T_{n, D}$ with respect to collections of functions of the form $f_{d}(z) = (-n+i)^{-1}(z + d)$ where $d\in D\subset\{0, 1, \ldots, n^{2}\}$ and $n$ is a positive integer. They prove a theorem concerning the box-counting dimension that is similar to Theorem~\ref{thm:DH}. Let $F_{n, D}$ denote the set of translates $\alpha$ such that $T(\alpha) = T_{n, D} \cap (T_{n, D} + \alpha)$ is nonempty. They define the function $\Phi_{n, D}$ on the subset of $F_{n, D}$ for which the box-dimension of $T(\alpha)$ exists and maps $\alpha$ to that quantity. The set $F_{n, D}$ is equal to $T_{n, D-D}$. In the case of the middle third Cantor set, this is $[-1, 1]$. 

\begin{theorem}[S. Pedersen, V. Shaw, \cite{PS21}, corollary 7.5] \label{thm:oldBound} 
Suppose that $n\geq3$ is an integer. If $D\subset\{0, 1, \ldots, \lfloor{n^{2}/2}\rfloor\}$ satisfies the condition $|e - e^{'}| > n$ for all $e \neq e^{'}$ in $D-D$, then the level sets of $\Phi_{n, D}$ are dense in $F_{n, D}$. 
\end{theorem}  

Conspicuously, the theorem requires that no difference of differences of $D$, that is an element of the form $(d_{1} - d_{2}) -(d_{3} - d_{4})$ where $d_{j}\in D$ for each $j$, has absolute value less than $n+1$. It also requires that $D$ contain no element larger than $n^{2}/2$. The separation and bounding condition in Theorem~\ref{thm:oldBound} together form a sufficient condition for the uniqueness of all expansions with respect to $-n+i$ which use coefficients in $D$. For any complex number $\alpha$ in $T_{n, D-D}$, this provides a way of expressing the intersection $T_{n, D} \cap (T_{n, D} + \alpha)$ as the set of expansions of the form $\sum\limits_{j=1}^{\infty}d_{j}(-n+i)^{-1}$ where $d_{j}\in D\cap(D+\alpha_{j})$ for each $j$. The sequence $(\alpha_{j})_{j=1}^{\infty}$ is the sequence of coefficients appearing, in order, in the base-$-n+i$ expansion of $\alpha$ using digits in $D-D$. The specific contributions of the two assumptions are as follows. 

The desired consequence of the bound $n^{2}/2$ on the elements of $D$ is that if
\begin{equation}
\sum_{j=1}^{\infty}(d_{j} - d_{j}^{'})(-n+i)^{-1} = \sum_{j=1}^{\infty}(d_{j}^{''} - d_{j}^{'''})(-n+i)^{-1}
\end{equation}
for elements $d_{j}, d_{j}, d_{j}, d_{j}\in D$ for each $j$, then the equivalence
\begin{equation}
\sum_{j=1}^{\infty}(d_{j} + d_{j}^{'''})(-n+i)^{-1} = \sum_{j=1}^{\infty}(d_{j}^{''} + d_{j}^{'})(-n+i)^{-1}
\end{equation}
is one of the expansions where the coefficients are in $\{0, 1, \ldots, n^{2}\}$. This is convenient because the condition that $|d - d^{'}| > n$ for all distinct $d, d^{'}\in D$ implies that those expansions are unique. Therefore $(d_{j} - d_{j}^{'}) = (d_{j}^{''} - d_{j}^{'''})$ for each $j$. 

We show in Chapter~\ref{chp:neighbours} that the expansions with coefficients in $D$ are unique when the translations of $T_{n, D}$ by the real elements of $D-D$ do not intersect $T_{n, D}$. The reason the condition $|d - d^{'}| > n$ is $n$-dependent has to do with the strategy by which $T_{n, D}$ was separated from its translates. The approach was to bound the set $T_{n}:= T_{n, \{0, 1, \ldots, n^{2}\}}$ in a ball and then discern the horizontal translations that ensure the balls do not intersect. 

Consider the approximations of $T_{n}$ when $n=2, 3, 4$ contained in Figures~\ref{fig:2t}, ~\ref{fig:3t}, and~\ref{fig:4t}.  

\begin{figure}
\centering
\includegraphics[scale=0.45]{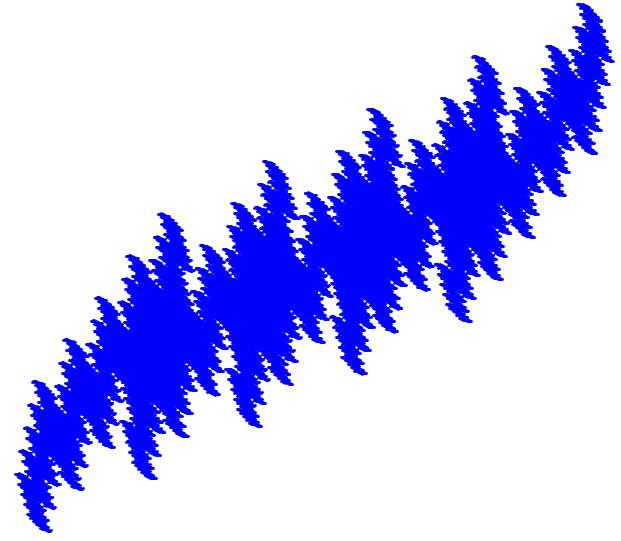}
\caption{The set $T_{2}$}
\label{fig:2t}
\end{figure}

\begin{figure}
\hspace{-3.0em}
\includegraphics[scale=0.45]{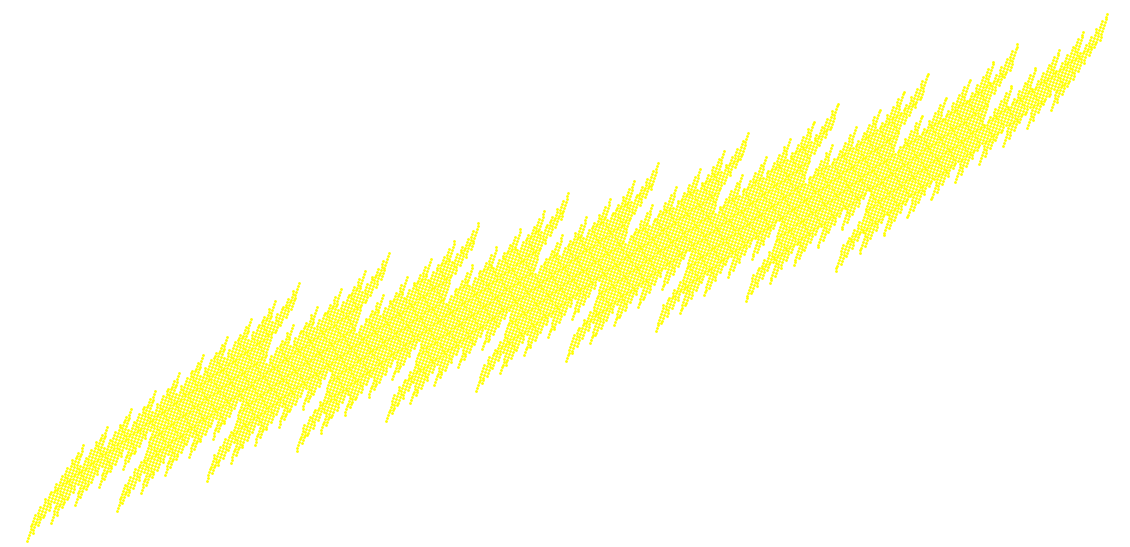}
\caption{The set $T_{3}$}
\label{fig:3t}
\end{figure}

\begin{figure}
\hspace{-3.0em}
\includegraphics[scale=0.45]{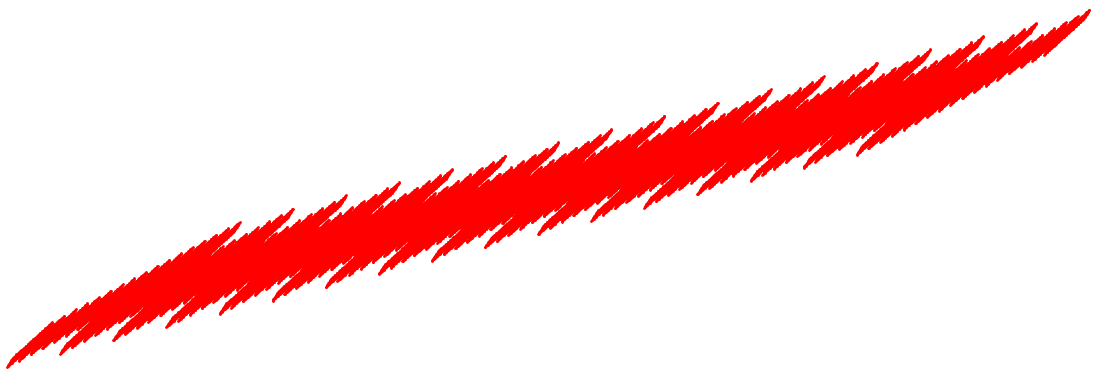}
\caption{The set $T_{4}$}
\label{fig:4t}
\end{figure}

In each of Figures~\ref{fig:2t}, ~\ref{fig:3t}, and~\ref{fig:4t}, the complement of $T_{n}$ within a ball containing it is large. In Chapter~\ref{chp:neighbours}, we lower the separation condition by refining our understanding of the boundary of $T_{n}$. For slightly larger separations, which are still $n$-independent, we even show that the bound of $n^{2}/2$ can be removed. 

Connected to this topic is a paper by Li and Xiao published in 1999 \cite{LX99}. In that paper, they investigate the Hausdorff dimension of the level sets of the function from a translate $\alpha$ to the Hausdorff dimension of $C\cap(C+\alpha)$ where $C$ is the middle third Cantor set \cite{LX99}. Twenty fours years later, Huang and Kong treated what might be considered the dual of this problem in \cite{HK23} concerning $(1-\lambda)$-Cantor sets. We present this dual version and our own discussion of the dimension of the level sets when our self-similar set is generated by affine transformations on $\mathbb{R}^{n}$ in Section~\ref{sec:levelSetDim}. 


Now we turn our attention to the literature that inspired Chapter~\ref{chp:discMultInv}. We begin with the precise definition of multiplicative invariance in $\mathbb{T} := \faktor{\mathbb{R}}{\mathbb{Z}}$.  

\begin{definition} \label{def:timesr} 
Let $r$ be an integer greater than one. 
We call a nonempty closed subset $Y \subset \mathbb{T}$ $\times r-$\textit{invariant} if $rY := \{ry: y\in Y\}$ is a subset $Y$. A subset $Y$ is called \textit{multiplicatively invariant} if it is $\times r$-invariant for some integer $r\geq2$.
\end{definition}

\remark{Furtstenberg showed in \cite{F67} (proposition III.1) that the box-counting dimension of a multiplicatively invariant set exists and is equal to its Hausdorff dimension.}

Glasscock, Moreira, and Richter provided the following version of multiplicative invariance for subsets of the nonnegative integers $\mathbb{N}_{0}$ \cite{GMR24}. Let us define two functions. 

Let $r$ be an integer greater than or equal to $2$. For an integer $n\geq1$, let $k := \lfloor \log_{r}n \rfloor$. This is the greatest integer power of $r$ which is less or equal to $n$. Define $\phi_{r}: \mathbb{N}_{0} \rightarrow  \mathbb{N}_{0}$ and $\psi_{r}: \mathbb{N}_{0} \rightarrow \mathbb{N}_{0}$ by
\begin{align}
&\phi_{r}(n) = \lfloor n/r \rfloor, \\
&\psi_{r}(n) =  n - r^{k}\lfloor n/r^{k} \rfloor.
\end{align}

\begin{definition}
A set $A \subset \mathbb{N}_{0}$ is called times-$r$ invariant if $\phi_{r}(A) \subset A$ and $\psi_{r}(A) \subset A$. We call $A\subset\mathbb{N}_{0}$ multiplicatively invariant if it is times-$r$ invariant for some $r\geq2$. 
\end{definition}

This use of two maps is admittedly curious. At first glance, $\phi_{r}$ is the natural analogue to multiplication by $r$ on $\mathbb{T}$. The latter can be thought of as multiplying an element of the unit interval by $r$ and taking the fractional part. The map $\phi_{r}$ divides by $r$ and takes the integer part. Meanwhile, what multiplication of an element of $x\in[0, 1)$ by $r$ does is map the base-$r$ expansion of $x$, $\sum\limits_{j=1}^{\infty}x_{j}r^{-j}$ where $r_{j}\in\{0, 1, \ldots, r-1\}$, to the expansion $x _{1} + \sum\limits_{j=2}^{\infty}x_{j}r^{-j}$. The latter is equal to $\sum\limits_{j=2}^{\infty}x_{j}r^{-j}$ as an element of $\mathbb{T}$. In this way, multiplication by $r$ mimics the left-shift operator on the sequence of coefficients defining the base-$r$ expansion. The map $\psi_{r}$ behaves similarly. To see this, we require a different perspective of the action of $\psi_{r}$. 

The images of the functions $\phi_{r}$ and $\psi_{r}$ are clearer if we replace $n$ by its base-$r$ representation. If $n = a_{k}r^{k} + \cdots a_{1}r + a_{0}$, $a_{k} \neq 0$, then 
\begin{align}
&\phi_{r}(n) = a_{k}r^{k-1} + \cdots + a_{2}r + a_{1}, \\
&\psi_{r}(n) = a_{k-1}r^{k-1} + \cdots a_{1}r + a_{0}. 
\end{align}

Multiplicative invariant subsets of $N_{0}$ can be linked to multiplicative invariant subsets of $\mathbb{T}$. If $E\subset\mathbb{N}_{0}$ is $\times r$ invariant, then the limiting set of the sequence $(r^{-k}(E\cap [0, r^{k})))_{k=1}^{\infty}$ is a subset of $[0, 1]$ which is $\times r$-invariant when viewed as a subset of $\mathbb{T}$ \cite{GMR24}. The elements of the sequence have the form $a_{k-1}r^{-1} + a^{k-2}r^{-2} +\cdots+a_{0}r^{-k}$ for each $k$. If we multiply by $r$ and take the fractional part, we obtain $a_{k-2}r^{-1} + a^{k-3}r^{-2} +\cdots+a_{0}r^{-k+1}$. This is the same result we would obtain if we applied $\psi_{r}$ to the expansion $a_{k-1}r^{k-1} + a^{k-2}r^{k-2} +\cdots+a_{0}$ and divided by $r^{k-1}$. In this way, $\psi_{r}$ is playing the role of the left-shift operator. Without means to define fractional or integer parts coherently, we define the invariance of a subset of $\mathbb{Z}^{n}$ using ``base-$B$" expansions where $B$ is a matrix: $d_{0} + Bd_{1} + \cdots + B^{k-1}d_{k-1}$. We discuss the existence of these expansions in Chapter~\ref{chp:discMultInv}. 

We do not treat Furstenberg's $\times 2, \times 3$ conjecture in this thesis. Interested readers can peruse \cite{T23} which maps out the development of near affirmative answers to the truth of the conjecture. Higher dimensional multiplicative invariance has been considered in \cite{B83}. Therein a higher dimensional version of Furstenberg's theorem that $\mathbb{T}$ is the only infinite subset of $\mathbb{T}$ that is simultaneously $\times 2$-invariant and $\times 3$-invariant is proven. We mention the precise result in Section~\ref{sec:multConj}.

\chapter{Fractal Geometry and Iterated Function Systems}\label{chp:IFS}

In this chapter we define the tools of fractal geometry that we use to study the structure of the intersections $T \cap (T+\alpha)$ where $T$ is a generalization of the middle third Cantor set. 

\section{Dimension Theory}\label{sec:dimTheory} 

Fractal dimensions of sets are, like many other dimensions in mathematics, descriptions of size. The Hausdorff dimension of a subset of a metric space is a refinement of the elementary idea that shapes be assigned dimensions one, two, or three based on whether or not they have finite length, area or volume. The box-counting
dimension of a bounded set is arguably the second most popular fractal dimension after Hausdorff dimension. Given a bounded set, it is roughly the rate at which the number of sets of diameter $\delta$ needed to cover the set increases as $\delta$ decreases. These descriptions might seem unrelated, but consider that we
can approximate the volume of a set by covering it with shapes we already know the volume of. We first describe the kind of coverings used to define Hausdorff dimension. 

\begin{definition} 
Let $\delta > 0$ and $V$ be a subset of a metric space $X$. A countable collection of sets $\{U_{k} \subset X\}$ is called a $\delta$-\textit{cover} of $V$ if 
\begin{enumerate}
\item[(i)] $V \subset \bigcup\limits_{k}U_{k}$,
\item[(ii)] $\diam{U_{k}} \leq \delta$ for each $k$. 
\end{enumerate}
Here and throughout this thesis, we use the notation $\diam{F}$ to denote the diameter of a subset $F$ of a metric space $(X, d)$. That is, $\diam{F} := \sup_{x, y\in F}d(x, y)$. 
\end{definition}
Now we define the continuum of measures used to assign dimensions (real numbers) to sets. 

\begin{definition} 
Let $V$ be a subset of a metric space and let $s > 0$. For every $\delta > 0$, define the quantity
\begin{equation}
\mathcal{H}_{\delta}^{s}(V) := \inf{\bigg\{\sum_{k}(\diam{U_{k}})^{s} : \{U_{k}\}\: \text{is a $\delta$-cover of}\: V \bigg\}}.
\end{equation}
The $s$-\textit{dimensional Hausdorff measure} of $V$ is the limiting value $\mathcal{H}^{s}(V) := \lim_{\delta\rightarrow 0^{+}}\mathcal{H}_{\delta}^{s}(V)$ (possibly positive infinity). 
\end{definition}

The discrete family of measures on Euclidean space that generalize the concept of everyday volume are the $n$-dimensional Lebesgue measures. When $n$ is a positive integer, there exists a positive constant such that the $n$-dimensional Hausdorff measure on the measurable sets of $\mathbb{R}^{n}$ is equal to the $n$-dimensional Lebesgue measure scaled by a constant (\cite{M95}, pg. 56). In this sense, the Hausdorff measures on Euclidean space are an extension of Lebesgue measures from the positive integers to the entirety of $[0,\infty)$. The $s$-dimensional Hausdorff measure is not a measure on the power set of the metric space (definition 1.18 in \cite{R87}). It is at best countably subadditive on that domain. The $s$-dimensional Hausdorff measure is a measure on the Borel sigma algebra associated with the metric (\cite{M95}, theorem 4.2). The following property of the Hausdorff measure is crucial to defining the Hausdorff dimension. 

Suppose that $V$ is a subset of a metric space, the pair of real numbers $s, t$ satisfy $0\leq s < t$, and $\delta > 0$. If $\{U_{k}\}_{k=1}^{\infty}$ is a $\delta$-cover of $V$, then 
\begin{equation}
\sum_{k=1}^{\infty}(\diam{U_{k}})^{t} = \sum_{k=1}^{\infty}(\diam{U_{k}})^{t-s}(\diam{U_{k}})^{s} \leq \delta^{t-s}\sum_{k=1}^{\infty}(\diam{U_{k}})^{s}. 
\end{equation}
It follows from the definition of the infimum that $\mathcal{H}_{\delta}^{t}(V) \leq \delta^{t-s}\mathcal{H}_{\delta}^{s}(V)$. If we take the limit as $\delta$ tends to zero from the right and we know that $\mathcal{H}^{t}(V)$ is nonzero, then it must be that $\mathcal{H}^{s}(V)$ is positive infinity. On the other hand, if $\mathcal{H}^{s}(V)$ is zero, then so is $\mathcal{H}^{t}(V)$. 

\begin{definition} 
Let $V$ be a subset of a metric space. We call the quantity
\begin{equation}
\dim_{H}{V} := \inf{\{s\geq0 : \mathcal{H}^{s}(V) = 0\}}
\end{equation}
the \textit{Hausdorff dimension} of $V$. 
\end{definition}

Since the $s$-dimensional Hausdorff measure of any set is positive infinity whenever $s$ is less than the Hausdorff dimension and zero after, it follows that $\dim_{H}V = \sup{\{s\geq0 : \mathcal{H}^{s}(V) > 0\}}$. 

\begin{example}
Let us determine the Hausdorff dimension of the open unit interval $J = (0, 1)$. It is sufficient to restrict the $\delta$ coverings used in the definition of the Hausdorff measures to $\delta$-covers of open balls \cite{F90}. For dimension one, the covering elements can be taken to be open intervals. We can cover $I$ using ($\lfloor 1/\delta \rfloor$ + 1) intervals of length $\delta$. Since $(\lfloor 1/\delta \rfloor + 1) \delta^{s}$ tends to zero as $\delta$ approaches zero for all $s > 1$, we see that $\mathcal{H}^{s}(J) = 0$ for all $s > 1$. On the other hand, suppose $s < 1$ and $(I_{k})_{k=1}^{\infty}$ is a  $\delta$-cover of $J$ comprised of intervals. For all sufficiently small $\delta$, we have $\sum\limits_{k=1}^{\infty} (\diam{I_{j}})^{s} \geq \sum\limits_{k=1}^{\infty}\text{length}(I_{k}) \geq \text{length}(J) = 1$. Therefore $\mathcal{H}^{1-\varepsilon}(J) > 0$ for every $\varepsilon$ in $(0, 1)$. We have already shown that $\dim_{H}J \leq 1$. Therefore $\dim_{H}J = \sup{\{s\in[0, 1] : \mathcal{H}^{s}(J) > 0\}} = 1$. 
\end{example}

Computing the Hausdorff dimension of a set without being able to make an educated guess is difficult when the only tool we have is the definition. In Section~\ref{sec:IFS}, we present a class of sets for which the task of computing their Hausdorff dimension is considerably easier. For this class, the Hausdorff dimension agrees with the following fractal dimension. 

\begin{definition}\label{def:boxcounting} 
Let $V$ be a bounded subset of a metric space. Given $\delta>0$, we let $N_{\delta}(V)$ denote the smallest number of sets of diameter $\delta$
needed to cover $V$. The \textit{upper box-counting dimension} and the \textit{lower box-counting dimension of $V$} are
\begin{align}
\overline{\dim}_{B}V &:= \limsup_{\delta\rightarrow0^{+}}\frac{\log(N_{\delta}(V))}{-\log(\delta)}, \\
\underline{\dim}_{B}V &:= \liminf_{\delta\rightarrow0^{+}}\frac{\log(N_{\delta}(V))}{-\log(\delta)}, 
\end{align}
respectively. If these quantities are equal, then that value is \textit{the box-counting dimension of $V$} and is denoted by $\dim_{B}V$. 
\end{definition} 

If the box-counting dimension of a set $V$ exists, then $N_{\delta}(V) \approx \delta^{-\dim_{B}}$. In other words, the box-counting dimension of $V$ is the exponential rate at which $N_{\delta}(V)$ increases as $\delta$ decreases. The reason it is called box-counting dimension is that, in $\mathbb{R}^{n}$, the counter $N_{\delta}(V)$ can be swapped with the number of $\delta$-mesh squares that intersect $V$ (\cite{F90}, equivalent definitions 3.1 ). The collection of $\delta$-mesh squares in $\mathbb{R}^{n}$ are the cubes $[k_{1}\delta, (k_{1} + 1)\delta] \times \cdots \times [k_{n}\delta, (k_{n} + 1)\delta]$ where $k_{i}$ is an integer for each $i$. This is a grid of $\delta$ by $\delta$ boxes in the plane when considered for $n=2$. 

The primary advantage of box-counting dimension is that it is easier to compute than Hausdorff dimension. The primary disadvantage is that it is not derived from a measure. The classification of the size of sets by dimension is coarser than that of classification by volume. The existence of an $s$-dimensional Hausdorff measure allows us to better grasp the size of the set. Additionally, the subadditivity of the Hausdorff measure is also a useful property that translates to the Hausdorff dimension. That is, if $\{V_{i}\}_{i=1}^{\infty}$ is a countable collection of subsets of a metric space, then $\dim_{H}(\cup_{i=1}^{\infty}V_{i}) = \sup_{i\geq1}\dim_{H}(V_{i})$. This is called countable stability. The Hausdorff dimension of a point in a metric space is zero because we can always cover the point by a single neighbourhood of diameter $\delta$. It follows that the Hausdorff dimension of a countable set is zero because we can decompose it into a countable union of single element sets (singletons). The box-counting dimension is not countably stable. For example, the box-counting dimension of the rational numbers contained in the unit interval is $1$ because we require approximately $1/\delta$ intervals of length $\delta$ to cover the set. 

This discrepancy is due to the fact that Hausdorff dimension treats covers that have diameter in $[0, \delta]$, while box-counting dimension restricts itself to covers with elements of strictly diameter $\delta$. Another notable example is that the box-counting dimension of the sequence $(1/k)_{k=1}^{\infty}$ is one half. This is undesirable because this sequence is not what most mathematicians and scientists mean when they speak of fractals. 

In more than one sense, Hausdorff dimension and box-counting dimension are opposites. The Hausdorff dimension is a well behaved mathematical object that is impractical to compute directly. The box-counting dimension fails to satisfy some ``nice" properties, but has a practical utility. In the next section, we present a class of sets for which the two notions of dimension agree.

\section{Iterated Function Systems}\label{sec:IFS}

In this section, we cover facts about iterated function systems which we refer back to throughout the later chapters. It is here that we provide a mathematical formulation of the concept of a shape being made up of smaller versions of itself. 
 
\begin{definition} 
Let $(X, d)$ be a metric space. We call a function $f: X \rightarrow X$ a \textit{contraction on $(X, d)$} if there exists $c\in[0, 1)$, called the \textit{contraction coefficient}, such that for every pair $x_{1}, x_{2}\in X$, 
\begin{equation} \label{eq:contractDef}
d(f(x_{1}), f(x_{2})) \leq cd(x_{1}, x_{2}).
\end{equation}
We call a contraction a \textit{similarity} if we always have equality in (\ref{eq:contractDef}). 
\end{definition}

\begin{example}
Let $c\in(-1, 1)$ and let $r$ be any real number. The simplest contractions on $\mathbb{R}$ are the functions $f(x) = cx +r$. More generally, if $v\in\mathbb{R}^{n}$, the affine transformation $g(x) = cx + v$ is a similarity on $\mathbb{R}^{n}$ with respect to Euclidean distance. 
\end{example}

\begin{definition} 
Let $(X, d)$ be a metric space. We call a finite set of contractions on $(X, d)$ an \textit{iterated function system defined on $(X, d)$}. 
\end{definition}

We use the abbreviation IFS for iterated function system and often omit mention of the metric space that a contraction or set of contractions are defined on when there is no ambiguity. The Banach fixed point theorem states that, given a contraction $f$ defined on a nonempty complete metric space $X$, there exists a unique element $x_{0}\in X$ such that $f(x_{0}) = x_{0}$ (\cite{DD02}, 11.1.6). Hutchinson showed that finite collections of contractions on a complete metric space have fixed points at the level of compact sets \cite{H81}. A precise description is given by the following theorem. 

\begin{theorem} \label{thm:hutchinson} [J.E. Hutchinson, \cite{H81}, theorem 1] 
Let $\mathcal{F}$ be an IFS defined on a complete metric space $X$. There exists a unique nonempty closed and bounded set $A\subset X$ such that 
\begin{equation}\label{eq:hutchOp}
\bigcup\limits_{f\in\mathcal{F}}f(A) = A.
\end{equation}
\end{theorem}

The idea behind its proof is that the set $A$ is the fixed point of the contraction $K \mapsto \bigcup\limits_{f\in\mathcal{F}}f(K)$ defined on the space of nonempty compact subsets of $X$ equipped with the Hausdorff distance. We make use of Hausdorff distance in Chapter~\ref{chp:discMultInv} (Definition~\ref{def:hausDist}).

\begin{definition} 
Let $\mathcal{F}$ be an IFS defined on a complete metric space $(X, d)$. We call the unique set satisfying (\ref{eq:hutchOp}) the attractor of $\mathcal{F}$. We call the attractor a \textit{self-similar set} or simply \textit{self-similar} if all the contractions in $\mathcal{F}$ are similarities. If $X \subset \mathbb{R}^{n}$ and $\mathcal{F}$ exclusively contains affine transformations, then we call the attractor a \textit{self-affine set} or simply \textit{self-affine}. Both $\mathcal{F}$ and its attractor are called \textit{homogeneous} if all the contraction coefficients of the elements of $\mathcal{F}$ are equal. 
\end{definition}

\begin{remark}
An IFS has a unique attractor, but it is not true that the IFS is unique to the attractor. Observe that for any IFS $\mathcal{F}$, we can define the IFS $\mathcal{G} = \{h \circ f : h, f\in\mathcal{F}\}$. For any $h\in\mathcal{F}$, it follows from (\ref{eq:hutchOp}) that $h(A) = \bigcup\limits_{f\in\mathcal{F}}h(f(A))$. Therefore $\bigcup\limits_{g\in\mathcal{G}}g(A) = \bigcup\limits_{h\in\mathcal{F}}(\bigcup\limits_{f\in\mathcal{F}}h(f(A))) = A$. 
\end{remark}

We present some examples. 

\begin{example}\label{ex:selfSimCantor}
The set of numbers of the form $\sum\limits_{j=1}^{\infty}d_{j}3^{-j}$ where $d_{j}\in\{0, 2\}$ is the attractor of the IFS $\{f_{0}(x) = x/3,  f_{2}(x) = (x+2)/3\}$. The attractor is the classical middle third Cantor set. 
\end{example} 

\begin{example}
Let $x_{0}$ be a fixed real number and consider the family of contractions $f_{c}: \mathbb{R}\rightarrow\mathbb{R}$ given by $f_{c}(x) = c(x-x_{0}) + x_{0}$ where $c\in(0, 1)$. The set $\{x_{0}\}$ is the attractor of $\{f_{c}\}$ for every $c\in(0, 1)$.  
\end{example}

Although subsets of a metric space containing only one element is self-similar, finite sets with more than one element are never self-similar. Suppose that $A$ is a finite set contained in a complete metric space $(X, d)$ and that $A$ is the attractor of a finite collection of similarities $\mathcal{F}$. Let $f$ be an element of $\mathcal{F}$ and suppose $A$ contains two distinct elements, $x$ and $y$. The function $f$ is injective and thus $d(f^{m}(x), f^{m}(y)) > 0$ for each $m$. The notation $f^{m}$ denotes the $m$-fold composition of $f$ with itself. By assumption $x$ and $y$ are elements of $A$. This means that $f^{m}(x)$ and $f^{m}(y)$ are element of $A$ for each $m$. In particular, because $A$ is finite (but has at least two elements), the distance between $f^{m}(x)$ and $f^{m}(y)$ is bounded below by $\eta = \min\limits_{u \neq v\in A}d(u, v) > 0$. This is a contradiction because there exists a sufficiently large $m_{0}$ such that $d(f^{m_{0}}(x), f^{m_{0}}(y)) \leq c^{m_{0}}d(x, y) < \eta$ where $c$ is the contraction coefficient of $f$. 

Self-similar sets are the ``nice" sets of fractal geometry. We are guaranteed that the box-counting dimension of a self-similar set exists and is equal to its Hausdorff dimension (\cite{F97}, corollary 3.3). There are self-similar sets that are particularly nice. These are the self-similar sets which are generated by an IFS that satisfies a separation condition. While there are a number of separation conditions that an IFS can satisfy, in this thesis we discuss arguably the two most popular: the strong separation condition (SSC) and the open set condition (OSC). 

\begin{definition}\label{def:strongSep} 
Suppose $A$ is the attractor of an IFS given by $\mathcal{F} = \{f_{i}\}_{i=1}^{N}$. We say that $\mathcal{F}$ satisfies the \textit{strong separation condition} (SSC) if the images $f_{i_{1}}(A)$ and $f_{i_{2}}(A)$ are disjoint for every distinct pair $1 \leq i_{1}, i_{2} \leq N$. 
\end{definition}

\begin{definition}\label{def:openSep} 
Suppose that $\mathcal{F} = \{f_{i}\}_{i=1}^{N}$ is an IFS. We say that $\mathcal{F}$ satisfies the \textit{open set condition} (OSC) if there exists an open set $\mathcal{O}$ such that images $f_{i_{1}}(\mathcal{O})$ and $f_{i_{2}}(\mathcal{O})$ are disjoint for every distinct pair $1 \leq i_{1}, i_{2} \leq N$ and $\bigcup\limits_{i=1}^{N}f_{i}(\mathcal{O}) \subset \mathcal{O}$. 
\end{definition}

The strong separation condition implies the open set condition because, for a sufficiently small $\varepsilon>0$, we could take the set of all points within $\varepsilon$ distance of an element of the attractor as an open set that satisfies the definition. The converse is not true. For example, the interval $[0, 1]$ is the attractor of the functions $\{g_{0}(x) = x/2, g_{1}(x) = (x+1)/2\}$. The union $g_{0}((0, 1)) \cup g_{1}((0, 1))$ is equal to the disjoint union $(0, 1/2) \cup (1/2, 1) \subset [0, 1]$, but the number $1/2$ is an element of $g_{0}([0, 1]) \cap g_{1}([0, 1])$. 

The reason these conditions are desirable is because they imply that the Hausdorff dimension of the self-similar set satisfies a simple equation involving the contraction coefficients of the functions in the IFS. 

\begin{definition} 
Let $\mathcal{F}$ be finite collection of similarities and let $c_{f}$ denote the contraction coefficient of $f\in\mathcal{F}$. We call the unique $s>0$ satisfying the equation 
\begin{equation}
\sum_{f\in\mathcal{F}}c_{f}^{s} = 1
\end{equation}
the \textit{similarity dimension}. 
\end{definition}

To see that this definition is well-defined, first observe that the function $h(t) = \sum\limits_{f\in\mathcal{F}}c_{f}^{t}$ is continuous on $[0, \infty)$. Since $h(0) = |\mathcal{F}| \geq 1$ and $\lim\limits_{t\rightarrow\infty}h(t) = 0$, there exists $s\in[0, \infty)$ such that $h(s) = 1$ by the intermediate value theorem. The solution is unique because $h(t)$ is strictly decreasing on $[0, \infty)$. 

\begin{theorem}[K. Falconer, \cite{F97}, corollary 3.3] \label{thm:SepSimDim} 
Suppose $A$ is the attractor of the set of similarities $\mathcal{F} = \{f_{i}\}_{i=1}^{N}$. For each $i$, let $r_{i}$ be the contraction ratio of $f_{i}$. If $\mathcal{F}$ satisfies the OSC, then $\dim_{B}A = \dim_{H}A = s$ where $s$ is the similarity dimension of $\mathcal{F}$. 
\end{theorem}

\begin{example} \label{ex:simDimCantor}
The similarity dimension of the IFS $\{f_{0}(x) = x/3, f_{2}(x) = (x+2)/3\}$ is the value $s$ such that $3^{-s} + 3^{-s} = 1$. This is $\log(2)/\log(3)$. Since this IFS satisfies the strong separation condition, this quantity is the Hausdorff dimension of its attractor, the middle third Cantor set. 
\end{example}

Another well-known self-similar set and fractal is the Sierpinski triangle that is depicted in Figure~\ref{fig:SierpinskiTriangle}. 

The Sierpinski triangle is the attractor of the IFS containing the functions $f_{d}: \mathbb{R}^{2} \rightarrow \mathbb{R}^{2}$ given by $f_{d}(x) = B(x+d)$ where $B = \begin{bmatrix}1/2 & 0 \\ 0 & 1/2\end{bmatrix}$ and $d\in\left\{\begin{bmatrix}0 \\ 0 \end{bmatrix},\begin{bmatrix}1/2 \\ 0 \end{bmatrix}, \begin{bmatrix}1/4 \\ \sqrt{3}/4 \end{bmatrix}\right\}$.\\

There is a natural way of relating the attractor to a symbolic space. This is accomplished by viewing the IFS $\mathcal{F}$ as a finite alphabet. 

\begin{definition} 
Let $\mathcal{F} = \{f_{1}, f_{2}, \ldots, f_{n}\}$ be an IFS defined on a complete metric space $X$. Let $A$ denote the attractor of $\mathcal{F}$ and $\Sigma := \{1, 2,\ldots, n\}^{\mathbb{N}}$. The map $\pi_{\mathcal{F}}:\Sigma\to A$ given by $\pi_{\mathcal{F}}(i_{j})_{j=1}^{\infty} = y$
where $y$ is the unique element of 
\begin{equation}
\bigcap_{k=1}^{\infty}(f_{i_{1}}\circ f_{i_{2}}\circ \cdots \circ f_{i_{k}})(A)
\end{equation}
is called the \textit{coding map} and $\Sigma$ is referred to as the \textit{coding space}.
\end{definition}

\begin{remark}
In \cite{H81}, it is demonstrated that the attractor is compact. The images in the intersection are then compact because the compositions are continuous functions. The unique element $y$ then exists by an application of Cantor's Intersection Theorem (see \cite{DD02}). 
\end{remark}

In certain metric spaces, such as Euclidean space, we can understand the form of the elements of an attractor using the following reformulation of the coding map. 

\begin{lemma}\label{lem:codingMapForm} 
Suppose $\mathcal{F} = \{f_{1}, f_{2}, \ldots, f_{n}\}$ is an IFS on a metric space $X$. For every $x\in X$, 
\begin{equation}
\pi_{\mathcal{F}}(i_{j})_{j=1}^{\infty} = \lim_{k\rightarrow\infty} (f_{i_{1}}\circ f_{i_{2}} \circ \cdots\circ f_{i_{k}})(x)
\end{equation}
for any $(i_{j})_{j=1}^{\infty}\in\Sigma$. 
\end{lemma}

\begin{proof} 
 Fix $x\in X$ and let $y$ denote $\pi_{\mathcal{F}}(i_{j})_{j=1}^{\infty}$. Observe that the quantity $R := \sup\limits_{a\in A}d(a, x)$ is finite because $A$ is bounded. 

Let $\varepsilon > 0$. There exists $K$ such that  $\prod\limits_{j=1}^{K} c_{i_{j}} < \varepsilon/R$. For all $k$, $y = (f_{i_{1}}\circ f_{i_{2}} \circ \cdots\circ f_{i_{k}})(x_{k})$ for some $x_{k}\in A$. If $k\geq K$, then
\begin{equation}
d((f_{i_{1}}\circ f_{i_{2}} \circ \cdots\circ f_{i_{k}})(x), y) \leq \bigg(\prod\limits_{j=1}^{K}c_{i_{j}}\bigg) d(x_{k}, y) \leq  R\bigg(\prod\limits_{j=1}^{K}c_{i_{j}}\bigg) < \varepsilon. 
\end{equation}
\end{proof}

We end with a convenient fact about translations of attractors that we use frequently in Chapter~\ref{chp:highDimSEP} and Chapter~\ref{chp:limFormula}. Let $M_{n}(\mathbb{R})$ denote the set of $n$ by $n$ matrices with real entries. 

\begin{lemma}\label{lem:attractorShift} 
Let $u$ be an element of $\mathbb{R}^{n}$. If $S\subset\mathbb{R}^{n}$ is the attractor of the IFS given by $\{f_{i}(x) = r_{i}x + v_{i}\}_{i=1}^{N}$ where $r_{i} \in M_{n}(\mathbb{R})$ and $v_{i}\in\mathbb{R}^{n}$, then the translation $S + u$ is the attractor of the IFS $\{g_{i}(x) = f_{i}(x) + (1 - r_{i})u\}_{i=1}^{N}$. 
\end{lemma}

\begin{proof}
Firstly $S + u$ is a nonempty compact set because $S$ is a nonempty compact set. Since attractors are unique we need only observe that $\bigcup\limits_{i=1}^{N}g_{i}(S + u) = \bigcup\limits_{i=1}^{N}(f_{i}(S) + u) = \bigg(\bigcup\limits_{i=1}^{N}f_{i}(S)\bigg) + u = S + u$. 
\end{proof}

\chapter{Self-Affinity and Strong Eventual Periodicity}\label{chp:highDimSEP}

In this chapter, we begin our examination of the intersection $T\cap(T+\alpha)$ when $T\subset\mathbb{R}^{n}$ is an attractor of an IFS and $\alpha$ is an element of $\mathbb{R}^{n}$. We specifically consider IFSs that contain affine transformations that share a common linear factor. This is the class that has received the most study on the real line and so it is natural for it to be the first type to extend to $\mathbb{R}^{n}$. Let us introduce the class of matrices that serve as the linear factor of the affine transformations. We use the notation $M_{n}(S)$ to denote the set of $n$ by $n$ matrices with entries in $S$. 

\section{Intersections Involving Self-Affine Sets}

\begin{definition} 
Let $A$ be an element of $M_{n}(\mathbb{R})$. The matrix $A$ is called an \textit{expanding matrix} if all its eigenvalues have modulus greater than one. 
\end{definition}

An expanding matrix $A$ is necessarily invertible since it does not have an eigenvalue equal to zero. With respect to a choice of norm on $\mathbb{R}^{n}$, multiplication by $A^{-1}$ is a contraction. In particular, given a fixed $d\in\mathbb{R}^{n}$, the function $f(x) = A^{-1}(x+d)$ is a contraction. It follows from Theorem~\ref{thm:hutchinson} that finite collections of affine maps of this kind generate an attractor. We provide the construction of the norm from \cite{LW96b} with the desired property. 

\begin{lemma}\label{lem:contractingNorm}
Suppose that $A\in M_{n}(\mathbb{R})$ is an expanding matrix and let $r = \min_{\lambda\in\sigma(A)}|\lambda|$. For any $\rho\in(1, r)$, the function $f:\mathbb{R}^{n} \rightarrow \mathbb{R}^{n}$ given by $f(x) = A^{-1}x$ is a contraction with respect to the norm $\norm{x}_{A} := \sum\limits_{j=1}^{\infty} \rho^{j}\norm{A^{-j}x}$ where $\norm{\cdot}$ is the Euclidean norm. 
\end{lemma}

\begin{proof}
First we argue that $||x||_{A}$ converges for all $x\in\mathbb{R}^{n}$. Let $\norm{\cdot}_{\text{op}}$ be the operator norm on the space of linear operators on $\mathbb{R}^{n}$ with respect to the Euclidean norm. For any $A\in M_{n}(\mathbb{R})$, $\norm{A}_{op} = \sup_{\norm{u}\leq1}\norm{Au}$. By the spectral radius formula, the limit of $\norm{A^{-k}}_{\text{op}}^{1/k}$ as $k$ tends to infinity is $1/r$ (\cite{Z93}, theorem 5.5). Since $1/r < 1/\rho$, there exists $\delta \in (0, 1)$ such that $1/r < \delta/\rho$. Using the definition of the existence of the limit, there exists $K$ such that, $\norm{A^{-k}} < \delta^{k}/\rho^{k}$ for all $k \geq K$. It follows that
\begin{equation}
\sum_{k=K}^{\infty}\rho^{k}\norm{A^{-k}x} \leq \sum_{k=K}^{\infty}\rho^{k}\norm{A^{-k}}_{\text{op}}\norm{x} < \sum_{k=K}^{\infty}\delta^{k} < \infty. 
\end{equation}
since the series of powers of $\delta$ is a convergent geometric series. The function $\norm{x}_{A}$ inherits the properties of a norm from the Euclidean norm. 
To see that $f(x) = A^{-1}x$ is a contraction on the normed space $(\mathbb{R}^{n}, \norm{\cdot}_{A})$, observe that
$\norm{A^{-1}x}_{A} = \sum\limits_{j=1}^{\infty} \rho^{j}\norm{A^{-j}(A^{-1}x)} = \sum\limits_{j=2}^{\infty}\rho^{j-1}\norm{A^{-j}x} \leq \rho^{-1}\norm{x}_{A}$. Since $\rho$ was assumed to be greater than one, we have that $\rho^{-1}$ is in $(0, 1)$ and is a valid contraction coefficient. 
\end{proof}

\begin{remark}
Not all expansive matrices on $\mathbb{R}^{n}$ are contractions with respect to the Euclidean norm. Consider the matrix $A = \begin{bmatrix}
\sqrt{2} & 1 \\
0 & \sqrt{2} \\
\end{bmatrix}.$
Given a matrix $B\in M_{n}(\mathbb{R})$, let $\rho(B)$ denote the spectral radius of $B$. We can compute $\norm{A^{-1}}_{\text{op}}$ using the formula $\sqrt{\rho(A^{-t}A^{-1})}$. This yields the value $1$. Therefore, for any $\tau\in(0, 1)$, there exists an element $x$ of the closed unit ball such that $\norm{A^{-1}x} > \tau\norm{x}$. Despite this, since all norms on $\mathbb{R}^{n}$ are equivalent, all sufficiently large powers of an expansive matrix are contractions with respect to the Euclidean norm. 
\end{remark}

\begin{definition} 
Let $A\in M_{n}(\mathbb{R})$ be an expanding matrix and let $D$ be a finite subset of $\mathbb{R}^{n}$. We call the attractor of the iterated function system containing precisely the maps $\phi_{d}(x) = A^{-1}(x + d)$ where $d\in D$ the \textit{self-affine set generated by $(A, D)$}. We denote the set by $T_{A, D}$. 
\end{definition}

\begin{definition} 
Let $A\in M_{n}(\mathbb{R})$ be an expanding matrix and let $D$ be a finite subset of $\mathbb{R}^{n}$. We call the coding map associated with IFS $\{f_{d}(x) = A^{-1}(x+d):d\in D\}$ the \textit{$(A, D)$-coding map} and we denote it be $\pi_{A, D}$. 
\end{definition}

\begin{definition} 
Let $A\in M_{n}(\mathbb{R})$ be an expanding matrix and let $D$ be a finite subset of $\mathbb{R}^{n}$. Let $x$ be an element of $T_{A, D}$.  We call any sequence $(x_{j})_{j=1}^{\infty}\in D^{\mathbb{N}}$ satisfying $\pi_{A, D}(x_{j})_{j=1}^{\infty} = x$ an \textit{$(A, D)$-representation of $x$}. We say that a pair of $(A, D)$-representations, $(x_{j})_{j=1}^{\infty}$ and $(y_{j})_{j=1}^{\infty}$, are \textit{equivalent} if $\pi_{A, D}(x_{j})_{j=1}^{\infty} = \pi_{A, D}(y_{j})_{j=1}^{\infty}$. An $(A, D)$-representation is called \textit{unique} it is only equivalent to itself. 
\end{definition}

\begin{lemma}\label{lem:selfSimForm} 
Let $A\in M_{n}(\mathbb{R})$ be an expanding matrix and let $D$ be a finite subset of $\mathbb{R}^{n}$. The self-affine set generated by $(A, D)$ is equal to $$\left\{\sum_{j=1}^{\infty}A^{-j}d_{j} : d_{j}\in D\right\}.\nonumber$$ 

In particular, if $(x_{j})_{j=1}^{\infty}$ is a sequence of elements of $D$, then $\pi_{A, D}(x_{j})_{j=1}^{\infty} = \sum\limits_{j=1}^{\infty}A^{-j}x_{j}$.
\end{lemma}

\begin{proof}
If we label the affine transformations $f_{d}(x) = A^{-1}(x+d)$ that generate $T_{A, D}$ by their parameter $d$, then the coding map is defined on the sequences contained in $D^{\mathbb{N}}$. Let $(x_{j})_{j=1}^{\infty}$ be a sequence of elements of $D$. By Lemma~\ref{lem:codingMapForm},
\begin{equation}
\pi_{A, D}(x_{j})_{j=1}^{\infty} = \lim_{k\rightarrow\infty}(f_{x_{1}} \circ \cdots \circ f_{x_{k}})(0) = \lim_{k\rightarrow\infty} \sum_{j=1}^{k}A^{-j}x_{j}.
\end{equation}
It follows that $T_{A, D} = \pi_{A, D}(D^{N})$. 
\end{proof}

For any subset $S \subset\mathbb{R}^{n}$ and $t\in\mathbb{R}^{n}$, we define the operation $S+t := \{s + t: s\in S\}$. Given an expanding matrix $A\in M_{n}(\mathbb{R})$ and finite set $D\subset\mathbb{Z}^{n}$, we want to determine when the intersections of the form $T_{A, D} \cap (T_{A, D} + \alpha)$ are themselves self-affine for $\alpha\in\mathbb{R}^{n}$. This is only interesting for those $\alpha\in\mathbb{R}^{n}$ which yield nonempty intersections. This occurs if and only if $\alpha$ has an $(A, D-D)$-representation $(\alpha_{j})_{j=1}^{\infty}$ where $D-D := \{d-d^{'}:d, d^{'}\in D\}$. Under certain conditions, it is a property of the sequence of sets $(D \cap (D + \alpha_{j}))_{j=1}^{\infty}$ that determines when the intersection $T_{A, D} \cap (T_{A, D} + \alpha)$ is self-affine.

\begin{definition} 
A sequence $(A_{j})_{j=1}^{\infty}$ of nonempty finite subsets of $\mathbb{R}^{n}$ is called \textit{strongly eventually periodic} (SEP) if there exist two finite sequences of subsets of $\mathbb{R}^{n}$, $(B_{\ell})_{\ell = 1}^{p}$ and $(C_{\ell})_{\ell = 1}^{p}$, where $p$ is a positive integer, such that 
\begin{equation}
(A_{j})_{j=1}^{\infty} = (B_{\ell})_{\ell=1}^{p}\overline{(B_{\ell} + C_{\ell})_{\ell = 1}^{p}}, 
\end{equation}
where $B + C = \{b + c : b\in B, c\in C\}$ and $\overline{(D_{\ell})_{\ell = 1}^{p}}$ denotes the infinite repetition of the finite sequence of sets $(D_{\ell})_{\ell = 1}^{p}$. 
\end{definition}


\begin{example}
Let $(A_{j})_{j=1}^{\infty}$ be the sequence $$ (\{0, 4\}, \{0\}, \overline{\{0, 4, 8\}, \{A\}}).$$ Here $p=2$, $(B_{\ell})_{\ell=1}^{2} = (\{0, 4\}, \{0\})$, and $(C_{\ell})_{\ell=1}^{2} = (\{0, 4\}, \{8\})$. If $D = \{0, 4, 8\}$ and $(\alpha_{j})_{j=1}^{\infty} = (-4, -8, \overline{0, 8})$, then $_{j} = D\cap (D + \alpha_{j})$. This is relevant in the context of Theorem~\ref{thm:attractSEP}. 
\end{example}

The following theorem captures a necessary and sufficient condition for when $T_{A, D} \cap (T_{A, D} + \alpha)$ is self-affine. 

\begin{theorem}\label{thm:attractSEP} 
Suppose an expanding matrix $A\in M_{n}(\mathbb{R})$ and finite set $D\subset\mathbb{R}^{n}$ are chosen such that $(A, D-D)$ representations are unique. Suppose that $\alpha$ has the $(A, D-D)$-representation $(\alpha_{j})_{j=1}^{\infty}$. The set $T_{A, D}\cap(T_{A, D} + \alpha)$ is the attractor of an IFS of the form $\{A^{-p}x + r_{i}\}_{i=1}^{N}$, where $r_{i}$ is an element of $\mathbb{R}^{n}$ for $i=1, 2,\ldots, N$ and $p$ is a positive integer, if and only if the sequence $((D\cap (D+\alpha_{j}) - \beta_{j})_{j=1}^{\infty}$ is SEP for some sequence $(\beta_{j})_{j=1}^{\infty} \in D$. 

In particular, if $((D\cap (D + \alpha_{j})) - \beta_{j})_{j=1}^{\infty} = (U_{\ell})_{\ell=1}^{p}\overline{(U_{\ell} + V_{\ell})_{\ell=1}^{p}}$, then the set $T_{A, D} \cap (T_{A, D} + \alpha)$ is the attractor of the IFS precisely containing the maps
\begin{equation*}
f(x) = A^{-p}\bigg(x + \sum_{\ell=1}^{p}(A^{p-\ell}u_{\ell} +A^{-\ell} v_{\ell})-\beta\bigg)+\beta
\end{equation*}
where $u_{\ell} \in U_{\ell}$ and $v_{\ell} \in V_{\ell}$ for each $\ell$, and $\beta = \pi_{A, D}(\beta_{j})_{j=1}^{\infty}$. 
\end{theorem}

Our goal is to prove Theorem~\ref{thm:attractSEP}. Before we proceed, we admit that the condition that \emph{all} $(A, (D-D))$ representations be unique is strong. We treat a special case (Theorem~\ref{thm:selfSimSEPSp}) for which the sufficient and necessary condition can be reached using the weaker assumption that the translation $\alpha$ has a unique $(A, D-D)$-representation. The treatment of that special case also proceeds without assuming that the affine maps in the IFS feature the same linear component $A^{-p}$. 

The following result shows that the intersections $T_{A, D} \cap (T_{A, D} + \alpha)$ can be expressed as the image of a product of subsets of $D$ under $\pi_{A, D}$ when $\alpha$ has a unique $(A, D-D)$-representation. 

\begin{lemma} \label{lem:seqExpress} 
Suppose $A\in M_{n}(\mathbb{R})$ is an expanding matrix and $D\subset\mathbb{R}^{n}$ is finite. If $\alpha\in\mathbb{R}^{n}$ has a unique $(A, D-D)$-representation $(\alpha_{j})_{j=1}^{\infty}$, then $T_{A, D} \cap (T_{A, D} + \alpha) = \pi_{A, D}\left(\prod\limits_{j=1}^{\infty}D\cap(D+\alpha_{j})\right)$. 
\end{lemma}

\begin{proof}
Let $\pi = \pi_{A, D}$. Assume that $\alpha$ has a unique $(A, D-D)$-representation $(\alpha_{j})_{j=1}^{\infty}$. Let $T(\alpha)$ denote $T_{A, D} \cap (T_{A, D} + \alpha)$. If $x\in T(\alpha)$, then $x \in T_{A, D}$ and $x \in T_{A, D} + \alpha$. On one hand, $x = \pi(x_{j})_{j=1}^{\infty}$ with $x_{j}\in D$ for all $j$. On the otherhand, there exists $y \in T_{A, D}$ such that $x = y + \alpha$. Similarly, $y = \pi(y_{j})_{j=1}^{\infty}$ where $y_{j}\in D$ for all $j$. We conclude that $\alpha = \pi(x_{j} - y_{j})_{j=1}^{\infty}$ where $x_{j} - y_{j} \in (D-D)$ for all $j$. Since the $(A, D-D)$-representation of $\alpha$ is unique, we conclude that $x_{j} - y_{j} = \alpha_{j}$ and, in particular, $x_{j} = y_{j} + \alpha_{j}$ for all $j$. This means that $x_{j}\in D\cap(D+\alpha_{j})$ for all $j$. We leave the inclusion $\pi(D\cap(D + \alpha_{j})) \subset T(\alpha)$ for the reader.
\end{proof}

By Lemma~\ref{lem:seqExpress}, the hypotheses of Theorem~\ref{thm:attractSEP} imply that $T\cap(T+\alpha)$ can be expressed as $\pi_{A, D}\left(\prod\limits_{j=1}^{\infty}(D\cap(D+\alpha_{j}))\right)$ where $\alpha_{j}$ is an $(A, D)$-representation of $\alpha$. The underlying set of sequences has the form $\prod\limits_{j=1}^{\infty}D_{j}$ where $D_{j}\subset \mathbb{R}^{n}$ for each $j$. We proceed by arguing that if $(D_{j} - \beta_{j})_{j=1}^{\infty}$ is SEP for some sequence $(\beta_{j})_{j=1}^{\infty}\in D^{\mathbb{N}}$, then the set $\pi_{A, D}\left(\prod\limits_{j=1}^{\infty}D_{j}\right)$ is self-affine. We note that a special case of these sets contained in the unit interval, called generalized Cantor sets, is treated in \cite{K14}. 

\begin{lemma}\label{lem:sepGen} 
Suppose that $A\in M_{n}(\mathbb{R})$ is an expanding matrix and $D\subset\mathbb{R}^{n}$ is finite. Suppose $(D_{j})_{j=1}^{\infty}$ is a sequence of subsets of $D$ and $(\beta_{j})_{j=1}^{\infty}$ is a sequence of elements of $D$.  If $(D_{j}-\beta_{j})_{j=1}^{\infty}$ is SEP, then the set $\pi_{A, D}\left(\prod\limits_{j=1}^{\infty}D_{j}\right)$ is self-affine. 

In particular, if $(D_{j})_{j=1}^{\infty} = (U_{\ell})_{\ell=1}^{p}\overline{(U_{\ell} + V_{\ell})_{\ell=1}^{p}}$, then $\pi_{A, D}\left(\prod\limits_{j=1}^{\infty}D_{j}\right)$ is the attractor of the IFS precisely containing the maps
\begin{equation*}
f(x) = A^{-p}\bigg(x - \beta + \sum_{\ell=1}^{p}(u_{\ell}A^{p-\ell} + v_{\ell}A^{-\ell})\bigg) + \beta
\end{equation*}
where $u_{\ell} \in U_{\ell}$ and $v_{\ell} \in V_{\ell}$ for each $\ell$, and where $\beta = \pi_{A, D}(\beta_{j})_{j=1}^{\infty}$. 
\end{lemma}

The proof of this result is a generalization of the one-dimensional version in \cite{LYZ11}. 

\begin{proof}
For convenience, let $X$ denote $\pi_{A, D}\left(\prod\limits_{j=1}^{\infty}D_{j}\right)$. Assume that there exist sets $U_{1}, U_{2}, \ldots, U_{p}$ and $V_{1}, V_{2}, \ldots, V_{p}$ such that $(D_{j})_{j=1}^{\infty} = U_{1}\ldots U_{p}\overline{(U_{\ell}+V_{\ell})_{\ell = 1}^{p}}$. 

If $x$ is an element of $X-\beta$, then $x = \sum\limits_{j=1}^{\infty}A^{-j}x_{j}$ with $x_{j}\in D_{j} - \beta_{j}$. By assumption, we can write
\begin{equation}
x = \sum_{j=1}^{p}A^{-j}u_{0, j} + \sum_{k=1}^{\infty}\bigg(\sum_{j=1}^{p}A^{-(kp + j)}(u_{k, j} + v_{k, j})\bigg)
\end{equation}
with $u_{k, j}\in U_{j}$ and $v_{k, j}\in V_{j}$ for each $j$. We rearrange the terms to obtain 
\begin{equation}\label{eq:selfSimForm}
x =\sum\limits_{k=0}^{\infty}A^{-(k+1)p}\bigg(\sum_{j=1}^{p}(A^{p-j}u_{k, j}+ A^{-j}v_{k+1, j})\bigg). 
\end{equation}

By Lemma~\ref{lem:selfSimForm}, $x$ is an element of the attractor of the IFS $\{h(x) = A^{-p}(x + \sum\limits_{\ell = 1}^{p}(A^{p-\ell}u_{\ell}+A^{-\ell}v_{\ell}) : u_{\ell}\in U_{\ell}, v_{\ell}\in V_{\ell}\}$ if and only if it can be expressed in the form given in (\ref{eq:selfSimForm}). To conclude, apply Lemma~\ref{lem:attractorShift} to the IFS that generates $X-\beta$ to obtain the form of the IFS for $X$.
\end{proof}

It is not true that all self-affine sets of the form $\pi_{A, D}(\prod\limits_{j=1}^{\infty}D_{j})$ come from SEP sequences. Consider the following example. 

\begin{example}\label{ex:nonSEP} 
We first give an example on $\mathbb{R}$ and then show that the main idea extends to $\mathbb{R}^{n}$. Let $A = 9$ and $D = \{0, 1, \ldots, 8\}$. Let the sequence of sets $(D_{j})_{j=1}^{\infty}$ be the sequence $(\{0\}, \{0, 6\}, \overline{\{0, 1\}, \{0, 3, 6\}})$. There is no sequence $(\beta_{j})_{j=1}^{\infty}\in D^{\mathbb{N}}$ for which $(D_{j} -\beta_{j})$ is SEP because there is no subset $S \subset \mathbb{R}$ for which $\{-\beta_{2}, 6-\beta_{2}\} + S = \{-\beta_{4}, 3-\beta_{4}, 6-\beta_{4}\}$. If such an $S$ existed, it would have to contain either $3-\beta_{4}+\beta_{2}$ or $-3-\beta_{4}+\beta_{2}$ to ensure that $3-\beta_{4}$ is an element of the sumset. The inclusion of $3-\beta_{4}+\beta_{2}$ or $-3-\beta_{4}+\beta_{2}$ implies that $9-\beta_{4}$ or $-3-\beta_{4}$, respectively, are elements of the sumset. 

Nonetheless, $X := \pi_{9, D}(\{0\} \times \{0, 6\} \times \overline{\{0, 1\} \times \{0, 3, 6\}})$ is the attractor of the IFS containing precisely the following eight maps: 
\begin{equation}
 \begin{split}
    &f_{1}(x) = 9^{-2}x, \\
    &f_{2}(x) = 9^{-2}x + 3/9^{4}, \\
    &f_{3}(x) = 9^{-2}x + 6/9^{4},\\
    &f_{4}(x) = 9^{-2}x + 1/9^{3},
  \end{split}
\quad
  \begin{split}
    &f_{5}(x) = f_{1}(x) + 6/9^{2}, \\
    &f_{6}(x) = f_{2}(x) + 6/9^{2},\\
    &f_{7}(x) = f_{3}(x) + 6/9^{2},\\
    &f_{8}(x) = f_{4}(x) + 6/9^{2}.
  \end{split}
\end{equation}

The set $X$ can be expressed as the union of the image of twelve cylinder sets which specify a single element in the first four places. For example, there are those numbers which have $(9, D)$-representations $(x_{j})_{j=1}^{\infty}$ such that $x_{1} = 0$, $x_{2} = 0$, $x_{3} = 1$, and $x_{4} = 3$. These numbers are contained in the image of $X$ under $f_{3}$. This is because $X$ contains the image of the cylinder $y_{1} = 0$, $y_{2} = 6$. The fourth place of the image of such a $y$ under $f_{3}$ is $6/9^{4} + 6/9^{4} = 1/9^{3} + 3/9^{3}$. On the other hand, $X$ contains the image of the cylinder given by $y_{1} = 0$ and $y_{2} = 0$. The image of that set under $f_{3}$ is the image of the cylinder $x_{1} = 0$, $x_{2} = 0$, $y_{1} = 0$, $y_{2} = 6$. This is contained in $X$. A similar analysis can be made of the remaining cylinder sets to verify that $\bigcup\limits_{i=1}^{8}f_{i}(X) = X$. 


Suppose we fix $v\in \mathbb{R}^{n}$ and choose $A = 9I$ where $I$ denotes the identity matrix. The argument above applies to the case
\begin{equation*}
(D_{j})_{j=1}^{\infty} =  (\{0\}, \{0, 6v\}, \overline{\{0, v\}, \{0, 3v, 6v\}})
\end{equation*}
by changing the translations in the one dimensional example to scalings of $v$. Moreover, for any finite sequence of subsets, $C_{1}, C_{2}, \ldots, C_{N}$, we could apply the same argument to
\begin{equation*}
(D_{j})_{j=1}^{\infty} = (C_{1}, C_{2}, \ldots, C_{N}, \{0\}, \{0, 6v\}, \overline{C_{1}, C_{2}, \ldots, C_{N}, \{0, v\}, \{0, 3v, 6v\}})
\end{equation*}
by taking the power of $9^{-1}I$ to be $N+2$ instead of $2$. 
\end{example}

To ensure an equivalence between self-similarity and the SEP property, we require further restrictions on the pair $(A, D)$ that defines the set. This is why, at the very least, the uniqueness of $(A, D-D)$-representations is assumed in Theorem~\ref{thm:attractSEP}. In particular, the assumption that $(A, D-D)$-representations are unique is sufficient to show that if a set of the form $\pi_{A, D}\left(\prod\limits_{j=1}^{\infty}D_{j}\right)$, with $D_{j}\subset D$ is self-similar, then $(D_{j} - \beta_{j})_{j=1}^{\infty}$ is SEP for some sequence of $(\beta_{j})_{j=1}^{\infty}$. In order to establish that claim, we make use of the following technical lemma concerning the SEP condition. 

\begin{lemma}\label{lem:sepPP} 
Suppose $(U_{j})_{j=1}^{\infty}$ is a sequence of finite subsets of $\mathbb{R}^{n}$ such that $U_{j} \subset U_{j + q}$ for some positive integer $q$. Suppose there exists a uniform bound, for all $j$, on the cardinality of $U_{j}$. If there exists a sequence of subsets $(V_{j})_{j=1}^{\infty}$ of $\mathbb{R}^{n}$ such that $U_{j} + V_{j} = U_{j+q}$, then $(U_{j})_{j=1}^{\infty}$ is SEP. 
\end{lemma}

\begin{proof}
The uniform bound on the cardinality of $U_{j}$ for all $j$ together with the assumption $U_{j} \subset U_{j+q}$ implies the existence of a positive integer $p$, such that $U_{j + qp} = U_{j + q(p + k)}$ for all $k\geq0$ and $j = 1, 2, \ldots, q$. It follows that for $j = 1, 2, \ldots, q$, 
\begin{align}
U_{j + qp} &= U_{j+ q(p-1)} +V_{j + q(p-1)} \\
&= U_{j + q(p-2)} + V_{j + q(p-1)} + V_{j+q(p-2)} \\
&\;\;\vdots \\
&= U_{j} + \sum_{\ell = 1}^{p} V_{j + q(p-\ell)}. 
\end{align}
By choosing $W_{j + qk} = \sum\limits_{\ell = 1}^{p-k}V_{j + q(p-\ell)}$ for $j = 1, 2, \ldots, q$ and $k = 0, 1, \ldots, p - 1$, we obtain $(U_{j})_{j=1}^{\infty} = (U_{1}U_{2}\ldots U_{qp})\overline{(U_{1} + W_{1})(U_{2}+W_{2})\ldots (U_{qp}+W_{qp})}$. 
\end{proof}

We now apply this lemma to establish ``self-similarity implies SEP". The proof is a generalization of an argument in \cite{PP14}. 

\begin{lemma}\label{lem:selfSimGen} 
Suppose that an expanding matrix $A\in M_{n}(\mathbb{R})$ and a finite set $D\subset\mathbb{R}^{n}$ are chosen such that $(A, D-D)$-representations are unique. Suppose $(D_{j})_{j=1}^{\infty}$ is a sequence of subsets of $D$. If the set $\pi_{A, D}\left(\prod\limits_{j=1}^{\infty}D_{j}\right)$ is the attractor of an IFS of the form $\{f_{i}(x) = A^{-p}x + r_{i}\}_{i=1}^{N}$ where $p$ is a positive integer and $r_{i}\in\mathbb{R}^{n}$ for each $i$, then there exists a sequence $(\beta_{j})_{j=1}^{\infty} \in D^{\mathbb{N}}$ such that $(D_{j} - \beta_{j})_{j=1}^{\infty}$ is SEP. 
\end{lemma}

\begin{proof} 
For convenience, let $X = \pi_{A, D}\left(\prod\limits_{j=1}^{\infty}D_{j}\right)$. Let $I$ denote the $n$ by $n$ identity matrix. We claim that $A^{-p} - I$ is invertible. If it was not, then $A^{-p}$ would have the number one as an eigenvalue. This is a contradiction since $A^{-p}$ is a contraction. Let $\beta = (I-A^{-p})^{-1}r_{1}$. By Lemma~\ref{lem:attractorShift}, $X - \beta$ is the attractor of the IFS $\{g_{i}(x) = A^{-p}x+s_{i}\}_{i=1}^{N}$ where $s_{i} = r_{i} - r_{1}$. Therefore $g_{1}(x) = A^{-p}x$. It follows that $X - \beta$ contains the origin. Therefore $\beta$ is an element of $X$ and has an $(A, D)$-representation $(\beta_{j})_{j=1}^{\infty}$. The presence of the origin also means that $g_{i}(0) = s_{i}$ is an element of $X - \beta$ for $i = 1, 2, \ldots, N$. It follows that each $s_{i}$ can be expanded into $\sum\limits_{j=1}^{\infty}A^{-j}s_{i, j}$ with $s_{i, j}\in E_{j} := D_{j} -\beta_{j}$ for each $j$. 

Define $S_{j} := \{s_{i, j+p}: i = 1, 2, \ldots, N\}$. We argue that $E_{j} + S_{j} = E_{j + p}$ for all $j\geq1$. 

Fix a positive integer $k$ and assume $d$ is an element of $E_{k} + S_{k}$. There exists $x_{k}^{'}\in E_{k}$ and $s_{i, k+p}\in S_{k}$, for some $i$, such that $d = x_{k}^{'} + s_{i, k+p}$. Choose $x\in X - \beta$ with $(A, D-D)$ representation $(x_{j})_{j=1}^{\infty}$ such that $x_{k} = x_{k}^{'}$. Since $x$ is an element of $X-\beta$, so is the image $g_{i}(x) = \sum\limits_{j=1}^{p}A^{-j}r_{i, j} + \sum\limits_{j=1}^{\infty}A^{-(p+j)}(x_{j} + r_{i, j+p})$. Therefore it has an expansion $\sum\limits_{j=1}^{\infty}A^{-j}e_{j}$ with $e_{j}\in E_{j}$ for each $j$. It follows that
\begin{equation}
\sum_{j=1}^{\infty}A^{-(p+j)}x_{j} = \sum_{j=1}^{\infty}A^{-j}(e_{j} - s_{i, j}). 
\end{equation}
For each $j$, $x_{j}$ is an element of $E_{j} = D_{j} - \beta_{j}$ and thus $x_{j}$ is an element of $D-D$. Similarly, $e_{j}$ and $s_{i,j}$ are both elements of $D_{j}-\beta_{j}$ and thus $e_{j} - s_{i, j}$ is an element of $D_{j} - D_{j}\subset D-D$. Both sides of the equation are $(A, D-D)$-representations and so by uniqueness we obtain $x_{j} = e_{j+p} - s_{i, j + p}$. In particular, $d = x_{k} + s_{i, k+p} = x_{k}^{'} + s_{i, k+p}$ is equal to $e_{k+p}$. We conclude that $d$ is an element of $E_{k+p}$. 

Observe now that if $E_{k} + S_{k} \subsetneq E_{k+p}$, then 
\begin{align}
\bigcup_{i=1}^{N}g_{i}(X-\beta) &\subset \pi_{A, D-D}(E_{1}\times\cdots\times E_{k+p-1} \times (E_{k} + S_{k}) \times \prod\limits_{j=1}^{\infty}E_{k + p + j}) \\
&\subsetneq \pi\left(\prod\limits_{j=1}^{\infty}E_{j}\right) = X - \beta. 
\end{align}
This contradicts the fact that $X-\beta$ is the attractor of $\{g_{i}\}_{i=1}^{N}$. Therefore $E_{j} + S_{j} = E_{j+p}$ for all $j\geq1$. 

Since $g_{1}$ is among the contractions that generate $X$, it follows that $E_{j} \subset E_{j+p}$ for all $j\geq1$. We also recall that $E_{j} = D_{j} - \beta_{j}$ is a subset of the finite set $D-D$ for all $j\geq1$. By Lemma~\ref{lem:sepPP}, it follows that $(D_{j} - \beta_{j})_{j=1}^{\infty}$ is SEP. 
\end{proof}

We now prove Theorem~\ref{thm:attractSEP}.

\begin{proof} [Proof of Theorem~\ref{thm:attractSEP}]
By Lemma~\ref{lem:seqExpress}, the intersection $T_{A, D}\cap(T_{A, D} + \alpha)$ is equal to $\pi_{A, D}\left(\prod\limits_{j=1}^{\infty}(D\cap(D+\alpha_{j})\right)$. For each $j$, $D\cap(D+\alpha_{j})$ is a subset of $D$. If there exists a sequence of $(\beta_{j})_{j=1}^{\infty}\in D^{\mathbb{N}}$ such that $((D\cap(D+\alpha_{j}))-\beta_{j})_{j=1}^{\infty}$ is SEP, then $\pi_{A, D}\left(\prod\limits_{j=1}^{\infty}(D\cap(D+\alpha_{j})\right)$ is the attractor of an IFS of the desired form by Lemma~\ref{lem:sepGen}. The converse is achieved by a direct application of Lemma~\ref{lem:selfSimGen}. 
\end{proof}

\section{Intersections Involving Self-Similar Sets}

If multiplication by a matrix $L\in M_{n}(\mathbb{R})$ is a similarity on $\mathbb{R}^{n}$ with contraction coefficient $c$, it follows from $L(0) = 0$ that $\norm{L(x)} = c\norm{x}$ for all $x\in \mathbb{R}^{n}$. In the context of Theorem~\ref{thm:attractSEP}, it follows that $T_{A, D} \cap (T_{A, D} + \alpha)$ is self-similar when multiplication by $A^{-1}$ is a similarity. In such cases, the Hausdorff and box-counting dimensions agree and in fact may be equal to similarity dimension of the IFS given in Theorem~\ref{thm:attractSEP}. In particular, we determine when this occurs as a consequence of the strong separation condition (SSC). For the sake of convenience, we provide the definition of the SSC from Chapter~\ref{chp:IFS}. 

\begin{definition} 
Suppose $X$ is the attractor of an IFS given by $\mathcal{F} = \{f_{i}\}_{i=1}^{N}$. We say that $\mathcal{F}$ satisfies the \textit{strong separation condition} (SSC) if the images $f_{i_{1}}(X)$ and $f_{i_{2}}(X)$ are disjoint for every pair $1 \leq i_{1} < i_{2} \leq N$. 
\end{definition}

\begin{lemma}\label{lem:Urep} 
Suppose $A\in M_{n}(\mathbb{R})$ is an expanding matrix and $D\subset\mathbb{R}^{n}$ is finite. If $\alpha\in\mathbb{R}^{n}$ has a unique $(A, D-D)$-representation, then every $x \in T_{A, D} \cap (T_{A, D} + \alpha)$ has a unique $(A, D)$-representation.
\end{lemma}

\begin{proof} 
Suppose that $x$ is an element of $T_{A, D} \cap (T_{A, D} + \alpha)$ and has two $(A, D)$-representations, $(x_{j})_{j=1}^{\infty}$ and $(x_{j}^{'})_{j=1}^{\infty}$. By definition there exists $y = \pi_{A, D}(y_{j})_{j=1}^{\infty}$ in $T_{A, D}$ such that $x = y + \alpha$. Therefore $\alpha = x - y = \pi_{A, D-D}(x_{j} - y_{j})_{j=1}^{\infty} = \pi_{A, D-D}(x_{j}^{'} - y_{j})_{j=1}^{\infty}$. The uniqueness of the $(A, D-D)$-representation of $\alpha$ implies that $x_{j} - y_{j} = x_{j}^{'} - y_{j}$ for each $j$. It follows that $(x_{j})_{j=1}^{\infty} = (x_{j}^{'})_{j=1}^{\infty}$. 
\end{proof}

\begin{theorem}\label{thm:snSSC} 
Suppose that $A \in M_{n}(\mathbb{R})$ is invertible and multiplication by $A^{-1}$ is a similarity on $(\mathbb{R}^{n}, \norm{\cdot})$ and $D\subset\mathbb{R}^{n}$ is finite. Suppose further that $\alpha\in\mathbb{R}^{n}$ has a unique $(A, D-D)$ representation and there exists $\{U_{\ell}\}_{\ell=1}^{p}$ and $\{V_{\ell}\}_{\ell=1}^{p}$ such that $((D\cap(D+\alpha_{j}) - \beta_{j})_{j=1}^{\infty} = (U_{\ell})_{\ell=1}^{p}\overline{(U_{\ell} + V_{\ell})_{\ell=1}^{p}}$ for some sequence $(\beta_{j})_{j=1}^{\infty}\in D^{\mathbb{N}}$.

The IFS precisely containing the maps
\begin{equation*}
f(x) = A^{-p}\bigg(x + \sum_{\ell=1}^{p}(A^{p-\ell}u_{\ell} + A^{-\ell}v_{\ell})\bigg)
\end{equation*}
with $u_{\ell} \in U_{\ell}$ and $v_{\ell} \in V_{\ell}$ for each $\ell$ satisfies the strong separation condition if and only if $|U_{\ell} + V_{\ell}| = |U_{\ell}||V_{\ell}|$ for each $\ell$. 
\end{theorem}

\begin{proof}
Define $\beta := \pi_{A, D}(\beta_{j})_{j=1}^{\infty}$ and let $\mathcal{F}$ denote the IFS of interest. By Theorem~\ref{thm:attractSEP}, the set $X := T(\alpha) - \beta$ is the attractor of the IFS $\mathcal{F}$. We first prove (i) by assuming that the condition $|U_{\ell} + V_{\ell}| = |U_{\ell}||V_{\ell}|$ holds for each $\ell$. 

Assume that $h, g\in\mathcal{F}$ and that there exist $x_{1}$, $x_{2}\in X$ such that $h(x_{1}) = g(x_{2})$. Let us denote the parameters of $h$ and $g$ by $u_{\ell}^{(h)}, v_{\ell}^{(h)}$ and $u_{\ell}^{(g)}, v_{\ell}^{(g)}$ respectively. By assumption, $x_{k} = \pi_{A, D-D}(x_{j}^{(k)})_{j=1}^{\infty}$ for $k=1, 2$ where $x_{j}^{(k)}\in (D\cap(D+\alpha_{j})) - \beta_{j}$. It follows that 
\begin{align}
h(x_{1}) &= \pi_{A, D-D}(u_{1}^{(h)},\ldots u_{p}^{(h)},(v_{1}^{(h)} + x_{1}^{(1)}),\ldots (v_{p}^{(h)} + x_{p}^{(1)}),x_{p+1}^{(1)},x_{p+2}^{(1)},\ldots) \label{eq:ha} \\
g(x_{2}) &= \pi_{A, D-D}(u_{1}^{(g)},\ldots u_{p}^{(g)},(v_{1}^{(g)} + x_{1}^{(2)}),\ldots (v_{p}^{(g)} + x_{p}^{(2)}),x_{p+1}^{(2)},x_{p+2}^{(2)},\ldots) \label{eq:ga}
\end{align}
where we recall that for $k=1, 2$, the coefficient $x_{\ell}^{(k)}$ is an element of $U_{\ell}$ for $\ell = 1, 2, \ldots, p$ and of $U_{\ell} + V_{\ell}$ for $\ell > p$. It follows that both (\ref{eq:ha}) and (\ref{eq:ga}) are images of $(A, D-D)$-representations under $\pi_{A, D-D}$. Adding $\beta$ to both radix expansions converts them to images of $(A, D)$ expansions under $\pi_{A, D}$. By Lemma~\ref{lem:Urep}, it follows that $u_{\ell}^{(h)} = u_{\ell}^{(g)}$ and $v_{\ell}^{(h)} + x_{\ell}^{(1)} = v_{\ell}^{(g)} + x_{\ell}^{(2)}$ for each $\ell = 1, 2, \ldots, p$. The assumption $|U_{\ell} + V_{\ell}| = |U_{\ell}||V_{\ell}|$ implies that $x_{\ell}^{(1)} = x_{\ell}^{(2)}$ and, in particular, $v_{\ell}^{(h)} = v_{\ell}^{(g)}$. We conclude that $h = g$. If $h \neq g$, then $h(X)$ and $g(X)$ are disjoint. 

Conversely, suppose $|U_{\ell} + V_{\ell}| < |U_{\ell}||V_{\ell}|$ for some $\ell$. Without loss of generality, this implies the existence of $u_{1}, u_{1}^{'}\in U_{1}$ and $v_{1}, v_{1}^{'}\in V_{1}$ such that $u_{1} + v_{1} = u_{1}^{'} + v_{1}^{'}$ and $u_{1} \neq u_{1}^{'}$, $v_{1} \neq v_{1}^{'}$. We can choose $t = \pi_{A, D-D}(t_{j})_{j=1}^{\infty}$ and $t^{'} = \pi_{A, D-D}(t_{j}^{'})_{j=1}^{\infty}$ in $X$ such that $t_{1} = u_{1}, t_{1}^{'} = u_{1}^{'}$ and $t_{j} = t_{j}^{'}$ for all $j \geq 2$. Fix a finite sequence $(v_{\ell})_{\ell=1}^{\infty}$ and define function $f(x) = A^{-p}(x + A^{-1}v_{1} +\sum\limits_{\ell=2}^{p}(A^{p-\ell}u_{\ell} + A^{-\ell}v_{\ell}))$ for some fixed pairs $(u_{\ell}, v_{\ell}) \in U_{\ell} \times V_{\ell}$ for $\ell = 2, 3, \ldots, p$. Similarly define $g(x) = A^{-p}(x + v_{1}^{'} + \sum\limits_{\ell=2}^{p}(A^{p-\ell}u_{\ell} + A^{-\ell}v_{\ell}))$. By construction $f(t) = g(t^{'})$ and $f, g$ are elements of $\mathcal{F}$. Therefore $f(X) \cap g(X) \neq \emptyset$. 
\end{proof}

\begin{corollary}\label{cor:simDim}
Suppose that $A \in M_{n}(\mathbb{R})$ is such that multiplication by its inverse is a similarity. Suppose $\alpha\in\mathbb{R}^{n}$ has a unique $(A, D-D)$ representation. Suppose there exists $\{U_{\ell}\}_{\ell=1}^{p}$ and $\{V_{\ell}\}_{\ell=1}^{p}$ such that $((D\cap(D+\alpha_{j}) - \beta_{j})_{j=1}^{\infty} = (U_{\ell})_{\ell=1}^{p}\overline{(U_{\ell} + V_{\ell})_{\ell=1}^{p}}$ for some sequence of $\beta_{j} \in D^{\mathbb{N}}$. 

If $|U_{\ell} + V_{\ell}| = |U_{\ell}||V_{\ell}|$ for each $\ell$, then 
\begin{equation}
\dim_{H}T(\alpha) = \dim_{B}T(\alpha) = \frac{\sum\limits_{\ell=1}^{p}(\log{|U_{\ell}||V_{\ell}|})}{-p\log{c}}
\end{equation}
where $T(\alpha) := T_{A, D} \cap (T_{A, D} + \alpha)$ and $c$ is the contraction coefficient of $A^{-1}$. 
\end{corollary}

\begin{proof}
Define $\beta := \pi_{A, D}(\beta_{j})_{j=1}^{\infty}$. By Theorem~\ref{thm:attractSEP}, the set $T(\alpha) - \beta$ is the attractor of the IFS $\mathcal{F}$ which precisely contains the maps $f(x) = A^{-p}(x + \sum\limits_{\ell=1}^{p}(A^{p-\ell}u_{\ell} + A^{-\ell}v_{\ell}))$ where $a_{\ell} \in A_{\ell}$ and $u_{\ell} \in U_{\ell}$ for each $\ell$. By Theorem~\ref{thm:snSSC}, the IFS $\mathcal{F}$ satisfies the SSC. It follows from Theorem~\ref{thm:SepSimDim} that the Hausdorff dimension and box-counting dimension are equal to the value $s$ satisfying $\sum\limits_{f\in\mathcal{F}}c_{f}^{s} = 1$, where $c_{f}$ is the contraction coefficient of $f\in\mathcal{F}$. Since $A^{-1}$ has a contraction coefficient $c$, each map in $\mathcal{F}$ is a similarity with contraction coefficient $c^{p}$. Furthermore, the cardinality of $\mathcal{F}$ is $\prod\limits_{\ell=1}^{p}|A_{\ell}||U_{\ell}|$. The equation simplifies to $\prod\limits_{\ell=1}^{p}|U_{\ell}||V_{\ell}|c^{ps} = 1$. Isolating for $s$ completes the proof. 
\end{proof}

We illustrate Theorem~\ref{thm:snSSC} using an example where the strong separation condition does not hold and then apply it to a case where it does hold. 

\begin{example}\label{ex:noSSC} 
Let $A = \begin{bmatrix}
-3 & -1 \\
1  & -3 \\
\end{bmatrix}$.
This matrix encodes the linear transformation $L(z) = (-3+i)z$ for all $z\in\mathbb{C}$. We choose $D = \{0, 4e_{1}, 8e_{1}\}$ where $e_{1}\in\mathbb{R}^{2}$ is the standard basis vector $[1, 0]^{t}$ corresponding to the number $1$ in the complex plane. The inverse of $A$ is the matrix
$\sqrt{10}^{-1}\begin{bmatrix}
-3/\sqrt{10} & 1/\sqrt{10} \\
-1/\sqrt{10}  & -3/\sqrt{10} \\
\end{bmatrix}$. 
It can be verified directly that multiplication on $\mathbb{R}^{2}$ by this matrix is a similarity with contraction coefficient $\sqrt{10}^{-1}$. We choose $\alpha$ to be the element of $\mathbb{R}^{2}$ with $(A, D-D)$-representation $(-4e_{1}, -8e_{1}, \overline{0, 8e_{1}})$. This choice of representation yields
\begin{equation} 
(D\cap (D + \alpha_{j}))_{j=1}^{\infty} = (\{0, 4e_{1}\}, \{0\}, \overline{\{0, 4e_{1}, 8e_{1}\}, \{8e_{1}\}}).
\end{equation}

This sequence is SEP with period $2$ because $\{0, 4e_{1}, 8e_{1}\} = \{0, 4e_{1}\} + \{0, 4e_{1}\}$ and $\{8e_{1}\} = \{0\} + \{8e_{1}\}$. It follows from Theorem~\ref{thm:attractSEP}, that $T(\alpha) = T_{A, D}\cap(T_{A, D}+\alpha)$ is the attractor of the IFS given by the maps

\begin{equation}
 \begin{split}
    &f_{1}(x) = A^{-2}(x + 8A^{-2}e_{1}), \\
    &f_{2}(x) = A^{-2}x + (4A^{-1}e_{1}+8A^{-2}e_{1}), \\
  \end{split}
\quad
  \begin{split}
    &f_{3}(x) = f_{1}(x) + 4A^{-1}e_{1}, \\
    &f_{4}(x) = f_{2}(x) + 4A^{-1}e_{1}\\
  \end{split}
\end{equation}

Consider $y$ and $z$ in $T(\alpha)$ with the $(A, D)$-representations $(4e_{1}, 0, \overline{0, 8e_{1}})$ and $(0, 0, \overline{0, 8e_{1}})$ respectively. Since $f_{1}(y) = f_{2}(z)$, we have that $f_{1}(T(\alpha)) \cap f_{2}(T(\alpha))$ is nonempty for two distinct elements of the IFS. As an aside, notice that we could have chosen any tail for $y$ and $z$ besides the repetition of $(0, 8e_{1})$ (assuming we use the same tail in both representations). So we see that the overlap is arguably large relative to $T(\alpha)$. This a consequence of $|\{0, 4\}||\{0, 4\}| = (2)(2) = 4 \neq 3 = |\{0, 4e_{1}, 8e_{1}\}|$. It is sensible to suggest that the open set condition (see Definition~\ref{def:openSep}) may hold despite the negation of the SSC. If so, then the Hausdorff dimension of $T(\alpha)$ would still be the similarity dimension of this IFS. We eventually show that this is not the case (Corollary~\ref{cor:OSCisSSC}). 

Now we consider $T(\beta)$ where $\beta$ has $(A, D-D)$-representation $(\beta_{j})_{j=1}^{\infty} = (-4, -8, \overline{-4, 8})$. This choice of representation yields
\begin{equation} 
(D\cap (D + \beta_{j}))_{j=1}^{\infty} = (\{0, 4e_{1}\}, \{0\}, \overline{\{0, 4e_{1}\}, \{8_{1}\}}).
\end{equation}
This sequence is SEP with period $2$ because $\{0, 4e_{1}\} = \{0, 4\} + \{0\}$ and $\{8\} = \{0\} + \{8\}$. In this case, both the cardinalities of the sumsets are equal to the product of the cardinalities of their components. It follows that in addition to $T(\beta)$ being the attractor of the IFS containing the two maps $f(x) = A^{-2}(x + 8A^{-2}e_{1})$ and $g(x) = A^{-2}(x + 4Ae_{1} + 8A^{-2}e_{1})$, it has Hausdorff dimension $\log{2}/\log{10}$ by Corollary~\ref{cor:simDim}. 

We remark that we did not verify that the translations had unique $(A, D-D)$-representations in either case. This will be rectified in Chapter~\ref{chp:neighbours} (Theorem~\ref{thm:gaussianEquiv}). 
\end{example}

For the last example in this section, we demonstrate that Theorem~\ref{thm:attractSEP} can be applied to and consequently yield self-affine sets whose Hausdorff and box-counting dimensions do not agree. 

\begin{definition} 
Let $n > m \geq 2$. Let $D$ be a subset of $\{0, \ldots, m-1\} \times \{0, \ldots, n-1\}$. We call the attractor of an IFS on $\mathbb{R}^{2}$ precisely containing the functions $f_{d}(x) = B(x+d)$ where $B = \begin{bmatrix} 1/m& 0 \\ 0 & 1/n \end{bmatrix}$ and $d\in D$, a \textit{Bedford-McMullen Carpet}.
\end{definition}

Bedford-McMullen carpets are a subset of $[0, 1] \times [0, 1]$. Partition the unit interval on the horizontal axis into $1/m$ equal subintervals and similarly partition the vertical unit interval into $1/n$ subintervals. This describes an $n$ by $m$ grid over the unit square. Each vector $d \in \{0, \ldots, m-1\} \times \{0, \ldots n-1\}$ that appears in $D$ is a rectangle in the grid. Let $N$ be the total number of rectangles, let $M$ be the total number of columns with at least one rectangle corresponding to a vector in $D$, and lastly let $N_{i}$ be the number of rectangles in the $i$th column, $i = 1, 2, \ldots, m$, corresponding to vectors in $D$. The Hausdorff and box-counting dimensions of a Bedford-McMullen Carpet $F$ are known \cite{M84} and have values
\begin{align} 
\dim_{H}F &= \frac{\log(\sum\limits_{i=1}^{M}N_{i}^{\log(m)/\log(n)})}{\log(m)} \label{eq:BMH} \\
\dim_{B}F &= \frac{\log(M)}{\log(m)} + \frac{\log(N/M)}{\log(n)}, \label{eq:BMB}
\end{align}
respectively. 

\begin{example}
Let $(A, D)$ be the pair
\begin{equation}
\left(\begin{bmatrix} 7& 0 \\ 0 & 10 \end{bmatrix}, \left\{\begin{bmatrix} 0 \\ 0 \end{bmatrix}, \begin{bmatrix} 0 \\ 6 \end{bmatrix}, \begin{bmatrix} 0 \\ 9 \end{bmatrix}, \begin{bmatrix} 3 \\ 0 \end{bmatrix}, \begin{bmatrix} 6 \\ 0 \end{bmatrix}, \begin{bmatrix} 3 \\ 3 \end{bmatrix}, \begin{bmatrix} 3 \\ 6 \end{bmatrix}, \begin{bmatrix} 6 \\ 3 \end{bmatrix}, \begin{bmatrix} 6 \\ 6 \end{bmatrix}\right\}\right).
\end{equation}
Let $\alpha$ be the element with $(A, D-D)$-representation $(\alpha_{j})_{j=1}^{\infty}= (\overline{[3, 0]^{t}, [0, 3]^{t}})$. This yields
\begin{align}
D\cap(D+\alpha_{2j-1}) &= \left\{\begin{bmatrix} 3 \\ 0 \end{bmatrix}, \begin{bmatrix} 6 \\ 0 \end{bmatrix}, \begin{bmatrix} 3 \\ 6 \end{bmatrix}, \begin{bmatrix} 6 \\ 3 \end{bmatrix}, \begin{bmatrix} 6 \\ 6 \end{bmatrix}\right\}, \\
D\cap(D+\alpha_{2j})&=\left\{\begin{bmatrix} 3 \\ 3 \end{bmatrix}, \begin{bmatrix} 3 \\ 6 \end{bmatrix}, \begin{bmatrix} 6 \\ 3 \end{bmatrix}, \begin{bmatrix} 6 \\ 6 \end{bmatrix}, \begin{bmatrix} 0 \\ 9 \end{bmatrix}\right\},
\end{align}
for all $j\geq1$. By Theorem~\ref{thm:attractSEP}, the resulting IFS that generates $T_{A, D}\cap(T_{A, D} + \alpha)$ is precisely the set of functions $f(x) = A^{-2}(x+b)$ where $b$ is any of the twenty-five vectors of the form $Ad + d^{'}$ where $d\in D\cap(D+\alpha_{1})$ and $d^{'}\in D\cap(D+\alpha_{2})$.

The set $T_{A, D}$ and its intersection with its translation by $\alpha$ are examples of Bedford-McMullen carpets. The direct use of the formulas (\ref{eq:BMH}) and (\ref{eq:BMB}), applied to $T_{A, D} \cap (T_{A, D} + \alpha)$ produces values approximately equal to $1.536$ and $1.540$ respectively. We justify the implicit claim that the $(A, D-D)$-representations are unique in Remark~\ref{rem:realBasePhenom}. 
\end{example}

In general, it may be that the SEP condition only characterizes a subset of the $\alpha$ for which $T_{A, D} \cap (T_{A, D}+\alpha)$ is self-similar. Example~\ref{ex:nonSEP} demonstrates the existence of sets of the form $\pi_{A, D}\left(\prod\limits_{j=1}^{\infty}D_{j}\right)$ that are self-similar despite the sequence $(D_{j}-\beta_{j})_{j=1}^{\infty}$ failing to satisfy the SEP condition for any $(\beta_{j})_{j=1}^{\infty}$. Further research is required to determine if there exists a pair $(A, D)$ and a translation $\alpha$ for which $T_{A, D} \cap (T_{A, D} + \alpha)$ is among the sets described in Example~\ref{ex:nonSEP}. By Lemma~\ref{lem:selfSimGen}, it must be that there exist $(A, D-D)$-representations in the context of Example~\ref{ex:nonSEP} which are not unique. In Section~\ref{sec:nonUniqueReps}, we discuss the situation when $\alpha$ has multiple $(A, D-D)$-representations. Before then, we spend the next chapter examining the functions that map $\alpha$ to its Hausdorff and box-counting dimensions respectively when each $\alpha$ does have a unique representation.

\chapter{Dimensions of Self-Similar Sets and their Translates}\label{chp:limFormula}

\section{A Limit Formula for the Dimension of the Intersection}

Given an expanding matrix $A$ and a finite subset $D\subset\mathbb{R}^{n}$, the intersection $T(\alpha) := T_{A, D}\cap(T_{A, D} + \alpha)$ is not self-similar for every $\alpha$. Even if the intersection is self-similar, the IFS described in Theorem~\ref{thm:attractSEP} may not satisfy the open set condition. In either case, we cannot expect the Hausdorff dimension or box-counting dimension of $T(\alpha)$ to be the similarity dimension of a known IFS. We require a different formulation if we are to make claims about the dimension of $T(\alpha)$ in these cases. 

To achieve this goal, we restrict our study to matrices with integer entries whose inverses are similarities on $\mathbb{R}^{n}$ with respect to Euclidean distance. It is sensible to wonder what kind restriction this places on the entries of the matrix and if it is easy to construct examples. Recall that an orthogonal matrix is a real square matrix whose columns and rows are orthonormal vectors. We can check this property for a given matrix because the definition is equivalent to an invertible matrix with real entries whose inverse is equal to its transpose. 

\begin{theorem}\label{thm:simForm} 
Suppose that $A \in M_{n}(\mathbb{R})$ is invertible. The function $L:\mathbb{R}^{n} \rightarrow \mathbb{R}^{n}$ given by $L(x) = A^{-1}x$ is a similarity with respect to Euclidean distance if and only if $A = rO$ where $|r|>1$ and $O$ is an orthogonal matrix.
\end{theorem}

\begin{proof}
Suppose that $A^{-1}$ is a similarity with contraction coefficient $c$. Multiplication by $c^{-1}A^{-1}$ is then an isometry. The isometries of Euclidean space are functions of the form $f(x) = Ox + t$ where $O$ is an orthogonal matrix and $t\in \mathbb{R}^{n}$. Since $c^{-1}A^{-1}$ is a linear transformation, $c^{-1}A^{-1} = O$ for some orthogonal matrix $O$. Therefore $A = c^{-1}O^{-1}$. Since $c \in (0, 1)$ and the inverse of an orthogonal matrix is still orthogonal, the proof of this direction is complete. Conversely, if $A = rO$ where $|r| > 1$ and $O$ is orthogonal, then $A^{-1} = r^{-1}O^{-1}$. Then for all $x\in\mathbb{R}^{n}$, $\norm{A^{-1}x} = |r|^{-1}\norm{x}$. We conclude that multiplication by $A^{-1}$ is a similarity with contraction coefficient $1/|r|$. 
\end{proof}

\begin{corollary}\label{cor:detCoeff}
Suppose that $A \in M_{n}(\mathbb{R})$ is invertible. If the function $L: \mathbb{R}^{n}\rightarrow\mathbb{R}^{n}$ given by $L(x) = A^{-1}x$ is a similarity with respect to Euclidean distance, then its contraction coefficient is $|\det(A)|^{-1/n}$. 
\end{corollary}

\begin{proof} 
By Theorem~\ref{thm:simForm}, $A = rO$ where $|r|>1$ and $O$ is an orthogonal matrix. Since $O$ preserves Euclidean distance, the contraction coefficient of $L$ must be $|r|^{-1}$. Observe that $\det(O) = \det(O^{t}) = \det(O^{-1}) = 1/\det(O)$. It must be that $\det(O) = \pm 1$. Therefore $|\det(A)| = |r|^{n}$. 
\end{proof}

Unfortunately, a similarity with respect to Euclidean distance is not necessarily a similarity with respect to other norms on $\mathbb{R}^{n}$. 

\begin{example} 
Consider multiplication by the matrix
$$B = 3^{-1}\begin{bmatrix}1/\sqrt{2} & -1/\sqrt{2} \\ 1/\sqrt{2} & 1/\sqrt{2}\end{bmatrix}.$$
This action rotates elements of $\mathbb{R}^{2}$ by $45$ degrees counter-clockwise relative to the origin and then scales them down by a third. Let $\norm{\cdot}_{2}$ denote the Euclidean norm. For any $v = [x, y]^{t}\in\mathbb{R}^{2}$, 
\begin{equation}
\norm{Bv}_{2} = 3^{-1}\sqrt{2^{-1}((x-y)^{2} + (x+y)^{2})} = 3^{-1}\sqrt{x^{2} + y^{2}} = 3^{-1}\norm{v}_{2}. 
\end{equation}
Therefore multiplication by $B$ on $(\mathbb{R}^{2}, \norm{\cdot}_{2})$ is a similarity with contraction coefficient $3^{-1}$. Now consider the same function on $\mathbb{R}^{2}$ equipped with the taxi-cab norm. Recall that the taxi-cab norm is defined as
$\norm{v}_{1} := |x| + |y|$ for any $v = [x, y]^{t}\in\mathbb{R}^{2}$. Let $v_{1} = [1/\sqrt{2}, 1/\sqrt{2}]^{t}$ and let $v_{2} = [1, 0]^{t}$. Multiplication by $B$ does not scale the magnitude of these two vectors equally. 
We can verify directly that $\norm{Bv_{1}}_{1} = 3^{-1} = (\sqrt{2}/6)\norm{v_{1}}_{1}$. Meanwhile, $\norm{Bv_{2}}_{1} = 2/(3\sqrt{2}) = (\sqrt{2}/3)\norm{v_{2}}_{1}$. 
\end{example}

Let us proceed with our study of the fractal dimensions of $T_{A, D} \cap (T_{A, D} + \alpha)$ when $A$ has integer entries and its inverse is a similarity with respect to Euclidean distance. We remark that, by the equivalence of all norms in $\mathbb{R}^{n}$, many of the proofs in this chapter could be modified to apply to invertible matrices whose inverse is a similarity with respect to a non-standard norm on $\mathbb{R}^{n}$. The Hausdorff dimension is invariant under a change to an equivalent norm (\cite{F90}, corollary 2.4). 

\begin{definition} 
Let $A$ be an element of $M_{n}(\mathbb{Z})$. A set $\mathcal{D} \subset \mathbb{Z}^{n}$ is called a \textit{complete residue system mod $A$} if the canonical map from $\mathcal{D} \rightarrow \faktor{\mathbb{Z}^{n}}{A\mathbb{Z}^{n}}$ is bijective. 
\end{definition}

The following theorem in \cite{S16} implies that we can use complete residue systems of invertible matrices in our construction of self-affine sets. 

\begin{theorem}[\cite{S16}, theorem 2.3]\label{thm:sashaLOL}
Suppose $A$ is an $m$ by $n$ matrix with entries in a principle ideal domain $R$. Let $v_{j}$ denote the $j$th column of $A$. The quotient module $\faktor{R^{n}}{(v_{1}, \ldots, v_{n})}$ where $(v_{1}, \ldots, v_{n})$ is the free $R$-module generated by $v_{1}, \ldots, v_{n}$ is isomorphic to the direct sum $\faktor{R}{a_{1}R} \bigoplus \cdots \bigoplus \faktor{R}{a_{m}R}$ where $a_{1}, \ldots, a_{m}$ are the diagonal entries of the Smith normal form of $A$. 
\end{theorem}

Let $A\in M_{n}(\mathbb{Z})$ be invertible. We apply Theorem~\ref{thm:sashaLOL} to $A$ when $R = \mathbb{Z}$. Since the cardinality of the direct sum is the absolute value of the product of the diagonal entries of the Smith normal form of $A$, which is equal to $|\det(A)|$, we conclude that $\faktor{\mathbb{Z}^{n}}{A\mathbb{Z}^{n}}$ is $|\det(A)|$. In particular, it is finite. 

\begin{remark}
A coincidental aside: Henry Smith (of Smith normal form) wrote the paper ``On the integration of discontinuous functions" in which he describes the construction of the attractor of the similarities $f(x) = (x+d)/m$ where $m > 2$ and $d\in\{0, 1, \ldots, m-2\}$ (example 15 of \cite{S75}) before Georg Cantor popularized what we call the middle third Cantor set.
\end{remark}

\begin{definition} 
Let $\mathcal{F} = \{f_{i}\}_{i=1}^{N}$ be an iterated function system and let $k$ be a positive integer. We call the set $(f_{i_{1}}\circ f_{i_{2}} \circ \cdots \circ f_{i_{k}})(X)$  the \textit{$k$-tile of $\mathcal{F}$ corresponding to $(i_{1}, i_{2}, \ldots, i_{k})$} where $X$ is the attractor of $\mathcal{F}$.  
\end{definition}

Suppose that $A$ is an expanding matrix and $\mathcal{D}$ is a complete residue system mod $A$. The $k$-tiles of $T_{A, \mathcal{D}}$ are the sets $\{\pi_{A, D}(d_{j})_{j=1}^{\infty} : d_{j}=e_{j}, j = 1, \ldots, k\}$ for some finite sequence $e_{1}, e_{2}, \ldots, e_{k} \in \mathcal{D}$. Our intention is to capture the box-counting dimension of subsets of $T_{A, D}$ using $k$-tiles instead of $\delta$-mesh squares or balls of radius $\delta$. 

\begin{lemma} \label{lem:boxtile}
Let $X$ be a bounded subset of $\mathbb{R}^{n}$. Suppose $(\{U_{i}^{(k)}\}_{i=1}^{I_{k}})_{k=1}^{\infty}$ is a sequence of coverings of $X$ that satisfy the following properties:
\begin{itemize}
\item[(i)] For every $i$, $\diam{U_{i}^{(k)}} = r_{k}$ is a strictly decreasing sequence. 

\item[(ii)] The sequence $(r_{k})_{k=1}^{\infty}$ tends to zero as $k$ tends to infinity. 

\item[(iii)] There exists $c\in(0, 1)$ such that $r_{k+1}/r_{k} \geq c$. 

\item[(iv)] There exists a positive integer $M$, for every $k$, such that an $n$-dimensional cube with side length $r_{k}$ intersects at most $M$ elements of $\{U_{i}^{(k)}\}_{i=1}^{I_{k}}$. 
\end{itemize}
Then
\begin{align}
\overline{\dim}_{B}X &= \limsup_{k\rightarrow\infty}\frac{\log U_{k}(X)}{-\log{r_{k}}}, \\
\underline{\dim}_{B}X &= \liminf_{k\rightarrow\infty}\frac{\log U_{k}(X)}{-\log{r_{k}}},
\end{align}
where $U_{k}(X)$ denotes the number of elements of $\{U_{i}^{(k)}\}_{i=1}^{I_{k}}$ that intersect $X$. 
\end{lemma}

\begin{proof}
Let $N_{\delta}(X)$ be the smallest number of sets of diameter $\delta$ needed to cover $X$. Let $S_{\delta}(X)$ be the number of $\delta$-mesh squares that intersect $X$. For each $k$, we have $N_{r_{k}}(X) \leq U_{k}(X)$ by definition. By property (ii) we have $U_{k}(X) \leq MS_{r_{k}}(X)$. Therefore 
\begin{equation}
\frac{\log N_{r_{k}}(X)}{-\log r_{k}} \leq \frac{\log U_{k}(X)}{-\log r_{k}} \leq \frac{\log(MS_{r_{k}}(X))}{-\log r_{k}}.
\end{equation}
The result follows from taking the limit superior or limit inferior of either side. To see why it is sufficient to go along the subsequence $r_{k}$, consider $r\in[r_{k+1}, r_{k}]$. Observe that
\begin{align}
\frac{\log N_{r}(X)}{-\log r} &\leq \frac{\log N_{r_{k+1}}(X)}{-\log r_{k}}   \\
&= \frac{\log N_{r_{k+1}}(X)}{-\log r_{k+1} + \log (r_{k+1}/r_{k})} \\
&\leq \frac{\log N_{r_{k+1}}(X)}{-\log r_{k+1} + \log c}.
\end{align}

Taking the the limit superior of both sides yields $$\limsup\limits_{k\rightarrow \infty}\frac{\log N_{r}(X)}{\log r} \leq \limsup\limits_{k\rightarrow \infty}\frac{\log N_{r_{k+1}}(X)}{\log r_{k+1}}.$$ On the other hand, $\sup\limits_{\ell\geq k} \frac{\log N_{r_{\ell+1}}(X)}{-\log r_{\ell+1}} \leq \sup\limits_{t\in (0, r_{k}]} \frac{\log N_{t}(X)}{-\log{t}}$. The proof is complete because the upper bound converges to the limit superior as $k$ tends to infinity. Notice that we could replace $N_{\delta}$ with $S_{\delta}$ with no changes to the argument. To demonstrate the property for the limit inferior, a similar argument can be applied after observing that $\frac{\log N_{r}(X)}{-\log r} \geq \frac{\log N_{r_{k}}(X)}{-\log r_{k+1}}$.
\end{proof}

\begin{corollary}\label{lem:boxtile2}
Suppose that $A \in M_{n}(\mathbb{Z})$ is invertible and multiplication by its inverse is a similarity with respect to Euclidean distance. Suppose that $\mathcal{D}$ is a complete residue system mod $A$. If $F$ is a nonempty subset of $T_{A, \mathcal{D}}$, then
\begin{align}
\overline{\dim}_{B}F &= \limsup_{k\rightarrow\infty}\frac{\log N_{k}(F)}{k\log{|\det(A)|^{1/n}}}, \\
\underline{\dim}_{B}F &= \liminf_{k\rightarrow\infty}\frac{\log N_{k}(F)}{k\log{|\det(A)|^{1/n}}},
\end{align}
where $N_{k}(F)$ denotes the number of $k$-tiles of $T_{A, \mathcal{D}}$ that intersect $F$. 
\end{corollary}

\begin{proof} 
For the sake of convenience, let $T$ denote $T_{A, \mathcal{D}}$. We check the conditions of Lemma~\ref{lem:boxtile} for the collection of covers of the form $\{T_{d_{1}, d_{2}, \dots, d_{k}}: d_{j} \in \mathcal{D}\}, k\geq1$. The fact that they are covers can be verified using the self-similarity of $T$. Specifically, $T = \cup_{d\in\mathcal{D}}T_{d}$ implies that for each fixed $d^{'}\in \mathcal{D}$, $T_{d^{'}} = \cup_{d\in\mathcal{D}}T_{d^{'}, d}$. Therefore $T$ is equal to $\cup_{(d_{1}, d_{2}) \in \mathcal{D}\times\mathcal{D}}T_{d_{1}, d_{2}}$. Iterating this procedure will verify the claim for $k\geq 3$. 

By Corollary~\ref{cor:detCoeff}, we have $\diam{T_{d_{1}, d_{2}, \dots, d_{k}}} = |\det(A)|^{-k/n}\diam{T}$ for each positive integer $k$. Conditions (i) and (ii) hold since $|\det(A)|^{-1/n} \in (0, 1)$. Condition (iii) holds because the ratio of the diameter of a $k+1$ tile over the diameter of a $k$-tile is $|\det(A)|^{-1/n}$. 

To prove that (iv) holds, first consider a ball, which we denote by $V$, of radius $\sqrt{n}\diam{T}/2$. Since the diameter of $\cup_{t\in T} (V - t)$ is bounded by $(\sqrt{n} + 1)\diam{T}$, there is an upper bound $N$ on the cardinality of $\mathbb{Z}^{n} \cap \cup_{t\in T} (V - t)$. We claim that $N$ satisfies the property described in (iv). Notably, $N$ does not depend on $k$. 

Suppose $R_{k}$ is an $n$-dimensional cube with side length $|\det(A)|^{-k/n}\diam{T}$ and that a $k$-tile, $T_{d_{1}, d_{2}, \ldots, d_{k}}$ intersects $R_{k}$ for some fixed $d_{1}, d_{2}, \ldots, d_{k}\in\mathcal{D}$. Let $x$ be the center of $R_{k}$. The set $R_{k}$ is contained in the closed ball of radius $\sqrt{n}|\det(A)|^{-k/n}\diam{T}/2$ centered at $x$. It follows that $A^{k}T_{d_{1}, d_{2}, \ldots, d_{k}}$ intersects the ball of radius $\sqrt{n}|\det(A)|^{-k/n}\diam{T}/2$ centered at $A^{k}x$. This implies the existence of an element $\mathbb{R}^{n}$ of the form $\zeta + t$ in the ball centered at $A^{k}x$ where $\zeta$ is the element of $\mathbb{Z}^{n}$ given by $\sum\limits_{j=1}^{k}A^{k-j}d_{j}$ and $t$ is an element of $T$. Let $V$ denote the ball centered at $A^{k}x$. The vector $\zeta$ is an element of $\bigcup\limits_{t\in T} (V - t)$. The assumption that $\mathcal{D}$ is a complete residue system modulo $A$ implies that $\zeta$ is uniquely determined by $d_{1}, d_{2}, \ldots, d_{k}$ and therefore the number of $k$-tiles that intersect $R_{k}$ is bounded by the number of elements of $\mathbb{Z}^{n}$ that are contained in $\bigcup\limits_{t\in T} (V - t)$. This is the positive integer $N$. 
\end{proof}

The subsets $F \subset T_{A, \mathcal{D}}$ we want to apply \ref{lem:boxtile2} to are those of the form $T_{A, D} \cap (T_{A, D}+\alpha)$ when $\alpha\in T_{A, D-D}$. When the $(A, D-D)$-representation of $\alpha$ is unique, we can derive an expression for the upper and lower box-counting dimension of $T_{A, D} \cap (T_{A, D}+\alpha)$ in terms of that unique representation. We prove a lemma before giving the precise statement. 

\begin{lemma} \label{lem:kbound} 
Suppose that $A \in M_{n}(\mathbb{Z})$ is invertible and multiplication by its inverse is a similarity with respect to Euclidean distance. Suppose that $\mathcal{D}$ is a complete residue system mod $A$. There exists a bound, for every $k$, on the number of $k$-tiles of $T_{A, \mathcal{D}}$ that any given $k$-tile of $T_{A, \mathcal{D}}$ can intersect. 
\end{lemma}

\begin{proof}
Observe that $T = T_{A, \mathcal{D}}$ is a bounded set and thus there are only $M$ elements $g\in\mathbb{Z}^{n}$ for which $T$ and $T + g$ have a nonempty intersection. Fix $d_{1}, d_{2}, \ldots, d_{k}\in \mathcal{D}$ for some $k$. Let $a := A^{k-1}d_{1} + A^{k-2}d_{2} + \cdots +Ad_{k-1} + d_{k}$. The $k$-tiles $T_{d_{1}, d_{2}, \ldots, d_{k}}$ and $T_{d_{1}^{'}, d_{2}^{'}, \ldots, d_{k}^{'}}$ intersect exactly when $a-b + T$ intersects $T$ where $b = A^{k-1}d_{1}^{'} + A^{k-2}d_{2}^{'} + \cdots +Ad_{k-1}d_{k-1}^{'} + d_{k}^{'}$. Recall that $\mathcal{D}$ is a complete set of residues mod $A$. Therefore the function on the set of words formed by the digits in $\mathcal{D}$ defined by $(d_{1}, \ldots, d_{k}) \rightarrow \sum\limits_{j=1}^{\infty}A^{j}d_{k-j}$ is injective. It follows that every $k$-tile that intersects $T_{d_{1}, d_{2}, \ldots, d_{k}}$ yields a unique $g^{'}\in\mathbb{Z}^{n}$ such that the intersection of $T$ and $T+g^{'}$ is nonempty. We conclude that $T_{d_{1}, d_{2}, \ldots, d_{k}}$ intersects at most $M$ $k$-tiles, otherwise we contradict our initial observation. 
\end{proof}

\begin{theorem}\label{thm:unboundedDigitFormula} 
Suppose that $A \in M_{n}(\mathbb{Z})$ is invertible and multiplication by its inverse is a similarity with respect to Euclidean distance. Suppose that $D$ is a subset of a complete residue system mod $A$. If $(D_{j})_{j=1}^{\infty}$ is a sequence of subsets of $D$
\begin{align}
\underline{\dim}_{B}\pi_{A, D}\left(\prod\limits_{j=1}^{\infty}D_{j}\right) &= \liminf_{k\rightarrow\infty} \frac{\log(G_{k})}{k\log(|\det(A)|^{1/n})}, \\
\overline{\dim}_{B}\pi_{A, D}\left(\prod\limits_{j=1}^{\infty}D_{j}\right) &= \limsup_{k\rightarrow\infty} \frac{\log(G_{k})}{k\log(|\det(A)|^{1/n})},
\end{align}
where $G_{k} := \prod\limits_{j=1}^{k}|D_{j}|$. 
\end{theorem}

\begin{proof}
Let $X$ denote $\pi_{A, D}\left(\prod\limits_{j=1}^{\infty}D_{j}\right)$. From Lemma~\ref{lem:boxtile2}, it suffices to prove that $\lambda G_{k} \leq N_{k}(X) \leq \rho G_{k}$ for real contants $\lambda$ and $\rho$. Any $k$-tile $T_{d_{1}, d_{2}, \ldots, d_{k}}$ with $d_{j}\in D_{j}$ for $j=1, 2, \ldots, k$ intersects $X$. Therefore $G_{k} \leq N_{k}(X)$. On the other hand, for all $k\geq 1$, the set $X$ is contained in the union of $k$-tiles whose defining digits are in $D\cap(D+\alpha_{j})$ for each $j = 1, 2, \ldots, k$. By Lemma~\ref{lem:kbound}, it follows that $N_{k}(X) \leq M G_{k}$ for some sufficiently large constant $M$. 
\end{proof}

We provide examples of the application of Theorem~\ref{thm:unboundedDigitFormula} to the intersection of a self-similar set $T_{A, D}$ and with one of its translates by $\alpha$. In each example, $\alpha$ has a unique $(A, D-D)$-representation. Lemma~\ref{lem:seqExpress} then implies that $T(\alpha) = \pi\left(\prod\limits_{j=1}^{\infty}D\cap(D+\alpha_{j})\right)$. We do not justify the claim that the choices of $\alpha$ have unique $(A, D-D)$-representations in this chapter because it requires a detailed analysis of the particular choice of $(A, D)$. It is a consequence of Theorem~\ref{thm:gaussianEquiv}. 

\begin{example} \label{ex:noOSC} 
We recall Example~\ref{ex:noSSC}. The matrix $A$ was $\begin{bmatrix}
-3 & -1 \\
1  & -3 \\
\end{bmatrix}$. We chose $D = \{0, 4e_{1}, 8e_{1}\}$ where $e_{1}\in\mathbb{R}^{2}$ is the standard basis vector corresponding to the number $1$ in the complex plane. If $\alpha$ is the vector with $(A, D-D)$-representation $(-4e_{1}, -8e_{1}, \overline{0, 8e_{1}})$, then we have 
\begin{equation} 
(D\cap (D + \alpha_{j}))_{j=1}^{\infty} = (\{0, 4e_{1}\}, \{0\}, \overline{\{0, 4e_{1}, 8e_{1}\}, \{8e_{1}\}}).
\end{equation}
It follows that for $k \geq 2$
\begin{equation}
G_{k} = \begin{cases} 
2(3^{(k/2) - 1}) &\;\text{if}\;k \equiv 0 \mod{2},\\
2(3^{((k+1)/2)-1}) &\;\text{if}\;k \equiv 1 \mod{2}.
\end{cases}
\end{equation}
Therefore $2(3^{(k/2) - 1}) \leq G_{k} \leq 2(3^{((k+1)/2)-1})$ for all $k$. In particular, 
\begin{equation}
\frac{2\log{2}}{k\log{10}} + \frac{(k-2)\log{3}}{k\log{10}} \leq \frac{\log(G_{k})}{k\log{10}}\leq \frac{2\log{2}}{k\log{10}} + \frac{(k-1)\log{3}}{k\log{10}}.
\end{equation}
By Theorem~\ref{thm:unboundedDigitFormula}, we conclude that $\dim_{B}(T_{A, D}\cap(T_{A, D} + \alpha)) = \log{3}/\log{10}$. It is notable that the IFS in Example~\ref{ex:noSSC} that generated the intersection has similarity dimension $\log{4}/\log{10}$. 
\end{example}

\begin{example} 
In this example, we show that there are choices of $A, D$, and $\alpha$ for which the box-counting dimension of $T_{A, D} \cap (T_{A, D} + \alpha)$ does not exist. Let $A$ be the same matrix as in Example~\ref{ex:noOSC} and $D = \{0, 4e_{1}\}$. Let $\alpha = \pi_{A, D-D}(\alpha_{j})_{j=1}^{\infty}$ where $(\alpha_{j})_{j=1}^{\infty}$ is the $(A, D-D)$-representation given by
\begin{equation}
\alpha_{j} = \begin{cases} 4e_{1} &\;\text{if}\; 10^{k} + 1 \leq j \leq 10^{k+1}, \; k\geq0 \;\text{even}, \\
0 &\;\text{if}\; 10^{k}+1 \leq j \leq 10^{k+1}, \; k>0 \;\text{odd or}\; $j = 1$.
\end{cases}
\end{equation}
Observe that $|D\cap (D+4e_{1})| = 1$. It follows that
\begin{equation}
\frac{\log{G_{10^{k}}(x)}}{10^{k}} = \begin{cases} 
\frac{1+ 90 + 9000 + \cdots + 9(10)^{k-1}}{10^{k}}\log{2} &\;\text{if}\; k \;\text{even}\;, \\
\frac{1+ 90 + 9000 + \cdots + 9(10)^{k-2}}{10^{k}}\log{2} &\;\text{if}\; k \;\text{odd}.
\end{cases}
\end{equation}
When $k$ is even, we have $\frac{\log{G_{10^{k}}(x)}}{10^{k}} \geq \frac{9}{10}\log{2}$. Theorem~\ref{thm:unboundedDigitFormula} implies that $\overline{\dim}_{B}(T_{A, D}\cap(T_{A, D}+\alpha)) \geq \frac{18}{10}\log{2}/\log{10}$. When $k$ is odd, $\frac{\log{G_{10^{k}}(x)}}{10^{k}} \leq \frac{1}{10}\log{2}$, and therefore $\underline{\dim}_{B}(T_{A, D}\cap(T_{A, D}+\alpha)) \leq \frac{2}{10}\log{2}/\log{10}$.
\end{example}

Example~\ref{ex:noOSC} demonstrates that the generating IFS given in Theorem \ref{thm:attractSEP} does not satisfy the open set condition. If it did, then the similarity dimension of the IFS and the box-counting dimension of the self-similar set would agree. We generalize this argument in Corollary~\ref{cor:OSCisSSC} using the following theorem. 

\begin{theorem}\label{thm:IntDimSEP} 
Suppose that $A \in M_{n}(\mathbb{Z})$ is invertible and multiplication by its inverse is a similarity with respect to Euclidean distance. Suppose that $D\subset\mathbb{Z}^{n}$ is a subset of a complete residue system mod $A$. If $\alpha\in T_{A, D-D}$ has a unique $(A, D-D)$-representation $(\alpha_{j})_{j=1}^{\infty}$ and there exist $\{U_{\ell}\}_{\ell=1}^{p}$ and $\{V_{\ell}\}_{\ell=1}^{p}$ such that $((D\cap(D+\alpha_{j}) - \beta_{j})_{j=1}^{\infty} = (U_{\ell})_{\ell=1}^{p}\overline{(U_{\ell} + V_{\ell})_{\ell=1}^{p}}$ for some sequence $(\beta_{j})_{j=1}^{\infty}\in D^{\mathbb{N}}$, then 
\begin{equation}
\dim_{H}T_{A, D}\cap(T_{A, D} + \alpha) = \frac{\log\left(\prod\limits_{\ell=1}^{p}|U_{\ell} + V_{\ell}|\right)}{p\log(|\det(A)|^{1/n})}. 
\end{equation}
\end{theorem}

\begin{proof}
Let $T(\alpha)$ denote the intersection $T_{A, D} \cap (T_{A, D} + \alpha)$. Firstly, Lemma~\ref{lem:seqExpress} implies that $T(\alpha) = \pi\left(\prod\limits_{j=1}^{\infty}D\cap(D+\alpha_{j})\right)$. Theorem~\ref{thm:unboundedDigitFormula} then asserts that $\underline{\dim}_{B}T(\alpha) = \liminf_{k\rightarrow\infty}\frac{\log \left(\prod\limits_{j=1}^{k}|D\cap(D+\alpha_{j})|\right)}{k\log(|\det(A)|^{p/n})}$. Given a positive integer $k$, let $r_{k} \in \{0, 1, \ldots, p-1\}$ be the remainder left by division of $k$ by $p$. For $k$ sufficiently large, the quantity $\frac{\log\prod\limits_{j=1}^{k}|D\cap(D+\alpha_{j})|}{k\log(|\det(A)|^{1/n})}$ can be decomposed into the sum of the following three terms:
\begin{align} \label{eq:sepDim}
 &\frac{\log\left(\prod\limits_{j=1}^{p}|D\cap(D+\alpha_{j})|\right)}{k\log(|\det(A)|^{1/n})} + \frac{\log\left(\prod\limits_{j=p+1}^{k-r_{k}}|D\cap(D+\alpha_{j})|\right)}{k\log(|\det(A)|^{1/n})} + \\ & \frac{\log\left(\prod\limits_{j=k-r_{k}+1}^{k}|D\cap(D+\alpha_{j})|\right)}{k\log(|\det(A)|^{1/n})}.
\end{align}
The first term in (\ref{eq:sepDim}) tends to zero as $k$ tends to infinity. Since $|D\cap(D+\alpha_{j})| = |(D\cap(D+\alpha_{j})) - \beta_{j}|$ for each $j$, we can use the SEP assumption to obtain

\begin{align}
\frac{\log\left(\prod\limits_{j=p+1}^{k-r_{k}}|D\cap(D+\alpha_{j})|\right)}{k\log(|\det(A)|^{1/n})} &= \frac{\log\left[\left(\prod\limits_{\ell=1}^{p}|U_{\ell} + V_{\ell}|\right)^{((k-r_{k})/p) - 1}\right]}{k\log(|\det(A)|^{1/n})} \\
&= \frac{((k-r_{k})/p)-1}{k}\frac{\log\left(\prod\limits_{\ell=1}^{p}|U_{\ell} + V_{\ell}|\right)}{p\log(|\det(A)|^{1/n})}\\
\end{align}
which tends to $\frac{\log\left(\prod\limits_{\ell=1}^{p}|U_{\ell} + V_{\ell}|\right)}{p\log(|\det(A)|^{1/n})}$ as $k$ tends to infinity. Meanwhile, since $D\cap(D+\alpha_{j}) \subset D$ for each $j$ and $r_{k} \leq p-1$, we can bound the third term of (\ref{eq:sepDim}) by $\frac{(p-1)\log{|D|}}{k\log(|\det(A)|^{1/n})}$. This bounding sequence tends to zero as $k$ tends to infinity. The claim now follows because the limit inferior agrees with the limit and the box-counting dimension and Hausdorff dimension agree for self-similar sets. 
\end{proof}

\begin{corollary}\label{cor:OSCisSSC} 
Suppose that $A \in M_{n}(\mathbb{Z})$ is invertible and multiplication by $A^{-1}$ is a similarity with respect to Euclidean distance. Suppose $D\subset\mathbb{Z}^{n}$ is a subset of complete residue system mod $A$. Suppose further that $\alpha\in\mathbb{R}^{n}$ has a unique $(A, D-D)$ representation and there exists $\{U_{\ell}\}_{\ell=1}^{p}$ and $\{V_{\ell}\}_{\ell=1}^{p}$ such that $((D\cap(D+\alpha_{j}) - \beta_{j})_{j=1}^{\infty} = (U_{\ell})_{\ell=1}^{p}\overline{(U_{\ell} + V_{\ell})_{\ell=1}^{p}}$ for some sequence $(\beta_{j})_{j=1}^{\infty}\in D^{\mathbb{N}}$.

If the IFS precisely containing the maps
\begin{equation*}
f(x) = A^{-p}\bigg(x + \sum_{\ell=1}^{p}(A^{p-\ell}u_{\ell} + A^{-\ell}v_{\ell})\bigg)
\end{equation*}
with $u_{\ell} \in U_{\ell}$ and $v_{\ell} \in V_{\ell}$ for each $\ell$ satisfies the open set condition, then it must satisfy the strong separation condition. 
\end{corollary}

\begin{proof}
Suppose that the strong separation condition does not hold. By Theorem~\ref{thm:snSSC}, there exists $\ell \in \{1, 2, \ldots, p\}$ such that $|A_{\ell} + U_{\ell}| < |A_{\ell}||U_{\ell}|$. By Theorem~\ref{thm:IntDimSEP}, the Hausdorff dimension is strictly less than the similarity dimension. The open set condition must not hold because, otherwise, Theorem~\ref{thm:SepSimDim} would imply that the Hausdorff dimension is equal to the similarity dimension. 
\end{proof}

\section{On the Function Mapping Translations to Dimensions} 

Let $C$ be the classical middle thirds Cantor set. In our notation, this is the attractor generated by $T_{3, \{0, 2\}}$. One of the first results to be proven about the intersection $ C(\alpha) := C \cap (C+\alpha)$ concerns the level sets of the function $\Phi : (C-C) \rightarrow [0, \dim_{H}C]$ given by $\Phi(\alpha) = \dim_{H}C(\alpha)$. In \cite{DH95}, it was shown that the level sets of $\Phi$ are dense in the function's domain $(C-C)$. More recently, a condition on $D$ was shown to be sufficient to have an analogous result about the box-counting dimension hold for the two-dimensional case corresponding to the linear transformation $L(z) = (-n+i)z$ in the complex plane where $n$ is a positive integer greater than or equal to $2$ \cite{PS21}. In this section, we generalize the argument to expanding matrices with integer entries. In Section~\ref{sec:caseStudy}, we demonstrate that our generalization can be combined with deeper knowledge of the $-n+i$ case to yield a more widely applicable theorem than the one featured in \cite{PS21}. 

\begin{definition} 
Let $A\in M_{n}(\mathbb{R})$ be an expanding matrix and $D$ be a finite subset of $\mathbb{R}^{n}$. 
We call the set $\{\alpha\in\mathbb{R}^{n} : T_{A, D} \cap (T_{A, D} + \alpha) \neq \emptyset\}$ the \textit{nonempty translations of $T_{A, D}$} and denote it by $F_{A, D}$. 
\end{definition}

For any $S\subset \mathbb{R}^{n}$ and $t\in\mathbb{R}^{n}$, the set $S \cap (S+t)$ is nonempty exactly when $t$ is an element of the difference set $S-S$. Therefore $F_{A, D}$ is equal to $T_{A, D-D}$. Despite this, we introduce the notation $F_{A, D}$ because there is a difference between $T_{A, D-D}$ and the set $\{(\alpha, \beta): T\cap(T+\alpha)\cap(T+\beta) \neq \emptyset\}$. We treat multiple intersections in Section~\ref{sec:multiSec}. 

\begin{definition} 
Let $A \in M_{n}(\mathbb{Z})$ be invertible and let multiplication by its inverse be a similarity on $\mathbb{R}^{n}$ with respect to Euclidean distance. Let $D\subset\mathbb{Z}^{n}$ be a subset of a complete residue system mod $A$. We define the function
\begin{equation}
\begin{split}
&\Phi_{A, D} :\{\alpha\in F_{A, D}: \dim_{B}(T_{A, D}\cap(T_{A, D} + \alpha))\;\;\text{exists}\;\;\} \rightarrow [0, \dim_{B}T_{A, D}],\\
&\Phi_{A, D}(\alpha) = \dim_{B}(T_{A, D}\cap(T_{A, D} + \alpha)).\\
\end{split}
\end{equation} 
\end{definition} 

\begin{corollary}\label{cor:boxDense} 
Suppose $A \in M_{n}(\mathbb{Z})$ is invertible and multiplication by its inverse is a similarity with respect to Euclidean distance. Suppose $D\subset\mathbb{Z}^{n}$ is a subset of a complete residue system mod $A$. If $(A, D-D)$-representations are unique, then the level sets of $\Phi_{A, D}$ are dense subsets of $F_{A, D}$. 
\end{corollary}

\begin{proof}
Let $\alpha \in F_{A, D}$ be given. We construct $\beta$ in the domain of $\Phi_{A, D}$ such that $\Phi_{A, D}(\beta) = \lambda\dim_{B}T_{A, D}$ where $\lambda \in [0, 1]$. Let $\alpha$ have $(A, D-D)$-representation $(\alpha_{j})_{j=1}^{\infty}$. For any radius $r>0$, there exists an index $m$ at which any vector with  $(A, D-D)$-representation $(\alpha_{1}\alpha_{2}\ldots \alpha_{m}\beta_{m+1}\beta_{m+2}\ldots)$ is within distance $r$ of $\alpha$. We now describe how to choose the $\beta_{j}$ such that the image of the $(A, D-D)$-representation satisfies the desired properties. 

Suppose $0 < \lambda < 1$. There exists a sequence of integers $h_{j}$ such that $h_{j} \leq j\lambda < h_{j} + 1$. Since $\lambda < 1$ we see that $(j+1)\lambda < j\lambda + 1 < (h_{j}+1)+1$. It follows that either $h_{j+1} = h_{j}$ or $h_{j+1} = h_{j} + 1$. Fix $v$ to be an element of $D-D$ of maximal norm. For all $j > m$, let 
\begin{equation}
\beta_{j} = \begin{cases}
v &\;\text{if}\; h_{j} = h_{j-1}, \\
0 &\;\text{if}\; h_{j} = h_{j-1} + 1. 
\end{cases}
\end{equation}

We claim that $|D\cap(D+\beta_{j})|$ is equal to either $1$ or $|D|$ for $j>m$. The latter case is clear when $\beta_{j} = 0$. The former happens when $\beta_{j} = v$. If it did not, then there exists four distinct elements in $D$, $u_{1}, u_{2}, u_{3}, u_{4}$ such that $u_{1} + v = u_{2}$ and $u_{3} + v = u_{4}$. We claim that $||u_{4} - u_{1}||$ or $||u_{3} - u_{2}||$ is greater than $||v||$. This is true because either all four points are on the same line, or they form a parallelogram whose side lengths are less than one of its diagonals. 

The key observation is that in either case we obtain $|D\cap(D+\beta_{j})| = |D|^{h_{j}-h_{j-1}}$. 

We can compute the lower box-counting dimension of $T(\beta) := T_{A, D} \cap (T_{A, D} + \beta)$ using Theorem~\ref{thm:unboundedDigitFormula} since, by assumption, the $(A, D-D)$-representation of $\beta$ is unique. Observe that 
\begin{align}
&\lim_{k\rightarrow\infty}\frac{\log{G_{k}}}{k\log(|\det(A)|^{1/n})} \\
&= \lim_{k\rightarrow\infty}\frac{\log(\prod_{j=1}^{m}|D\cap(D+\alpha_{j})|)}{k\log(|\det(A)|^{1/n})} + \lim_{k\rightarrow\infty}\frac{\log(\prod_{j=m+1}^{k}|D\cap(D+\beta_{j})|)}{k\log(|\det(A)|^{1/n})} \\
&= 0 + \lim_{k\rightarrow\infty}\frac{\log(\prod_{j=m+1}^{k}|D|^{h_{j}-h_{j-1}})}{k\log(|\det(A)|^{1/n})} \\
&=  \lim_{k\rightarrow\infty}\frac{(h_{k}-h_{m})\log{|D|}}{k\log(|\det(A)|^{1/n})} \\
&= \lambda\frac{\log{|D|}}{\log(|\det(A)|^{1/n})}.
\end{align}
Since the limit exists, we have $\dim_{B}T(\beta) = \underline{\dim_{B}}T(\beta) = \lambda\frac{\log{|D|}}{\log(|\det(A)|^{1/n})}$. The uniqueness of $(A, D-D)$-representations implies the uniqueness of $(A, D)$-representations. Therefore the IFS containing the similarities $f(x) = A^{-1}(x+d)$ where $d\in D$, which generates $T_{A, D}$, satisfies the strong separation condition. We conclude by Theorem~\ref{thm:SepSimDim} that $\frac{\log{|D|}}{\log(|\det(A)|^{1/n})} = \dim_{B}T_{A, D}$. 
 
For the case $\lambda = 0$, choose $\beta_{j} = v$ for all $j > m$. If $\lambda = 1$, then choose $\beta_{j} = 0$ for all $j > m$.
\end{proof}

\begin{remark}
A consequence of Corollary~\ref{cor:boxDense} is that the function $\Phi_{A, D}$ is discontinuous everywhere on its domain. This is because for any neighbourhood of a point $\alpha$, by the denseness of the level sets of $\Phi_{A, D}$, we can find $\beta$ in that neighbourhood satisfying $|\Phi_{A, D}(\alpha) - \Phi_{A, D}(\beta)| \geq \dim_{B}T_{A, D}/2$. 
\end{remark}

We can establish a version of Corollary~\ref{cor:boxDense} for the Hausdorff dimension when $T_{A, D}$ is ``nice". In the spirit of the strong separation condition and the open set condition,  ``nice'' means a restriction on the overlaps of the $1$-tiles of $T_{A, D}$. 

\begin{definition} 
Let $A$ be an expanding matrix and let $D\subset\mathbb{R}^{n}$ be finite. We call the attractor $T_{A, D}$ is a \textit{self-affine tile} if $T_{A, D}$ is the closure of its interior and $\lambda_{n}((T_{A, D} + d) \cap (T_{A, D} + d^{'})) = 0$ for all distinct pairs $d, d^{'}\in D$ where $\lambda_{n}$ is the $n$-dimensional Lebesgue measure. 
\end{definition}


\begin{example}
If $A$ is a positive integer $m > 1$ and $D = \{0, 1, \ldots, m-1\}$, then the set $T_{A, D}$ is the closed unit interval. The set $[0, 1]$ is a self-affine tile since it is the closure of $(0, 1)$ and $[d, d+1] \cap [d^{'}, d^{'}+1]$ is finite for all distinct pairs in $d, d^{'} \in D$. 
\end{example}

\begin{definition} 
Let $A \in M_{n}(\mathbb{Z})$ be invertible such that multiplication by $A^{-1}$ is a similarity with respect to Euclidean distance. Let $\mathcal{D}$ be a complete residue system modulo $A$ and $D\subset\mathcal{D}$. Let $T(\alpha) := T_{A, D} \cap (T_{A, D} + \alpha)$ for $\alpha \in \mathbb{C}$. We define the function
\begin{equation}
 \begin{split}
    &\Psi_{A, D} : F_{A, D} \rightarrow [0, \dim_{H}T_{A, D}],\\
    &\Psi_{A, D}(\alpha) = \dim_{H}T(\alpha).\\
  \end{split}
\end{equation}
\end{definition}

Under the conditions of Corollary~\ref{cor:boxDense}, we can show that the level sets of these functions are also dense in $F_{A, D}$ when $D\subset\mathcal{D}$ and $T_{A, \mathcal{D}}$ has minimal overlaps. Our tool is the notion of a Moran set. 

\begin{definition} 
Let $(n_{k})_{k=1}^{\infty}$ be a sequence of positive integers, and let $(R_{k})_{k=1}^{\infty}$ be a sequence of vectors $(c_{k, 1}, c_{k, 2}, \ldots, c_{k, n_{k}})$ such that $0 < c_{k, j} < 1$ for all $k\geq1$ and $1 \leq j \leq n_{k}$. Let $J\subset\mathbb{R}^{n}$ be a compact set with nonempty interior. Let $\mathcal{F}$ denote a set of subsets of $\mathbb{R}^{n}$ indexed by $W := \{\emptyset\}\cup(\cup_{k=1}^{\infty}W_{k})$ where $W_{k} := \{(j_{1}, j_{2}, \ldots, j_{k}) : 1\leq j_{\ell} \leq n_{\ell}, 1\leq\ell\leq k\}$. Given two words $\sigma$ and $\tau$ in $W$, we denote their concatenation by $\sigma * \tau$. The set $\mathcal{F}$ is said to \textit{satisfy the Moran structure given by $((n_{k})_{k=1}^{\infty}, (R_{k})_{k=1}^{\infty}, J)$} if it satisfies the following five conditions:

\begin{itemize} 
\item[(i)] $J_{\emptyset} = J$.
\item[(ii)] For any $\sigma\in W\setminus \{\emptyset\}$, there exists a similarity $S_{\sigma}:\mathbb{R}^{n}\rightarrow\mathbb{R}^{n}$ such that $S_{\sigma}(J) = J_{\sigma}$. 
\item[(iii)] For any $k\geq 1$ and $\sigma\in W_{k-1}$, we have $J_{\sigma*j} \subset J_{\sigma}$ for all $j = 1, 2, \ldots, n_{k}$. 
\item[(iv)] For any $k\geq 1$ and $\sigma\in W_{k-1}$, $1 \leq j \leq n_{k}$, $\frac{\diam(J_{\sigma*j})}{\diam(J_{\sigma})} = c_{k, j}$. 
\item[(v)] For any $k\geq 1$ and $\sigma\in W_{k-1}$, the interiors of $J_{\sigma*j}$ and $J_{\sigma*\ell}$ are disjoint for all $1 \leq j < \ell \leq n_{k}$. 
\end{itemize}

For $\mathcal{F}$ satisfying the Moran structure, set $E_{k} = \cup_{\sigma\in W_{k}}J_{\sigma}$ and $E = \cap_{k=1}^{\infty}E_{k}$. We call $E$ the \textit{Moran set associated with the collection $\mathcal{F}$}. 
\end{definition}

\begin{example}
The middle third Cantor set is a Moran set. Recall that it is the attractor of the similarities $f_{1}(x) = x/3$ and $f_{2}(x) = (x+2)/3$. 

For all $k$, we choose $n_{k} = 2$ and $R_{k} = (1/3, 1/3)$. Let $J = [0, 1]$ and let it have the empty word as a label. For each integer $k\geq1$,  let $W_{k} = \{(i_{1}, i_{2}, \ldots, i_{k}) : i_{j}\in\{1, 2\}\}$. We define $W_{0}$ to be the set that contains the empty word. We define the set $\mathcal{F} = \{J_{\sigma}:\sigma\in \cup_{k=10}^{\infty}W_{k}\}$ by assigning each $J_{\sigma}$ the set $f_{i_{1}}\circ f_{i_{2}} \circ \cdots \circ f_{i_{k}}([0, 1])$ when $\sigma = (i_{1}, i_{2}, \ldots, i_{k})$.We claim that $\mathcal{F}$ satisfies the Moran structure given by $((n_{k})_{k=1}^{\infty}, (R_{k})_{k=1}^{\infty}, J)$. 

Condition (i) is satisfied immediately. To see that (ii) holds, choose $S_{\sigma}$ to be the composition of similarities $(f_{i_{1}} \circ \cdots \circ f_{i_{k}})$, for $\sigma = (i_{1}, i_{2}, \ldots, i_{k})$. It follows from our definitions that $S_{\sigma}(J) = J_{\sigma}$. Condition (iii) holds because $f_{j}([0, 1])$ is a subset of $[0, 1]$ for both $j=1$ and $j=2$. For condition (iv), observe that the diameter $J_{\sigma}$ is $1/3^{k}$ where $k$ is the length of the word $\sigma$. If $\sigma$ is a word of length $k$, then $\sigma * j$ is a word of length $k+1$. It follows that $\diam{J_{\sigma*j}} / \diam{J_{\sigma}} = 1/3$. Lastly, if $j \neq \ell$, then without generality $j = 1$ and $\ell=2$. Since $f_{1}([0, 1]) \cap f_{2}([0, 1]) = [0, 1/3] \cap [2/3, 1]$ is empty, condition (v) holds. 
\end{example}

The following theorem from \cite{HRWW00} expresses the Hausdorff dimension of a Moran set in terms of a sequence derived from the ``contraction coefficients" $c_{k, j}$. This is done in a way that is analogous to the similarity dimension of an iterated function system of similarities. 

\begin{theorem}[Hua S., Rao H., Wen Z., Wu J., \cite{HRWW00}, theorem 1.1]\label{thm:moranDim} 
Suppose $E$ is the Moran set associated with a collection $\mathcal{F}$ satisfying the Moran structure given by  $((n_{k})_{k=1}^{\infty}, (R_{k})_{k=1}^{\infty}, J)$. For each $k$, let $s_{k}$ be the solution to the equation $\sum\limits_{\sigma\in W_{k}}c_{\sigma}^{s_{k}} = 1$ where $c_{\sigma} := \prod\limits_{\ell = 1}^{k}c_{\ell, j_{\ell}}$ for $\sigma = (j_{1}, j_{2}, \ldots, j_{k})$. If $\inf_{i, j}c_{i, j} > 0$, then
\begin{equation}
\dim_{H}E = \liminf_{k\rightarrow\infty}s_{k}.
\end{equation}
\end{theorem}

\begin{remark}
The original theorem also states that $\dim_{p}E$ is equal to the limit superior of the sequence $s_{k}$ where $\dim_{p}$ is the packing dimension. The packing dimension is dual to the Hausdorff dimension insofar that it roughly considers the supremum of powers of diameters of subsets that are contained by the set being measured as opposed to taking an infimum over coverings. 
\end{remark}

We intend to apply this to sets of the form $T_{A, D} \cap (T_{A, D} + \alpha)$. Under additional assumptions, these intersections are Moran sets. 

\begin{lemma}\label{lem:moranCheck} 
Suppose $A \in M_{n}(\mathbb{Z})$ is invertible and multiplication by its inverse is a similarity on $\mathbb{R}^{n}$ with respect to Euclidean distance. If $T_{A, \mathcal{D}}$ is a self-affine tile for $\mathcal{D}\subset\mathbb{Z}^{n}$, then for any sequence $(D_{j})_{j=1}^{\infty}$ of subsets of $\mathcal{D}$, the image $\pi_{A, \mathcal{D}}\left(\prod\limits_{j=1}^{\infty}D_{j}\right)$ is a Moran set. 
\end{lemma}

\begin{proof}
For each positive integer $k$, let $\mathcal{F}_{k}$ be the collection of all sets of the form $T_{d_{1}, \ldots, d_{k}} = \pi_{A, \mathcal{D}}(\{d_{1}\} \times \cdots \times \{d_{k}\} \times \mathcal{D}^{\mathbb{N}})$ where $d_{j}\in D_{j}$ for $j = 1, 2, \ldots, k$. We leave the details of seeing that $\pi_{A, \mathcal{D}}\left(\prod\limits_{j=1}^{\infty}D_{j}\right)$ is equal to $$\bigcap\limits_{k=1}^{\infty}(\bigcup\limits_{(d_{1}, d_{2}, \ldots, d_{k}), d_{j}\in D_{j}}T_{d_{1}, d_{2}, \ldots, d_{k}})$$ to the reader. 

For each $k$, let $R_{k}$ be the $|D_{k}|$ length vector where each entry is $|\det(A)|^{-1/n}$. We claim that $\mathcal{F} = \bigcup\limits_{k=1}^{\infty}\mathcal{F}_{k} \cup \{T_{A, \mathcal{D}}\}$ satisfies the Moran structure of $$((|D_{k}|)_{k=1}^{\infty},(R_{k})_{k=1}^{\infty}, T_{A, \mathcal{D}}).$$ We observe that $T_{A, \mathcal{D}}$ is compact because it is the attractor of an IFS and has nonempty interior because a self-affine tile is equal to the closure of its interior.

There exists a bijection between the set of integer labels $1, 2, \ldots, |D_{j}|$ and the set $D \cap (D+\alpha_{j})$. For each $j$, fix one of these bijections. It follows that there is a bijective correspondence between the words $W := \cup_{k=1}^{\infty}\{(j_{1}, j_{2}, \ldots, j_{k}) : 1\leq j_{\ell} \leq |D_{\ell}|, 1\leq\ell\leq k\}$ and the set of finite digit sequences $\bigcup\limits_{k=1}^{\infty}\{(d_{1}, d_{2}, \ldots, d_{k}) : d_{j}\in D_{j}\}$. We can assign an element of $\mathcal{F}$, $T_{d_{1}, d_{2}, \ldots, d_{k}}$, the index $\sigma \in W$ corresponding to the sequence of digits $(d_{1}, d_{2}, \ldots, d_{k})$. Moreover, we define $J_{\emptyset}$ to be $T_{A, \mathcal{D}}$. From here on, we abuse notation and treat $\sigma$ as if it is equal to its corresponding string of digits. 

Property (i) is satisfied by definition. Property (ii) is satisfied because for each $\sigma = (d_{1}, d_{2}, \ldots, d_{k})$, $T_{\sigma} = (f_{d_{1}} \circ f_{d_{2}} \cdots \circ f_{d_{k}})(T_{B, \mathcal{D}})$ where $f_{d}(x) = A^{-1}(x + d)$. Since maps of this form are similarities and compositions of similarities are similarities, assigning $S_{\sigma} = (f_{d{1}} \circ f_{d_{2}}\circ \cdots \circ f_{d_{k}})$ completes the verification of (ii). To verify condition (iii), observe that $T_{d_{1}, \ldots, d_{k-1}, d} \subset T_{d_{1}, \ldots, d_{k-1}}$ for any $d \in D_{k}$. Let $c$ denote $|\det(A)|^{-1/n}$. The diameter of $T_{d_{1}, d_{2}, \ldots, d_{k}}$ is equal to $c^{k}\diam{T_{A, \mathcal{D}}}$ for every sequence $(d_{1}, d_{2}, \ldots, d_{k})$. It then follows that $\frac{\diam(J_{\sigma*j})}{\diam(J_{\sigma})} = c^{k}/c^{k-1} = c$ for any $\sigma\in W_{k-1}$. This shows that property (iv) holds. 

To check property (v), let $a$ and $b$ be distinct elements of $D_{k}$. If the intersection $\interior(T_{d_{1}, \ldots, d_{k-1}, a}) \cap \interior(T_{d_{1}, d_{2}, \ldots, d_{k-1}, b})$ is nonempty then, because it is open, it must contain a ball. This ball is contained in $T_{d_{1}, \ldots, d_{k-1}, a} \cap T_{d_{1}, d_{2}, \ldots, d_{k-1}, b}$. Multiplying both sets in the intersection by $A^{k}$ and then shifting both sets by $-\sum\limits_{j=1}^{k-1}A^{j}d_{k-j}$ implies that $(a + T_{A, \mathcal{D}}) \cap (b  + T_{A, \mathcal{D}})$ contains a ball. This contradicts the assumption that $T_{A, \mathcal{D}}$ is a self-affine tile, because then this intersection would have positive Lebesgue measure. 
\end{proof}

Recall that when $\alpha$ has a unique representation in base $(A, D-D)$, the image of the product set under $\pi$ and the intersection $T(\alpha) := T \cap (T+\alpha)$ are equal. We can leverage Theorem~\ref{thm:moranDim} to extend the result concerning the level sets of $\Phi_{A, D}$ to the level sets of $\Psi_{A, D}$. 

\begin{theorem} \label{thm:extDense} 
Suppose $A \in M_{n}(\mathbb{Z})$ is invertible and multiplication by its inverse is a similarity on $\mathbb{R}^{n}$ with respect to Euclidean distance. Suppose $\mathcal{D}$ is a complete residue system mod $A$ and $T_{A, \mathcal{D}}$ is a self-affine tile. If $D\subset\mathcal{D}$ is such that all $(A, D-D)$-representations are unique, then the level sets of $\Psi_{A, D}$ are dense in $F_{A, D}$.
\end{theorem}

\begin{proof}
Since $\alpha$ has a unique $(A, D-D)$-representation $\alpha_{j})_{j=1}^{\infty}$, Lemma~\ref{lem:seqExpress} asserts that $T(\alpha) := T \cap (T + \alpha)$ is equal to $\pi_{A, \mathcal{D}}\left(\prod\limits_{j=1}^{\infty}(D\cap(D+\alpha_{j}))\right)$. By Lemma~\ref{lem:moranCheck}, $T(\alpha)$ is a Moran set.  By Theorem~\ref{thm:moranDim}, we have $\dim_{H}T(\alpha) = \liminf\limits_{k\rightarrow\infty} s_{k}$ where $s_{k}$ is the solution to $\sum\limits_{\sigma\in W_{k}}c_{\sigma}^{s_{k}} = 1$. Explicitly, this is the equation $\left(\prod\limits_{j=1}^{k}|D\cap(D+\alpha_{j})|\right)|\det(A)|^{-ks_{k}/n} = 1$. Solving this equation for $s_{k}$ reveals that $s_{k} = G_{k}$ from Theorem~\ref{thm:unboundedDigitFormula}. Therefore $\dim_{H}T(\alpha) = \underline{\dim_{B}}T(\alpha)$. Recall that $s := \dim_{B}T_{A, D} = \dim_{H}T_{A, D}$ since $T_{A, D}$ is self-similar. For all $\lambda \in [0, s]$, the preimage $\Phi_{A, D}^{-1}(\lambda)$ is a subset of $\Psi_{A, D}^{-1}(\lambda)$. The conclusion now holds by Corollary~\ref{cor:boxDense}. 
\end{proof}

\section{The Dimension of a Level Set}\label{sec:levelSetDim} 

Throughout this section, let $A$ be an invertible $n$ by $n$ matrix with integer entries with the property that multiplication by its inverse is a similarity on $\mathbb{R}^{n}$ with respect to Euclidean distance. We do not provide any theorems or proofs in this section. Instead, we develop a conjecture for the value of the Hausdorff dimension of the level sets of $\Phi_{A, \{0, d\}}$ where $d$ is an element of $\mathbb{Z}^{n}\setminus\{0\}$. 

For $p \in [0, 1]$, we define 
\begin{equation}
X_{p} := \{(\alpha_{1}, \alpha_{2}, \ldots) \in \{-d, 0, d\}^{\mathbb{N}}: \lim_{k\rightarrow\infty}\frac{|\{j : \alpha_{j} = 0, i\leq k\}|}{k} = p\}.
\end{equation}
If $D = \{0, d\}$, then, for each $j$, $D\cap(D+\alpha_{j})$ is one of three sets: $\{0\}, \{d\}$, and $\{0, d\}$. It follows that $\sum\limits_{j=1}^{k}\log(|D\cap(D+\alpha_{j})|) = |\{j : \alpha_{j} = 0, i\leq k\}|\log{|D|}$. Therefore
\begin{equation}
\lim_{k\rightarrow\infty}\frac{|\{j : \alpha_{j} = 0, i\leq k\}|\log{|D|}}{k\log(|\det(A)|^{1/n})} = p\frac{\log{|D|}}{\log(|\det(A)|^{1/n})} = p\dim_{B}T_{A, D}.
\end{equation}
By Theorem~\ref{thm:unboundedDigitFormula}, if $\alpha \in X_{p}$, then the box-counting dimension of $T_{A, D} \cap (T_{A, D} + \alpha)$ is $p\dim_{B}T_{A, D}$. It follows that $\Omega_{p} := \pi_{A, \{0, \pm d\}}(X_{p})$ is a subset of the level set $\Phi_{A, D}^{-1}(p\dim T_{A, D})$. 

Given a homogeneous IFS $\mathcal{F} = \{f_{i}\}_{i=1}^{N}$ with contraction coefficient $c$, we can ascribe to it the set of sequences $\{1, 2, \ldots, N\}^{\mathbb{N}}$. Let $\mathcal{L}_{k}$ be the collection of all words of length $k$ that appear as subwords of sequences in $\{1, 2, \ldots, N\}^{\mathbb{N}}$. We can capture the exponential rate at which the cardinality of $\mathcal{L}_{k}$ increases with $k$ as the limit of $\lim\limits_{k\rightarrow\infty}\log(|\mathcal{L}_{k}|)/\log(k)$. Since $|\mathcal{L}_{k}| = N^{k}$ for each $k$, we obtain $\log(N)$. The ratio $\log(N)$ divided by $-\log(c)$ is the similarity dimension of $\mathcal{F}$, which could potentially be the Hausdorff dimension of its attractor. If we look at Theorem~\ref{thm:unboundedDigitFormula} through this lens, we see that the expression for the lower/upper box-counting dimension concerns this very same ratio where the underlying set of sequences is $\prod\limits_{j=1}^{\infty}|D\cap(D+\alpha_{j})|$. This theme is present for multiplicative invariant sets in both \cite{F67} and \cite{GMR24}. 

In our context, a word in $\mathcal{L}_{k}$ can be associated with the cylinder set of sequences which specifies the first $k$ entries of its elements to be that word. This is then associated with the $k$-tile equal to the image of said cylinder set under the map $\pi_{A, D}$. The set $\Omega_{p}$ is not necessarily generated by an IFS. So instead of using the number of maps to inform us of the size of ``$\mathcal{L}_{k}$", we copy another theme from this Chapter: count the number of $k$-tiles which intersect $\Omega_{p}$. Now, if we proceed with all the $k$'s being equal, we will not obtain anything interesting. This because the defining property of $X_{p}$ only considers the tail of sequences in $\{-d, 0, d\}^{\mathbb{N}}$. A sequence in $X_{p}$ can exhibit any possible behaviour for any finite number of indices. Therefore the number of tiles (resp. cylinders) intersecting $\Omega_{p}$ (resp. $X_{p}$) is always the maximal amount. We need to do something different. Suppose we do not use all the cylinders which specify digits up to some index $k$. Instead we limit ourselves to cylinder sets specifying $\alpha_{j}$ for $j=1, \ldots, k$ for which the proportion of $\alpha_{j}=0$ is equal to $q\in(p-\varepsilon, p+\varepsilon)$. The set $X_{p}$ is covered by the infinite collection of cylinder sets with this property because the proportion of zeros among $(\alpha_{1}, \ldots, \alpha_{k})$ tends to $p$ as $k$ tends to infinity. We proceed by approximating the number of cylinder sets of length $k$ with this property with the aim of finding the rate at which it increases as $k$ increases. 

Approximately, there are $2^{k(1-p)}{k \choose kp}$ elements of the cover which specify $k$ digits since $kp$ of the parameters are zero and then the remaining $k-kp$ parameters are $\pm d$. We can interpret this as a ratio of images of Gamma functions if $kp$ is not an integer. We are interested in the behaviour of this quantity as $k$ tends to infinity. Applying Stirling's approximation yields

\begin{align}
2^{k(1-p)}{k \choose kp} &= 2^{k(1-p)}\frac{\Gamma(k+1)}{\Gamma(kp+1)  \Gamma(k(1-p)+1)} \\
&\approx 2^{k(1-p)}\frac{\sqrt{2\pi k}(k/e)^{k}}{(\sqrt{2\pi kp}(kp/e)^{kp})(\sqrt{2\pi k(1-p)} (k(1-p)/e)^{k(1-p)})} \\
&= 2^{k(1-p)}\frac{1}{\sqrt{2\pi kp(1-p)}p^{kp}(1-p)^{k(1-p)}}. \label{eq:stirling} 
\end{align}

Our next step is to take the natural logarithm of the quantity in (\ref{eq:stirling}) and divide by $k$. Using logarithm rules, we obtain
\begin{align}
&\frac{\log(2^{k(1-p)}(\sqrt{2\pi kp(1-p)}p^{kp}(1-p)^{k(1-p)})^{-1})}{k} \\
&= \log(2^{1-p}p^{-p}(1-p)^{-(1-p)}) - \frac{\log(\sqrt{2\pi p(1-p)})}{k}. 
\end{align} 

This quantity tends to $\log(2^{1-p}p^{-p}(1-p)^{-(1-p)})$ as $k$ tends to infinity. Given the sources of our inspiration, it is suggested that the Hausdorff dimension of $\Omega_{p}$ is the division of this quantity by $\log(|\det(A)|^{1/n})$, 

\begin{conjecture}\label{conj:levelDim} 
Suppose $A\in M_{n}(\mathbb{Z})$ is invertible and that multiplication by $A^{-1}$ is a similarity on $\mathbb{R}^{n}$ with respect to Euclidean distance. For any $d\in\mathbb{Z}^{n}\setminus\{0\}$ and $p\in (0, 1)$, 
\begin{align}
&\dim_{H}[\Phi_{A, \{0, d\}}^{-1}(p\dim_{B}T_{A, \{0, d\}})] = \\ &[(1-p)(1-\log_{2}(1-p)) -p\log_{2}p]\dim_{H} T_{A, \{0, d\}}
\end{align}
\end{conjecture}

\begin{remark}
We admittedly only became aware of Li and Xiao's treatment of the problem for the middle third Cantor set in \cite{LX99} \emph{after} putting together this argument. Their approach using multifractals may show that Conjecture~\ref{conj:levelDim} holds. 
\end{remark}

Notably, this conjecture is similar to an established result in \cite{HK23} regarding middle-$(1-2\lambda)$ Cantor sets. It may suggest some kind of duality. For $\lambda\in(0, 1/2)$, let $C_{\lambda}$ denote the attractor of the IFS $\{g_{0}(x) = \lambda x, g_{1}(x) = \lambda x + (1-\lambda)\}$. For a fixed translation $t\in[-1, 1]\setminus\{0\}$, let $\Lambda = \{\lambda\in(0, 1/3]: C_{\lambda} \cap (C_{\lambda} + t) \neq \emptyset\}$. 

\begin{theorem}
The set $\{\lambda\in\Lambda(t): \dim_{H}(C_{\lambda} \cap (C_{\lambda} + t)) = \dim_{P}(C_{\lambda} \cap (C_{\lambda} + t)) = \beta\dim_{H}C_{\lambda}\}$ has Hausdorff and packing dimension
\begin{equation}
(-\beta\log(\beta) - (1-\beta)\log((1-\beta)/2))/\log(3). 
\end{equation}
\end{theorem}

We remark that we could extend our conjecture to functions involving the Hausdorff dimension on the basis that the intersection is a Moran set. 

\section{Intersections of Multiple Translates}\label{sec:multiSec} 
It is possible to extend our results about a single intersection, to any finite intersection. More precisely, suppose $T:=T_{A, D}$ is the self-affine set generated by an expanding matrix $A$ and finite $D\subset\mathbb{R}^{n}$ and suppose $\vec{\alpha}$ denotes a vector $(\alpha^{(1)}, \alpha^{(2)}, \ldots, \alpha^{(m)})$ where each component is an element of $T_{A, D-D}$. Let $T(\vec{\alpha})$ denote $T \cap (T+\alpha^{(1)}) \cap \cdots \cap (T+\alpha^{(m)})$. We first show under the assumptions that the intersection is not trivial and the $(A, D-D)$-representations of the components of $\vec{\alpha}$ are unique, that $T(\vec{\alpha})$ is self-affine if and only if a sequence of sets involving the components of $\vec{\alpha}$ is SEP. 

As before, we want to establish that $T(\vec{\alpha})$ is of the form $\pi_{A, D}\left(\prod\limits_{j=1}^{\infty}D_{j}\right)$ under a condition on $\vec{\alpha}$. 

\begin{lemma}\label{lem:multiSeqExpress} 
Suppose $A\in M_{n}(\mathbb{R})$ is an expanding matrix and $D\subset\mathbb{R}^{n}$ is finite. If each component of $\vec{\alpha} = (\alpha^{(1)}, \ldots \alpha^{(m)})$ has a unique $D-D$ representation, then $T(\vec{\alpha})$ is equal to $\pi_{A, D}\left(\prod\limits_{j=1}^{\infty}I_{j}\right)$ where $I_{j} = \bigcap\limits_{i=1}^{m}(D\cap(D + \alpha_{j}^{(i)}))$ for each $j$. 
\end{lemma}

\begin{proof}
For a vector $\beta\in\mathbb{R}^{n}$, we use the notation $T(\beta) := T_{A, D} \cap (T_{A, D} + \beta)$. For each $i$, $\pi_{A, D}\left(\prod\limits_{j=1}^{\infty}I_{j}\right)$ is a subset of $\pi_{A, D}\left(\prod\limits_{j=1}^{\infty}(D\cap(D+\alpha_{j}^{(i)}))\right)$. By Lemma~\ref{lem:seqExpress}, $\pi_{A, D}\left(\prod\limits_{j=1}^{\infty}(D\cap(D+\alpha_{j}^{(i)}))\right)$ is equal to $T(\alpha^{(i)})$ for each $i$. Therefore $\pi_{A, D}\left(\prod\limits_{j=1}^{\infty}I_{j}\right)$ is subset of $T(\alpha^{(i)})$ for each $i$. In other words, the image of the product of $I_{j}$ is contained in $T(\vec{\alpha})$. Similarly $T(\vec{\alpha}) \subset T(\alpha^{(i)})$ for each $i$. Therefore $T(\vec{\alpha}) \subset \pi_{A, D}\left(\prod\limits_{j=1}^{\infty}(D\cap(D+\alpha_{j}^{(i)}))\right)$ for each $i$. The intersection of $\pi_{A, D}\left(\prod\limits_{j=1}^{\infty}(D\cap(D+\alpha_{j}^{(i)}))\right)$ is equal to $\pi_{A, D}\left(\prod\limits_{j=1}^{\infty}I_{j}\right)$ because the uniqueness of even one of the $(A, D-D)$-representations of an $\alpha^{(i)}$ is sufficient for elements of the set to have unique $(A, D)$-representations. 
\end{proof}

Logically, there are no relations between the $\alpha_{j}^{(i)}$ that need to be satisfied for Lemma~\ref{lem:multiSeqExpress} to hold. There are a relationships that are required for $T(\alpha_{1}, \ldots, \alpha_{m})$ to be nonempty in the first place. While $\alpha^{(i)} - \alpha^{(j)} \in T_{A, D-D}$ is necessary for pairwise intersections, it is insufficient for the total intersection to be non-empty when there are three or more translates (the original set $T$ can be viewed as the translation by $0$). Instead for any sequence of the parameters $i_{1}, i_{2}, \ldots, i_{m}$ we require that
\begin{equation}
\begin{split}
\alpha^{(i_{1})} &\in T-T ,\\
\alpha^{(i_{2})} &\in T(\alpha^{(i_{1})}) - T, \\
&\vdots \\
\alpha^{(i_{m})} &\in  T(\alpha^{(i_{1})}, \alpha^{(i_{2})}, \ldots, \alpha^{(i_{m-1})}) - T.\\
\end{split}
\end{equation}
where $T$ denotes $T_{A, D}$. 

The ability to express the intersection as the image of a sequence space allows us to extend the results shown for $T(\alpha)$ to $T(\vec{\alpha})$. 

\begin{theorem}\label{thm:multiAttractSEP} 
Suppose an expanding matrix $A\in M_{n}(\mathbb{R})$ and finite set $D\subset\mathbb{Z}^{n}$ are chosen $(A, D-D)$-representations are unique. Suppose further that $\vec{\alpha} = (\alpha^{(1)}, \alpha^{(2)}, \ldots, \alpha^{(m)}) \in T_{A, D-D}^{m}$ is chosen such that $T(\vec{\alpha}) := (\bigcap\limits_{i=1}^{m}(T+\alpha^{(i)}))\cap T$ is non-empty. 

The set $T(\vec{\alpha})$ is the attractor of an IFS of the form $\{A^{-p}x + v_{i}\}_{i=1}^{N}$, where $v_{i}$ is an element of $\mathbb{R}^{n}$ for $i=1, 2,\ldots, N$ and $p$ is a positive integer, if and only if the sequence $(I_{j}- \beta_{j})_{j=1}^{\infty}$ is SEP for some sequence of $(\beta_{j})_{j=1}^{\infty} \in D^{\mathbb{N}}$ where $I_{j} = \bigcap\limits_{i=1}^{m}(D\cap(D + \alpha_{j}^{(m)}))$. 
\end{theorem}

\begin{proof}
By Lemma~\ref{lem:multiSeqExpress}, the intersection $T(\vec{\alpha})$ has the form $\pi_{A, D}\left(\prod\limits_{j=1}^{\infty}D_{j}\right)$. The forward direction is a consequence of Lemma~\ref{lem:sepGen}. The other direction holds by Lemma~\ref{lem:selfSimGen}. 
\end{proof}

We now restrict ourselves, as we have for the majority of this chapter, to when $A\in M_{n}(\mathbb{Z})$ is invertible and multiplication by its inverse is a similarity with respect to Euclidean distance. For such a matrix $A$, $D$ contained in a complete residue system mod $A$, and positive integer $m$, define $F_{A, D}^{(m)}$ to be the vectors $\vec{\alpha} = (\alpha^{(1)}, \ldots, \alpha^{(m)})$ such that $T(\vec{\alpha})$ is nonempty. When $m=1$, the set $F_{A, D}^{(1)}$ is $T_{A, D-D}$. For a positive integer $m$, we define the function 
\begin{equation}
 \begin{split}
    &\Theta_{A, D}^{(m)} : \{\vec{\alpha} \in F_{A, D}^{(m)} :\dim_{B}T(\vec{\alpha})\;\;\text{exists}\} \rightarrow [0, \dim_{B}T_{A, D}], \\
    &\Theta_{A, D}^{(m)}(\alpha) = \dim_{B}T(\vec{\alpha}).
  \end{split}
\end{equation}

\begin{theorem}\label{thm:multiDense} 
Suppose $A\in M_{n}(\mathbb{Z})$ is invertible and multiplication by its inverse is a similarity with respect to Euclidean distance. Suppose $D$ is a subset of a complete residue system mod $A$ such that all $(A, D-D)$ expansions are unique. The level sets of $\Theta_{A, D}^{(m)}$ are dense in $F_{A, D}^{(m)}$. 
\end{theorem}

\begin{proof}
First suppose that $\lambda$ is an element of $(0, 1)$ and that $\varepsilon$ is a real number greater than zero. Suppose that $\vec{\alpha}$ is an element of $F_{A, D}^{(m)}$. There exists a positive integer $p$, such that $\norm{\sum\limits_{j=p+1}^{\infty}A^{-j}u_{j}} < \varepsilon/\sqrt{m}$ for any vectors $u_{j}$ in the difference of differences sets $(D-D)-(D-D)$. Therefore any $\vec{\beta}$ with the property its $i$th component has the $(A, D-D)$-representation of the form $(\alpha_{1}^{(i)}, \ldots \alpha_{p}^{(i)}, \beta_{p+1}^{(i)}, \beta_{p+2}^{(i)}, \ldots)$ for each $i$, is within $\varepsilon$ distance of $\vec{\alpha}$. Choose $d$ to be an element of maximal norm in $D-D$ and choose $(h_{j})_{j=1}^{\infty}$ to be a sequence of integers for which $h_{j} \leq j\lambda < h_{j} + 1$ . For each $i = 1, 2, \ldots, m$, choose $\beta_{j}^{(i)} = v$ if $h_{j} = h_{j-1}$ and $0$ otherwise. When $\lambda = 0$ or $1$ choose $\beta_{j} = v$ for all $j$ or $\beta_{j} = 0$ for all $j$ respectively. 

By Lemma~\ref{lem:multiSeqExpress}, the intersection $T(\vec{\beta})$ is of the form $\pi_{A, D}\left(\prod\limits_{j=1}^{\infty}D_{j}\right)$. We can apply Theorem~\ref{thm:unboundedDigitFormula} to directly compute the box-counting dimension. The remaining steps are the same as in the proof of Corollary~\ref{cor:boxDense} where we treated a single intersection. We omit the details. 
\end{proof}

\begin{remark} 
Let the notation $(x_{1}, \ldots, x_{t}) * (y_{1}, \ldots, y_{s})$ denote the $(t+s)$-length tuple $(x_{1}, \ldots, x_{t}, y_{1}, \ldots, y_{s})$. For any fixed $\vec{\beta} \in F_{A, D}^{(k)}$, we can define $F_{A, D}^{(m, \vec{\beta})}$ to be the set of $\vec{\gamma}\in F_{A, D}^{(m)}$ for which $T(\vec{\beta} * \vec{\gamma})$ is nonempty whenever $m > k$. Then, let $\Theta_{A, D}^{(m, \vec{\beta})}:  \{\vec{\gamma} \in F_{B, D-D}^{(m-k)} :\dim_{B}T((\vec{\beta} * \vec{\gamma})\;\;\text{exists}\} \rightarrow [0, \dim_{B}T(\vec{\beta})]$ be the slice of $\Theta_{A, D}^{(m)}$ given by $\Theta_{B, D}^{(m, \vec{\beta})}(\vec{\gamma}) = \dim_{B}T(\vec{\beta} * \vec{\gamma})$. The argument used in the proof of Theorem~\ref{thm:multiDense} can be used to show that the level sets of the $m$th slice are dense in $F_{A, D}^{(m-k)}$. 
\end{remark}

The claim still holds if $\Theta_{A, D}$ is swapped for the function (and its slices) that maps $\vec{\alpha}$ to the Hausdorff dimension of $T(\vec{\alpha})$ when $T_{A, D}$ is a self-affine tile. Since Lemma~\ref{lem:multiSeqExpress} implies that $T(\vec{\alpha})$ has the form $\pi_{A, D}\left(\prod\limits_{j=1}^{\infty}D_{j}\right)$ for $D_{j}\in D$ for each $j$, it is a Moran set. The conclusion now follows from the argument used in Theorem~\ref{thm:extDense} for the case of a single intersection. 

We end this section by observing that a multiple intersection can be captured by a single intersection in special cases. The middle third Cantor set is equal to $T_{3, \{0, 2\}}$ in our notation. Let $\alpha$ and $\beta$ be elements of $T_{B, D-D}$. If $(\alpha_{j})_{j=1}^{\infty}$ is the unique representation of $\alpha$, we can express $T(\alpha)$ as the set $\sum\limits_{j=1}^{\infty}x_{j}3^{-j}$ where $x_{j} = 2$ if $\alpha_{j} = 2$, $0$ if $\alpha_{j} = -2$, and either $0$ or $2$ when $\alpha_{j} = 0$. Suppose that both $\alpha$ and $\beta$ have unique expansions. In order for $T(\alpha, \beta)$ to be nonempty, it must be that for each $j$ either $\alpha_{j} = \beta_{j}$ or one of $\alpha_{j}$ and $\beta_{j}$ is $0$. 

Under the assumption of nonemptiness, $D \cap (D+\alpha_{j}) \cap (D+\beta_{j}) = D \cap (D + \gamma_{j})$ where $\gamma_{j} = \alpha_{j}$ if $\alpha_{j} \neq 0$ and $\alpha_{j} = \beta_{j}$ otherwise. Therefore $T(\alpha) \cap T(\beta) = \pi_{A, D-D}\left(\prod\limits_{j=1}^{\infty}(D\cap(D+\gamma_{j}))\right)$. If $(\gamma_{j})_{j=1}^{\infty}$ is unique then $T(\alpha) \cap T(\beta) = T(\gamma)$ where $\gamma = \pi_{A, D-D}(\gamma_{j})_{j=1}^{\infty}$. The number $\gamma$ may not have a unique representation. For example, choose $\alpha = \pi_{3, \{0, 1, 2\}}(0,\overline{0,2})$ and $\beta = \pi_{3, \{0, 1, 2\}}(0,\overline{2,0})$ in base three. The approach described above yields the sequence $\gamma = \pi_{3, \{0, 1, 2\}}(0,\overline{2})$. However, we also have $\gamma = \pi_{3, \{0, 1, 2\}}(2,\overline{-2})$. It follows that $T(\alpha)\cap T(\beta) = \{\pi_{3, \{0, 1, 2\}}(\overline{2}), \pi_{3, \{0, 1, 2\}}(0,\overline{2}\}) = \{1, 1/3\}$, yet $T(\gamma)$ also contains the element $\pi_{3, \{0, 1, 2\}}(2,\overline{0}) = 2/3$. 

The collapse of a double intersection into a single intersection cannot always be done. The reason it was possible to reduce the intersection of $T_{3, \{0, 2\}}$ with two of its translates to an intersection with only one of its translates is a consequence of $\{0, 2\}$ containing only two elements. This property is independent of the base and is not true in general for sets of integers containing more than two elements. Consider $D = \{0, 3, 6, 9\}$. Observe that $D \cap (D - 3) \cap (D + 6) = \{6\}$. There is no element $m\in\{0, \pm3, \pm6, \pm9\}$ for which $D \cap (D + m) = \{6\}$.

\section{When Translations have Multiple Representations}\label{sec:nonUniqueReps} 

Given an expanding matrix $A$ and $D\subset \mathbb{R}^{n}$, it is possible for $\alpha \in T_{A, D-D}$ to have multiple representations. Given such an $\alpha$, the claim $T \cap (T+\alpha) = \pi\left(\prod\limits_{j=1}^{\infty}D\cap(D+\alpha_{j})\right)$ may not hold. What we can deduce is that $T \cap (T + \alpha)$ is a union of sets of the form $\pi_{A, D}\left(\prod\limits_{j=1}^{\infty}D_{j}\right)$ where $D_{j}\subset D$ for each $j$. 

\begin{lemma}\label{lem:unionExpress} 
Suppose $A\in M_{n}(\mathbb{R})$ is an expanding matrix and $D\subset\mathbb{R}^{n}$ is finite. If $\alpha\in T_{A, D-D}$, then 
\begin{equation}
T_{A, D} + (T_{A, D} + \alpha) = \bigcup\limits_{\pi_{A, D-D}^{-1}(\alpha)}\pi_{A, D}\left(\prod\limits_{j=1}^{\infty}(D\cap(D+\alpha_{j}))\right). 
\end{equation}
\end{lemma}

\begin{proof}
Let $x \in T(\alpha) := T_{A, D} + (T_{A, D} + \alpha)$. By definition, $x = y + \alpha$ where $x = \pi_{A, D}(x_{j})_{j=1}^{\infty}$ and $y=\pi_{A, D}(y_{j})_{j=1}^{\infty}$ for some choice of $x_{j}, y_{j} \in D$ for each $j$. It follows that $\alpha = \pi_{A, D-D}(x_{j}-y_{j})_{j=1}^{\infty}$ is a valid $(A, D-D)$-representation of $\alpha$. Therefore, for each $j$, $x_{j}$ is an element of $D\cap(D+(x_{j} - y_{j})$. Therefore $x \in \pi\left(\prod\limits_{j=1}^{\infty}D\cap(D+\alpha_{j})\right)$ for some representation $(\alpha_{j})_{j=1}^{\infty}$ of $\alpha$. Conversely, suppose $x = \pi(x_{j})$ is such that $x_{j} \in D\cap(D+\alpha_{j})$ for some choice of representation of $\alpha$. Let $y_{j}$ denote $x_{j} - \alpha_{j}$ for each $j$. By assumption, this defines a sequence of digits in $D$.  Since $x = \pi_{A, D}(x_{j}) = \pi_{A. D}(y_{j} + \alpha_{j}) = \pi_{A, D}(y_{j}) + \pi_{A, D}(\alpha_{j}) = y + \alpha$, we obtain $x\in T(\alpha)$. 
\end{proof}

If the union is finite, then it follows that the Hausdorff dimension of $T_{A, D}\cap (T_{A, D} + \alpha)$ is equal to the maximum of the Hausdorff dimension of the components. We are not guaranteed that the union is countable, let alone finite. The following example is related to decimal expansions. 

\begin{example}\label{ex:uncountAlphaRe} 
The $(10, \{0, 1, 2, 5\})$-representation of $0.15$ is $(1, 5, \overline{0})$. The set of $(10, \{0, 1, 2, \pm 5)$-representations of $0.15$ include both the one just stated and $(2, -5, \overline{0})$. Suppose $\alpha$ has $(10, \{0, 1, 2, \pm 5)$-representation $(1, 5, \overline{1, 5})$. It follows that every representation of the form $(W_{j})_{j=1}^{\infty}$ where $W_{j}$ is either $(1, 5)$ or $(2, -5)$ for each $j$, is an equivalent representation of $\alpha$. 
\end{example}

Consequently, a version of Corollary~\ref{cor:boxDense} when $(A, D-D)$-representations are not unique is far less descriptive. 

\begin{theorem}\label{thm:fulldimDense} 
Suppose $A\in M_{n}(\mathbb{Z})$ is invertible and multiplication by its inverse is a similarity with respect to Euclidean distance. Suppose $D$ is a subset of a complete residue system mod $A$. The set of $\alpha \in T_{A, D-D}$ such that the box-counting dimension of $T_{A, D}\cap (T_{A, D} + \alpha)$ exists and is equal to $\dim_{B}T_{A, D}$ is dense in $T_{A, D-D}$. 
\end{theorem}

\begin{proof} 
Let $\varepsilon >0$ and let $(\alpha_{j})_{j=1}^{\infty}$ denote one of the $(A, D-D)$-representations of $\alpha$. There exists $m$ such that $\norm{\sum\limits_{j=m+1}^{\infty}A^{-j}\alpha_{j}} < \varepsilon$. Let $(\beta_{j})_{j=1}^{\infty}$ be a sequence of elements in $D-D$ such that $\beta_{j} = \alpha_{j}$ for $j < m$. If $\beta_{j} = 0$ for $j \geq m$, then $\dim_{B}\pi_{A, D}\left(\prod\limits_{j=1}^{\infty}(D\cap(D+\beta_{j}))\right) = \dim_{B}T_{A, D}$. The box-counting dimension of $T_{A, D}$ exists without the need for any kind of separation condition on $D$ because it is self-similar. 

By Lemma~\ref{lem:unionExpress}, the image $\pi_{A, D}\left(\prod\limits_{j=1}^{\infty}(D\cap(D+\beta_{j}))\right)$  is a subset of $T(\beta) := T_{A, D} \cap (T_{A, D} + \beta)$ where $\beta$ has $(A, D-D)$-representation $(\beta_{j})_{j=1}^{\infty}$. By the monotonicity of the box-counting dimension, $\dim_{B}T(\beta)$ is bounded below by $\dim_{B}\pi_{A, D}\left(\prod\limits_{j=1}^{\infty}(D\cap(D+\beta_{j}))\right) = \dim_{B}T_{A, D}$. On the other hand, $T(\beta)$ is a subset of $T_{A, D}$ and thus $\dim_{B}T_{A, D}$ is the largest $\dim_{B}T(\beta)$ can be. 
\end{proof}

\begin{remark}The assumption that there exists at least one $\alpha$ with a unique $(A, D-D)$-representations can places a restriction on the choices $D\subset\mathcal{D}$. Allowing for the non-uniqueness of $(A, D-D)$ representations of $\alpha$ includes cases where no elements of $T_{A, D-D}$ have unique $(A, D)$ representations. This includes cases where $T_{A, D}$ has non-empty interior. The Hausdorff dimension of $T_{A, D} \cap (T_{A, D} + \alpha)$ is $n$ unless $T + \alpha$ and $T$ only intersect at their boundaries. The claim of Theorem~\ref{thm:fulldimDense} is then immediate because the entire interior of $T_{A, D-D}$ is mapped to $n$ under $\Phi_{A, D}$. Some restriction on $D$ is required to allow for more precise claims about the behaviour of $\Phi_{A, D}$. 
\end{remark}

Example~\ref{ex:uncountAlphaRe} may inject some pessimism in the goal of understanding the structure of intersections when $\alpha$ does not have a unique $(A, D-D)$-representation. The following example demonstrates that the family of lower box-counting dimensions of $\pi_{A, D}\left(\prod\limits_{j=1}^{\infty}(D\cap(D+\alpha_{j}))\right)$, where $(\alpha_{j})_{j=1}^{\infty}$ is some $(A, D-D)$-representation of $\alpha$, can be understood. 

\begin{example}\label{ex:uncountAlphaIm} 
The phenomenon demonstrated in Example~\ref{ex:uncountAlphaRe} is present for other pairs $(A, D)$. Observe that the Gaussian integer $-3+i$ is the root of the polynomial $x^{2} + 6x + 10$. Therefore $2(-3+i)^{-3}$ has the $(-3+i, \{0, \pm 1, \ldots, \pm n^{2})$-representation $(-1, -6, -8, \overline{0})$. The multiplication of a complex number $z = x + iy$ by $-3+i$ corresponds to the multiplication of $[x, y]^{t}\in\mathbb{R}^{2}$ by the matrix $A = \begin{bmatrix} -3 & -1 \\ 1 & -3\end{bmatrix}$. Suppose $D = \{0, e_{1}, 6e_{1}, 8e_{1}\}$ where $e_{1}$ is standard basis vector $[1, 0]^{t}$ and $\alpha$ has $(A, D-D)$-representation $(\overline{0, 0, 2e_{1}})$. By Lemma~\ref{lem:unionExpress}, $T(\alpha) = T_{A, D} \cap (T_{A, D} + \alpha)$ contains the union of $\pi_{A, D}\left(\prod\limits_{j=1}^{\infty}(D \cap (D+\alpha_{j}))\right)$ over the sequences $(\alpha_{j})$ of the form $(W_{j})_{j=1}^{\infty}$ where $W_{j}$ is either $(0, 0, 2e_{1})$ or $(-e_{1}, -6e_{1}, -8e_{1})$. 

Observe that the segment $(0, 0, 2e_{1})$ yields $(D \cap (D+0), D \cap (D+0), D \cap (D+2e_{1})) = (D, D, \{8\})$. The block $(-e_{1}, -6e_{1}, -e_{1})$ yields the string of singletons $(\{0\}, \{0\}, \{0\})$. We omit the details, but it follows from Theorem~\ref{thm:rules} that the only equivalent $(A, D)$-representations are those contained in the set of representations whose tails are $(\overline{1, 6})$. It follows that the sets $\pi_{A, D}\left(\prod\limits_{j=1}^{\infty}(D \cap (D+\alpha_{j}))\right)$ are disjoint. In the extreme case when $W_{j} = (0, 0, 2)$ for all $j$, then by Theorem~\ref{thm:unboundedDigitFormula}, the dimension of the corresponding subset of $T(\alpha)$ is $(2/3)\log(4)/\log(\sqrt{10})$. On the other hand, if $W_{j} = (-e_{1}, -6e_{1}, -8e_{1})$, then the image only contains the origin and the box-counting dimension is consequently zero. All other choices of $(W_{j})$ yield lower box-counting dimensions that are between these two values. Roughly, some fixed proportion of occurrences of $(-e_{1}, -6e_{1}, -8e_{1})$ in the tail of the representation of $\alpha$ decreases the dimension of the corresponding component of the union. The dimension of the component rises toward the maximal value as that proportion decreases. 
\end{example}

We know that the relationship between equivalent decimal expansions is highly restrictive. The ways in which two equivalent representations can differ may be more varied for one choice of $(A, D)$ over another, but as we will see in Chapter~\ref{chp:neighbours}, equivalent $(A, D)$-representatons have a relationship between them (Theorem~\ref{thm:neighSeqEquiv}) common for all choices $(A, D)$. Leveraging that information may give way to interesting theorems when $(A, D-D)$-representations are not unique.

\chapter{Equivalence and Uniqueness of $(A, D)$-representations}\label{chp:neighbours}

Many of the results in Chapter~\ref{chp:highDimSEP} and Chapter~\ref{chp:limFormula} rely on the uniqueness of the $(A, D-D)$-representations. In this chapter we discuss how to ensure that uniqueness. Given an expanding matrix $A$ and finite set $D\subset\mathbb{R}^{n}$, the relationship between equivalent $(A, D)$-representations, let alone $(A, D-D)$, depend intimately on the choice of $A$ and $D$. Let us see some examples of the different patterns that can arise in the tails of equivalent $(A, D)$-representations.

\begin{example} 
The sequences $(1, \overline{0})$ and $(0, \overline{9})$ are equivalent representations of $1/10$. In fact, we know that the only equivalent pairs of representations $(x_{j})_{j=1}^{\infty}$ and $(y_{j})_{j=1}^{\infty}$ such that for some index $J$, $x_{J-1}\neq0$, for all $j > J$, $x_{j} = 0$ and $y_{j} =9$, and $y_{J-1} = x_{J-1}-1$. 
\end{example}

\begin{example}
It can be verified directly that the sequences 
$$(e_{1}, 13e_{1}, \overline{36e_{1}, 0, 49e_{1}}),$$ 
$$(e_{1}, 14e_{1}, \overline{49e_{1}, 36e_{1}, 0}),$$ 
$$(0, 0, \overline{0, 49e_{1}, 36e_{1}}),$$ 
where $e_{1} := [1, 0]^{t}$, are equivalent $(A, D)$-representations when 
\begin{equation}
 (A, D) := \left(\begin{bmatrix}-7&-1\\1&-7\end{bmatrix}, \left\{\begin{bmatrix}0\\0\end{bmatrix}, \begin{bmatrix}1\\0\end{bmatrix},\ldots, \begin{bmatrix}49\\0\end{bmatrix}\right\}\right).
\end{equation}
\end{example}

\begin{example}
In Example~\ref{ex:uncountAlphaIm}, we saw that $(A, D)$-representation $(\overline{0, 0, 2})$ has uncountably many distinct $(A, D-D)$-representations of the form $$(W_{1}, W_{2}, \ldots)$$ where $W_{j}$ is either $(0, 0, 2e_{1})$ or $(-e_{1}, -6e_{1}, -8e_{1})$ for each $j$. 
\end{example} 

Despite these differences, we can express the equivalence between $(A, D)$-representations as a relationship between the set $D$ and the self-affine set $T_{A, D}$. This will allow us to formulate a strategy for choosing particular pairs $(A, D)$ that produce unique $(A, D)$-representations. We introduce the following terminology. 

\begin{definition} 
Let $A\in M_{n}(\mathbb{R})$ be an expanding matrix and $D\subset\mathbb{R}^{n}$ be finite. We call $v\in\mathbb{R}^{n}\setminus\{0\}$ a \textit{neighbour of $T_{A, D}$} if $T_{A, D}\cap(v+T_{A, D})$ is nonempty. 
\end{definition}

\begin{lemma}\label{lem:neighUnique} 
Suppose $A\in M_{n}(\mathbb{R})$ is an expanding matrix and $D\subset\mathbb{R}^{n}$ is finite. The $(A, D)$-representations are unique if and only if $D-D$ does not contain a neighbour of $T_{A, D}$. 
\end{lemma}

\begin{proof}
Suppose that $(x_{j})_{j=1}^{\infty}$ and $(y_{j})_{j=1}^{\infty}$ are equivalent $(A, D)$-representations. By definition, we have 
$\sum\limits_{j=1}^{\infty}A^{-j}x_{j} = \sum\limits_{j=1}^{\infty}A^{-j}y_{j}$. Multiplying by $A$ and subtracting $y_{1}$ from both sides yields $(x_{1} - y_{1}) + \pi_{A, D}(x_{j})_{j=2}^{\infty} = \pi_{A, D}(y_{j})_{j=2}^{\infty}$. This means that $x_{1}-y_{1}\in D-D$ is a neighbour of $T_{A, D}$. Conversely, suppose $e\in D-D$ is a neighbour of $T_{A, D}$. By definition, there exists $d, d^{'}\in D$ and $t_{1}, t_{2}\in T_{A, D}$ such that $(d-d^{'}) + t_{1} = t_{2}$. Multiplying by $A^{-1}$ and adding $d^{'}$ on both sides implies that $(d, t_{1, 1}, t_{1, 2}, \ldots)$ and $(d^{'}, t_{2, 1}, t_{2, 2}, \ldots)$ are equivalent $(A, D)$-representations where $(t_{i, j})_{j=1}^{\infty}$ is the $(A, D)$-representation of $t_{i}$. 
\end{proof}

It follows that $(A, D-D)$-representations are unique if and only if the difference of differences $(D-D)-(D-D)$ does not contain a neighbour of $T_{A, D-D}$. 

Determining if the elements of $D-D$ are neighbours of a given self-affine set generated by an expanding matrix and a finite subset of $\mathbb{R}^{n}$ can be difficult. There exists an algorithm that determines the neighbours contained in $\mathbb{Z}^{n}$ when $A$ has integer entries and $D$ contains a complete residue system mod $A$ \cite{ST03}. Suppose $A$ is an expanding matrix with integer entries and $\mathcal{D}\subset{Z}^{n}$ is a complete residue system modulo $A$. Applying the algorithm yields the neighbours of $T_{A, \mathcal{D}}$ contained in $\mathbb{Z}^{n}$. If we choose $D\subset\mathcal{D}$ such that it excludes the neighbours of $T_{A, \mathcal{D}}$ from $D-D$, then the $(A, D)$-representations are unique by virtue of $T_{A, \mathcal{D}}$ containing $T_{A, D}$. The knowledge of neighbours of $T_{A, \mathcal{D}}$ provides us with a sufficient condition for $(A, D)$-representations to be unique whenever $D\subset\mathcal{D}$. 

The limitation here is that the algorithm treats particular instances of $(A, D)$ and does not formulate the set of neighbours for an entire family of attractors. We demonstrate an alternative derivation of sufficient conditions for the exclusion of neighbours from the set $D-D$ where the pair $(A, D)$ is a member of the family of pairs 
\begin{equation}
\left(A_{n} := \begin{bmatrix}-n & -1 \\ 1 & -n\end{bmatrix}, D\subset \mathcal{D}_{n} := \left\{\begin{bmatrix}0\\0\end{bmatrix}, \begin{bmatrix}1\\0\end{bmatrix}, \begin{bmatrix}2\\0\end{bmatrix}, \ldots, \begin{bmatrix}n^{2}\\0\end{bmatrix}\right\}\right)
\end{equation}
where $n$ is a positive integer greater than $2$. The matrix $A_{n}$ is expanding because its eigenvalues are $-n\pm i$. Our approach does not rely on the fact that $D$ is contained in the complete residue system $\mathcal{D}_{n}$ and the derivation of the sufficient condition for unique $(A_{n}, D)$-representations translates immediately to a sufficient condition for unique $(A_{n}, D-D)$-representations. 

When applicable, it is convenient to perform calculations with complex numbers instead of vectors. We introduce a slight abuse of notation to represent $T_{A_{n}, D}$ as a subset of the complex plane. First, let $\iota:\mathbb{R}^{2}\rightarrow\mathbb{C}$ be the vector space isomorphism and isometry given by $[x, y]^{t} \mapsto x+iy$. For any complex number $b$ with modulus greater than $1$ and finite subset $D\subset\mathbb{C}$, we denote by $T_{b, D}$ the image of $T_{A_{b}, \iota(D)}$ under $\iota$ where $A_{b}$ is the matrix representation of the linear transformation given by $[x, y]^{t} \mapsto \iota^{-1}(b(x+iy))$ with respect to the standard basis on $\mathbb{R}^{2}$. 

For $(b, \mathcal{D}) =(-n+i, \{0, 1, \ldots, n^{2}\})$, the attractor $T_{b, \mathcal{D}}$ is the set
\begin{equation}
\bigg\{\sum_{j=1}^{\infty}d_{j}b^{-j}:d_{j}\in\{0, 1, \ldots, n^{2}\}\bigg\}
\end{equation}
We initiate our study of $T_{-n+i, \{0, 1, \ldots, n^{2}\}}$ in the following section.

\section{A Case Study of {$-n+i$}}\label{sec:caseStudy} 

Let $b := -n+i$ for some positive integer $n$ and let $\mathcal{D} = \{0, 1, \ldots, n^{2}\}$. It is natural to ask why anyone would be interested in $T_{b, \mathcal{D}}$. For one, the set is a generalization of the popularized twin-dragon curve featured in \cite{K81}. Although arguably more interesting, is that the pair $(b, \mathcal{D})$ is, in some sense, canonical. 

Consider that for any positive integer $m > 1$, the phrase ``base $m$" is often unambiguous. In accordance with the standard base-$10$ and binary number systems, we interpret the phrase to refer to the expression of real numbers in the form 
\begin{equation}
d_{\ell}m^{\ell} + d_{\ell-1}m^{\ell-1} +\cdots + d_{0} + \sum_{j=1}^{\infty}d_{-j}m^{-j}
\end{equation}
where for $d_{j}\in\{0, 1, \ldots, m-1\}$ for all $j$. In a similar fashion, I. Katai and J. Szabo defined what they deemed to be the canonical meaning of ``base-$\zeta$" when $\zeta$ is a Gaussian integer in \cite{KS75}. The base-$\zeta$ expansion of a complex number $z$, if it exists, has the form 
\begin{equation}
d_{\ell}\zeta^{\ell} + d_{\ell-1}\zeta^{\ell-1} +\cdots + d_{0} + \sum_{j=1}^{\infty}d_{-j}\zeta^{-j}
\end{equation}
where $d_{j}\in\{0, 1, \ldots, |\zeta|^{2}-1\}$ for each $j$. In the real case when $m$ is a positive integer greater than one, there are two properties that make base-$m$ a sensible number system for applications. The first is that all natural numbers have the form $d_{\ell}\zeta^{\ell} + d_{\ell-1}\zeta^{\ell-1} +\cdots + d_{0}$ where $d_{j}\in\{0, 1, \ldots, |\zeta|^{2}-1\}$. The second is that those expansions are unique. It is shown in \cite{KS75} properties are not satisfied by all pairs $(\zeta, \{0, 1, \ldots, |\zeta|^{2}-1\})$ when examining the base-$b$ expansions of Gaussian integers. The existence and uniqueness of expansions of all Gaussian integers places a restriction on the form of $\zeta$. This is captured in the following result from \cite{KS75}. 

\begin{theorem}[I. Katai, J. Szabo, \cite{KS75}, theorem 1]  
Given a Gaussian integer $b$, every Gaussian integer $g$ can be uniquely written as
\begin{equation}\label{eq:integerRep} 
g = \lambda_{0} + \lambda_{1}b + \ldots + \lambda_{k}b^{k},
\end{equation}
with $\lambda_{j} \in \{0, 1, 2, \ldots, |b|^{2}-1\}$ if and only if $\Re(b) < 0$ and $\Im(b) = \pm 1$. 
\end{theorem}

In other words, the base must be of the form $b = -n\pm i$ where $n$ is a positive integer and the set of digits is $\{0, 1, \ldots, n^{2}\}$. This result can be used to prove the following result about complex radix expansions, also featured in \cite{KS75}. 

\begin{corollary}[I. Katai, J. Szabo, \cite{KS75}, theorem 2] \label{thm:radixExistCanon} 
Suppose $n$ is a positive integer and set $b = -n+i$. The set $\bigcup\limits_{\zeta\in\mathbb{Z}[i]}(T_{-n+i, \{0, 1, \ldots, n^{2}\}} + \zeta)$ is equal to the complex plane. 
\end{corollary}

\begin{remark}
In other words, every complex number $z\in\mathbb{C}$ has the form $\zeta + t$ where $\zeta$ is a Gaussian integer and $t\in T_{-n+i, \{0, 1, \ldots, n^{2}\}}$. When $z$ is of the form $p+iq$ where $p$ and $q$ are rational numbers, there exists an algorithm for finding the $(-n+i, \{0, 1, \ldots, n^{2}\})$-representation of $t$ \cite{G96}. 
\end{remark}

Suppose $n$ is a positive integer greater than one. By Lemma~\ref{lem:neighUnique}, the $(-n+i, D)$-representations are unique if none of the neighbours of $T_{-n+i, \{0, 1, \ldots, n^{2}\}}$ are contained in $D-D$. Since $D$ is a subset of $\{0, 1, \ldots, n^{2}\}$, the set $D-D$ is a subset of the integers. It follows that in order to ensure the uniqueness of  $(-n+i, D)$-representations, we need only actively avoid the integer neighbours of $T_{-n+i, \{0, 1, \ldots, n^{2}\}}$. Likewise, we need only actively avoid the integer neighbours of $T_{-n+i, \{0, \pm 1, \ldots, \pm n^{2}\}}$ to ensure the uniqueness of $(-n+i, D-D)$-representations. To this end, we prove the following lemma.

\begin{lemma}\label{lem:reimbound} 
Fix an integer $n \geq 3$. If $s$ is a neighbour of $T_{-n+i, \{0, 1, \ldots, n^{2}\}}$, then 
\begin{equation}
|\Re(s)-n\Im(s)| < 2.
\end{equation}
Moreover $|\Re(s)-n\Im(s)| < 3/2$ for $n\geq5$. 
\end{lemma}

\begin{proof}
Let $s$ be a neighbour of $T = T_{-n+i, \{0, 1, \ldots, n^{2}\}}$. For convenience, we set $\alpha := \Re(s)$ and $\beta := \Im(s)$. Likewise, let $\pi := \pi_{-n+i, \{0, 1, \ldots, n^{2}\}}$. By definition there exist sequences $(d_{j})_{j=1}^{\infty}$ and $(d_{j}^{'})_{j=1}^{\infty}$ with entries in $\{0, 1, \ldots, n^{2}\}$ such that 
\begin{equation}
s + \pi(d_{j})_{j=1}^{\infty} = \pi(d_{j}^{'})_{j=1}^{\infty}.
\end{equation}
Isolating for $s$ yields 
\begin{equation}
s = \pi(\delta_{j})_{j=1}^{\infty}
\end{equation}
where $\delta_{j}\in\{0, \pm 1, \ldots, \pm n^{2}\}$ for each $j$. 

We wish to estimate the difference between the real part of $s$ and $n$ times its imaginary part. We explicitly compute the first few terms of $s$ in terms of $n$. Observe that 
\begin{equation}
s = \frac{\delta_{1}}{b} +  \frac{\delta_{2}}{b^{2}} + \frac{\delta_{3}}{b^{3}} + \frac{\delta_{4}}{b^{4}} + \epsilon
\end{equation}
where $\epsilon$ denotes the tail $\sum\limits_{j=5}^{\infty}\delta_{j}b^{-j}$ and $b=-n+i$. 

Explicit computation yields
\begin{equation} \label{eq:neighbourreal} 
\alpha = \frac{-n}{n^{2}+1}\delta_{1} +  \frac{n^{2}-1}{(n^{2}+1)^{2}}\delta_{2} + \frac{-n^{3}+3n}{(n^{2}+1)^{3}}\delta_{3} + \frac{n^{4}-6n^{2}+1}{(n^{2}+1)^{4}}\delta_{4} + \Re(\epsilon) \\ 
\end{equation}
and 
\begin{equation} \label{eq:neighbourimag} 
\beta = \frac{-1}{n^{2}+1}\delta_{1} + \frac{2n}{(n^{2}+1)^{2}}\delta_{2} + \frac{-3n^{2}+1}{(n^{2}+1)^{3}}\delta_{3} + \frac{4n^{3}-4n}{(n^{2}+1)^{4}}\delta_{4} + \Im(\epsilon). \\
\end{equation}
Subtracting $n$ times (\ref{eq:neighbourimag}) from (\ref{eq:neighbourreal}) yields
\begin{equation} \label{eq:realimagdiff} 
\alpha - n\beta = \frac{-1}{n^{2}+1}\delta_{2} + \frac{2n}{(n^{2}+1)^{2}}\delta_{3} + \frac{-3n^{4}-2n^{2}+1}{(n^{2}+1)^{4}}\delta_{4} + \Re(\epsilon) - n\Im(\epsilon). 
\end{equation}
Recall that $\delta_{j}$ range from $-n^{2}$ to $n^{2}$. In order to bound $|\alpha - n\beta|$ with an expression that is only in terms of $n$, we maximize the sum of the first three terms of (\ref{eq:realimagdiff}) by choosing $\delta_{2} = \delta_{4} = -n^{2}$ and $\delta_{3} = n^{2}$ and estimate the absolute value of the last two terms using the bound $|\Re(\epsilon) - n\Im(\epsilon)| \leq (n+1)\sum\limits_{j=5}^{\infty}n^{2}|b|^{-j} = (n+1)(|b|+1)/|b|^{4}$. This results in the inequality
\begin{equation} \label{eq:boundnonly} 
|\alpha - n\beta| \leq \frac{n^{2}}{n^{2}+1} + \frac{2n^{3}}{(n^{2}+1)^{2}} + \frac{3n^{6}+2n^{4}-n^{2}}{(n^{2}+1)^{4}} + \frac{(n+1)(\sqrt{n^{2}+1}+1)}{(n^{2}+1)^{2}}
\end{equation}
Direct computation with $n = 3, 4, 5, 6$ yields bounds less than $1.85, 1.63, 1.5,$ and $1.41$.
For $n \geq 7$, observe that we can respectively bound each term of (\ref{eq:boundnonly}) by the following sequences. 
\begin{align}
\frac{n^{2}}{n^{2}+1} &< 1, \label{eq:firstbound}\\ 
\frac{2n^{3}}{(n^{2}+1)^{2}} &< \frac{2}{n}, \\
\frac{3n^{6}+2n^{4}-n^{2}}{(n^{2}+1)^{4}} &< \frac{5}{n^{2}}, \\
 \frac{(n+1)(\sqrt{n^{2}+1}+1)}{(n^{2}+1)^{2}} &< \frac{5}{n^{2}}. \label{eq:lastbound}
\end{align}
Consider the sum of all the sequences on the right hand side of a ``$<$" sign from (\ref{eq:firstbound}) to (\ref{eq:lastbound}). The sum is a strictly decreasing sequence and at $n = 7$ is less than $1.49$. This completes the proof. 
\end{proof}

\begin{remark}
When $n=1$ and $n=2$, (\ref{eq:boundnonly}) is insufficient. Rather than complicate the statement of Lemma~\ref{lem:reimbound}, we use the neighbour finding algorithm from \cite{ST03} to determine the real neighbours. The case $n=1$ is particularly degenerate. The only proper non-empty subsets of $\mathcal{D} = \{0, 1\}$ are $\{0\}$ and $\{1\}$. In this case the $T_{-1+i, D}$ is a singleton and thus has Hausdorff dimension zero. Furthermore, it can be shown that $T_{-1+i, \mathcal{D}}$ has Hausdorff dimension $2$.  
\end{remark}

It is reasonable to wonder why we do not bound the modulus of $s$ directly. This is a viable strategy for deriving a sufficient condition for $(m, \{0, 1, \ldots, m-1\})$-representations where $m$ is a positive integer greater than one. If $s$ is a neighbour of $T_{m}$, it has the form $s = \sum\limits_{j=1}^{m}e_{j}m^{-j}$ where $e_{j}\in\{0, \pm 1, \ldots, \pm (m-1)\}$ for each $j$. Therefore
\begin{equation}\label{eq:easyBound}
|s| \leq \sum_{j=1}^{\infty}(m-1)m^{-j} = (m-1)\frac{1/m}{1-1/m} = 1.
\end{equation}
This holds regardless of the choice of $m$ and therefore we can be assured that if $D\subset\{0, 1, \ldots, m-1\}$ satisfies the condition that $|d-d^{'}| \neq 1$ for all $d, d^{'}\in D$, then $(m, D-D)$-representations are unique. This does not work out so cleanly for $(-n+i, \{0, 1, \ldots, n^{2}\})$-representations. Observe that in this case, a neighbour $s$ satisfies the inequality
\begin{equation}
|s| \leq \sum_{j=1}^{\infty}n^{2}\sqrt{n^{2}+1}^{-j} = \frac{n^{2}}{\sqrt{n^{2}+1} - 1} = \sqrt{n^{2} + 1} + 1,
\end{equation}
which grows with $n$. The corresponding sufficient condition would be the $n$-dependent condition that $|d-d^{'}| > \sqrt{n^{2}+1} + 1$ for all $d, d^{'}\in D$. The advantage of bounding $|\Re(s)-n\Im(s)|$ is that we gain more precise information about the neighbours which we can use to derive $n$-independent conditions. 

\begin{remark}\label{rem:realBasePhenom}
The consequences of the upper bound in (\ref{eq:easyBound}) extend to matrices of the form $A = \begin{bmatrix} m & 0 \\ 0 & n \end{bmatrix}$ and $D \subset \{0, 1, \ldots, m-1\} \times \{0, 1, \ldots, n-1\}$. It follows that all $(A, D-D)$-expansions are unique if for any pair of vectors in the set $D-D$, the distance between at least one of their components is greater than $2$. 
\end{remark}

\begin{theorem}\label{thm:realNeighbours} 
Fix an integer $n\geq 2$. Let $\mathcal{D} := \{0, 1, \ldots, n^{2}\}$. The set of real neighbours of $T_{-n+i, \mathcal{D}}$ is
\begin{itemize}
\item[(i)] $\{0, \pm1, \pm2\}$ when $n=1$ or $n\geq5$ .
\item[(ii)] $\{0, \pm1, \pm2, \pm3\}$ when $n=2, 3,$ or $4$.
\end{itemize}
Additionally, the set of real neighbours of $T_{-n+i, \{0, 1, \ldots, n^{2}\}}$ is $\{0, \pm 1\}$ for all $n\geq1$. 
\end{theorem}

\begin{proof}
The set of neighbours of $T_{n} := T_{-n+i, \{0, 1, \ldots, n^{2}\}}$ when $n = 1, 2$ and the set of neighbours of $E_{n} := T_{-n+i, \{0, \pm1, \ldots, \pm n^{2}\}}$ when $n = 1, 2, 3$ or $4$ can be explicitly computed using the neighbour finding algorithm found in \cite{ST03}. We prove the remaining cases. We abuse notation and use $\pi$ to be either $\pi_{A, D}$ and $\pi_{A, D-D}$ depending on the context. If it is preferred, we can view $\pi_{A, D}$ as a restriction of a common evaluation map with repsect to $A$ defined on the set of bounded sequences of vectors. 

It follows immediately from Lemma~\ref{lem:reimbound} that the real neighbours of $T_{n}$ must be in $\{0, \pm 1\}$. To see the converse, it can be verified by direct computation that 
\begin{align}
1 + \pi((2n-1), \overline{((n-1)^{2}+1), 0}) &= \pi(0,\overline{0, ((n-1)^{2}+1)}), \label{eq:examplecomputation}\\
-1 + \pi(0, \overline{0, ((n-1)^{2}+1)}) &= \pi((2n-1),\overline{((n-1)^{2}+1),0}),
\end{align}
where the bar indicates infinite repetition of those digits in the order presented. We include the verification of (\ref{eq:examplecomputation}).

Let 
\begin{align}
z_{1} := \pi(0, \overline{((n-1)^{2}+1), 0}), \\
z_{2} := \pi(0, \overline{0, ((n-1)^{2}+1)}). 
\end{align}
Let $b := -n+i$. Since the sequences of digits are periodic with period two,
\begin{align}
b^{2}z_{1} - z_{1} &= ((n-1)^{2}+1)b, \\
b^{2}z_{2} - z_{2} &=  (n-1)^{2} + 1. 
\end{align}
From these equations it is possible to solve for $z_{1}$ and $z_{2}$ explicitly in terms of $n$. We obtain
\begin{align}
z_{1} &=  \frac{((n-1)^{2}+1)(-n+i)}{(n^{2}-2-2ni)} \\
z_{2} &=  \frac{((n-1)^{2}+1)}{(n^{2}-2-2ni)}.
\end{align}
We wish to show that 
\begin{equation}\label{eq:goal}
\frac{((n-1)^{2}+1)}{(n^{2}-2-2ni)} + \frac{2n-1}{-n+i} + 1 =  \frac{((n-1)^{2}+1)}{(n^{2}-2-2ni)(-n+i)}
\end{equation}
Observe that on the left hand side of (\ref{eq:goal}), after bringing it under a common denominator, the numerator is 
\begin{equation}
((n-1)^{2}+1)(-n+i) + (2n-1)(n^{2}-2-2ni)+(-n+i)(n^{2}-2-2ni). 
\end{equation}
It now can be seen from the right hand side of (\ref{eq:goal}) that it is sufficient to show that 
\begin{equation}
(2n-1)(n^{2}-2-2ni)+(-n+i)(n^{2}-2-2ni) = (1 + n - i)((n-1)^{2}+1). 
\end{equation}
We conclude with
\begin{align}
&(2n-1)(n^{2}-2-2ni)+(-n+i)(n^{2}-2-2ni) \\
&= (n - 1 + i)(n^{2}-2-2ni) \\
&= ((n-1)^{2}+1)(n^{2}-2-2ni)/(n - 1 - i) \\
&= ((n-1)^{2}+1)(n + 1 - i).
\end{align}

Now we consider the $E_{n}$ . If $p$ is a neighbour of $E_{n}$, then $p$ has representation $(p_{j})_{j=1}^{\infty}$ where $p_{j}\in\{0, \pm 1, \ldots, \pm (2n^{2}-1), \pm (2n^{2})\}$. It follows that the function bounding $|\Re(s)-n\Im(s)|$ in (\ref{eq:boundnonly}) merely needs to be doubled to bound $|\Re(p)-n\Im(p)|$. Therefore $|\Re(p)-n\Im(p)| < 3$ and thus the real neighbours of $E_{n}$ are contained in $\{0, \pm 1, \pm 2\}$.  To see that $2$ and $-2$ are neighbours of $E_{n}$ for $n\geq5$, it can be verified directly that $\pi(0,\overline{(-n^{2})(n-2)^{2}}) = 2 + \pi((4n-2),\overline{(n-2)^{2}(-n^{2})})$ where $\pi$ is the coding map corresponding to sequences of differences of differences. 
\end{proof}

\begin{theorem}\label{thm:gaussianEquiv} 
Suppose that $n$ is a positive integer greater than $1$ and $D$ is a subset of $\{0, 1, \ldots, n^{2}\}$. 
The $(-n+i, D-D)$-representations are unique if 
\begin{itemize}
\item[(i)] $n$ is greater than or equal to $5$ and $D\subset\{0, 1, \ldots, n^{2}\}$ satisfies $|\delta - \delta^{'}| > 2$ for all $\delta \neq \delta^{'}\in D-D$. 
\item[(ii)] $n$ is one of $2, 3$ or $4$ and $D\subset\{0, 1, \ldots, n^{2}\}$ satisfies $|\delta - \delta^{'}| > 3$ for all $\delta \neq \delta^{'}\in D-D$. 
\item[(iii)] $n$ is greater than or equal to $2$ and $D\subset\{0, 1, \ldots, \lfloor n^{2}/2 \rfloor\}$ satisfies $|\delta - \delta^{'}| > 1$ for all $\delta \neq \delta^{'}\in D - D$. 
\end{itemize}
\end{theorem}

\begin{proof}
The claims (i) and (ii) are immediate consequences of Lemma~\ref{lem:neighUnique} applied to (i) and (ii) of Theorem~\ref{thm:realNeighbours} respectively. We prove (i) using a different argument. Suppose that $(x_{j})_{j=1}^{\infty}$ and $(y_{j})_{j=1}^{\infty}$ are equivalent $(-n+i, D-D)$-representations. Each $x_{j} = d_{j} - d_{j}^{'}$ and $y_{j} = d_{j}^{''} - d_{j}^{'''}$ for some $d_{j}, d_{j}^{'}, d_{j}^{''}, d_{j}^{'''}\in D$. Therefore
\begin{equation}\label{eq:digitDiffs}
\sum_{j=1}^{\infty}(d_{j} - d_{j}^{'})(-n+i)^{-j} = \sum_{j=1}^{\infty}(d_{j}^{''} - d_{j}^{'''})(-n+i)^{-j}. 
\end{equation}
We can manipulate (\ref{eq:digitDiffs}) to yield
\begin{equation}\label{eq:digitSums}
\sum_{j=1}^{\infty}(d_{j} + d_{j}^{'''})(-n+i)^{-j} = \sum_{j=1}^{\infty}(d_{j}^{''} + d_{j}^{'})(-n+i)^{-j}. 
\end{equation}
Since $d\leq n^{2}/2$ for all $d\in D$, and $d_{j}, d_{j}^{'}, d_{j}^{''}$, and $d_{j}^{'''}$ are all elements of $D$ for all $j$, their pairwise sums are in $\{0, 1, \ldots, n^{2}\}$. Furthermore, any separation condition on the elements of $\Delta$ holds if and only the same condition holds for elements of $D+D$. This is because $(a+b) - (c+d) = (a-c) - (d-b)$ for any collection of integers $a, b, c, d$. We conclude that (\ref{eq:digitSums}) is the equivalence of two $(-n+i, D)$-representations where $D-D$ does not contain any neighbours of $T_{-n+i, D}$ (apply Lemma~\ref{thm:realNeighbours}). By Lemma~\ref{lem:neighUnique}, we must have $d_{j} + d_{j}^{'''} = d_{j}^{''} + d_{j}^{'}$ for each $j$. In particular, $x_{j} = y_{j}$ for each $j$.
\end{proof}

There are other statements that can be derived from Lemma~\ref{lem:reimbound} that reveal choices of $D$ yielding unique $(-n+i, D)$-representations. We give the following illustrative example. 

\begin{example} 
Suppose that $n\geq3$ is a positive integer and $D$ is a subset of $\{0, 1, \ldots, n^{2}\}$. Lemma~\ref{lem:reimbound} asserts that the neighbours of $T_{-n+i, D}$ must satisfy $|\Re(s) - n\Im(s)| < 2$. This inequality was derived with no restrictions on the set $D$. In fact, $D$ is allowed to be equal to $\{0, 1, \ldots, n^{2}\}$. Suppose we assumed that every $d\in D$ is bounded above by $n^{2}/2$. It can be seen from the proof of Lemma~\ref{lem:reimbound} that $|Re(s) - n\Im(s)| < 1$. This means that, for example, $(-3+i, \{0, 1, 2, 3, 4\})$-representations are unique even though the elements of $D$ are consecutive. 
\end{example}

We end this section by demonstrating another application of Lemma~\ref{lem:reimbound}. Consider the following result from \cite{G82}. 
\begin{theorem}[W. Gilbert, \cite{G82}, proposition 1]\label{thm:neighbourset}
A Gaussian integer $s$ is a neighbour of $T_{-n+i, \{0, 1, \ldots, n^{2}\}}$ if and only if
\begin{itemize}
\item[(i)] $s\in\{0, \pm1, \pm (n-1+i), \pm (n+i)\}$ and $n\geq3$.
\item[(ii)] $s\in\{0, \pm1, \pm (1+i), \pm (2+i), \pm i, \pm (2+2i)\}$ and $n=2$. 
\end{itemize}
\end{theorem}

Gilbert used this result to derive the rules governing radix expansions in base $(-n+i, \{0, 1, \ldots, n^{2}\})$ (see theorem 5 and theorem 8 of \cite{G82}). Our proof, by way of Lemma~\ref{lem:reimbound}, uses a different approach than that of Gilbert. We first make the following observation. 

\begin{lemma}\label{lem:infinitewalk} 
Fix an integer $n\geq 2$. Let $b:= -n+i$. If $s$ is a neighbour of $T_{-n+i, \{0, 1, \ldots, n^{2}\}}$, then $bs + \delta$ is a neighbour of $T_{-n+i, \{0, 1, \ldots, n^{2}\}}$ for some $\delta\in\{0, \pm 1, \ldots, \pm n^{2}\}$. 
\end{lemma}

\begin{proof}
There exist sequences $(d_{j})_{j=1}^{\infty}$ and $(d_{j}^{'})_{j=1}^{\infty}$ with entries in $\{0, 1, \ldots, n^{2}\}$ such that 
\begin{equation}
s + \pi_{A, D}(d_{j})_{j=1}^{\infty} = \pi_{A, D}(d_{j}^{'})_{j=1}^{\infty}. 
\end{equation}
This equation holds if and only if
\begin{equation}
bs + (d_{1}-d_{1}^{'}) + \pi_{A, D}(d_{j+1})_{j=1}^{\infty} = \pi_{A, D}(d_{j+1}^{'})_{j=1}^{\infty}. 
\end{equation}
This completes the proof. 
\end{proof}

We now prove Theorem~\ref{thm:neighbourset}. 

\begin{proof}[Proof of Theorem \ref{thm:neighbourset}] 
The case $n = 2$ can be computed explicitly using the neighbour finding algorithm in \cite{ST03}. We proceed under the assumption that $n\geq 3$. Let $T_{n}$ denote $T_{-n+i, \{0, 1, \ldots, n^{2}\}}$. Suppose $s$ is a neighbour of $T_{n}$ and is a Gaussian integer. 

By Lemma \ref{lem:reimbound}, if $s$ is real, then $s\in\{0, \pm 1\}$. That lemma also implies that $s$ cannot be purely imaginary. If that were the case, $|n\Im(s)| < 2$ where $n\geq3$. This is impossible. 

We now claim that $|\Im(s)|$ cannot be larger than $1$. Let us denote $\Im(s)$ by $\beta$. Lemma $\ref{lem:reimbound}$ implies that $Re(s)$ is one of $n\beta - 1$, $n\beta$, or $n\beta+1$. Let us assume $\beta$ is greater than or equal to $2$. The magnitude of $s$ is therefore bounded below by $n\beta-1$.
On the other hand, $|s| \leq n^{2}\sum\limits_{j=1}^{\infty}|b|^{-j} = \sqrt{n^{2}+1} + 1$. We will argue that this quantity is strictly less $n\beta-1$ for $n\geq3$. 
Observe that for $n\geq3$
\begin{align}
\sqrt{n^{2}+1} + 1 &< n+2 \\
&\leq 2n-1 \\
&\leq n\beta-1.  
\end{align}
The second inequality holds since $n - 3 \geq 0$. This shows that when $\beta\geq2$, the number $s$ is not a neighbour of $T_{n}$. The case when $\beta$ is less than or equal to $-2$ is similar.

The only remaining cases to consider are $\beta = 1$ or $-1$. Lemma \ref{lem:reimbound} implies that $s$ is one of $\pm(n-1+i), \pm(n+i)$, and $\pm(n+1+i)$. 

By Lemma~\ref{lem:infinitewalk}, if $n+1+i$ is a neighbour of $T_{n}$, there must exist $\delta\in\{0, \pm 1, \ldots, \pm n^{2}\}$ such that $b(n+1+i) + \delta$ is also a neighbour of $T_{n}$. 
The set of neighbours of $T_{-n+i}$ is a subset of $\{0, \pm1, \pm(n-1+i), \pm(n+i), \pm(n+1+i)\}$. We observe that $b(n+1+i) = -(n^{2}+n+1)+i$. We see that adding an element of $\{0, \pm 1, \ldots, \pm n^{2}\}$ (a real number) to this expression cannot result in any of $0$, $1$, and $-1$. To see that not a single neighbour is obtainable, we compute $\delta$ using the remaining potential neighbours of $T_{n}$. 
\begin{align}
(n+i) + n^{2}+n+1 - i &= n^{2} + 2n + 1, \\
(-n-i) + n^{2}+n+1 - i &= n^{2}+1-2i, \\
(n-1+i) + n^{2}+n+1 - i &= n^{2}+2n, \\
(-n+1-i) + n^{2}+n+1 - i &= n^{2} +2 -2i, \\
(n+1+i) + n^{2}+n+1 - i &= n^{2}+2n+2, \\
(-n-1-i) + n^{2}+n+1 - i &= n^{2} -2i.
\end{align}
All of these are either larger than $n^{2}$ or are not real and therefore are not in $\{0, \pm 1, \ldots, \pm n^{2}\}$. We conclude that $n+1+i$ is not a neighbour of $T_{n}$. It also follows that $-n-1-i$ is not a neighbour of $T_{n}$ because $\zeta$ is a neighbour if and only if $-\zeta$ is as well.  

We have shown that a neighbour of $T$ is an element of the set $\{0, \pm1, \pm (n-1+i), \pm (n+i)\}$. To see that the converse also holds, we show that there exists $t\in T\cap(T+s)$ for each $s$. Let $\pi$ denote the $(-n+i, \{0, 1, \ldots, n^{2}\})$- coding map. It can be verified explicitly that

\begin{align}
1 + \pi((2n-1),\overline{((n-1)^{2}+1), 0}) &= \pi(0,\overline{0,((n-1)^{2}+1)}), \\
-1 + \pi(0,\overline{0, ((n-1)^{2}+1)}) &= \pi((2n-1), \overline{((n-1)^{2}+1), 0}), \\
(n+i) + \pi(n^{2}, 0, \overline{0, ((n-1)^{2}+1)}) &= \pi(0,(2n-1)\overline{((n-1)^{2}+1), 0}), \\
(-n-i) + \pi(0, (2n-1), \overline{((n-1)^{2}+1)0}) &= \pi(n^{2}, 0, \overline{0, ((n-1)^{2}+1)}),\\
(n-1+i) + \pi(\overline{((n-1)^{2}+1)0}) &= \pi(\overline{0, ((n-1)^{2}+1)}),\\
(-n+1-i) + \pi(\overline{0((n-1)^{2}+1)}) &= \pi(\overline{((n-1)^{2}+1), 0}).
\end{align}
The verification can be performed in the same way as in the proof of Theorem~\ref{thm:realNeighbours}.
\end{proof}

\section{Neighbour Graphs}\label{sec:neighGraphsss} 

The sufficient condition for the uniqueness of $(-n+i, D-D)$-expansions given in Theorem~\ref{thm:gaussianEquiv} are both easy to understand and apply. However, it does not capture the whole picture. To fill the gap, knowledge of the precise relationship between a pair of equivalent $(-n+i, D)$-representations can be used to refine the sufficient condition for uniqueness into a sufficient and necessary condition. The cost of strengthening the result, aside from the time and energy required to prove it, is the burden of a complicated condition in place of the previous simple one. 

 We consider the following familiar example to make clear what we mean by ``the relationship between equivalent representations". 

\begin{example}\label{ex:basemRules} 
Let $m\geq 2$ be a positive integer. A pair of $(m, \{0, 1, \ldots, m-1\})$-representations $(x_{j})_{j=1}^{\infty}$ and $(y_{j})_{j=1}^{\infty}$ is equivalent if and only if, without loss of generality, there exists a positive integer $J$ such that $x_{j} = y_{j}$ for all $j < J$, $x_{J} = y_{J} + 1$, $x_{j} = 0$ for all $j > J$, and $y_{j} = m-1$ for all $j > J$. 
\end{example}

In the same way, we can grasp the rules governing equivalent $(-n+i, D)$-representations. These were originally derived in \cite{G82}. The rules are considerably more complicated than those for base-$10$ and expressing them in the manner used in Example~\ref{ex:basemRules} is tedious and inefficient. Instead, we visualize representations as infinite walks along a directed graph as they are presented in \cite{G82}. 

\begin{definition} 
Let $A\in M_{n}(\mathbb{Z})$ be an expanding matrix and let $D\subset\mathbb{Z}^{n}$ be finite. Let $x := (x_{j})_{j=1}^{\infty}$ and $y := (y_{j})_{j=1}^{\infty}$ be $(A, D)$-representations. We call the sequence $(\zeta_{k})_{k=0}^{\infty}$ given by
\begin{equation}
\zeta_{k} = \begin{cases} 0 \;\text{if}\; k = 0,\\
\sum\limits_{j=1}^{k}A^{k-j}(x_{j}-y_{j}) \; \text{if}\; k \geq 1. 
\end{cases}
\end{equation}
the \textit{integer sequence of $(x, y)$}. We call an integer sequence a \textit{neighbour sequence} if the elements of the sequence are contained in the set of neighbours of $T_{A, D}$ unioned with $\{0\}$. 
\end{definition}

Our particular concern is with the integer sequences of equivalent $(A, D)$-representations. 

\begin{theorem}\label{thm:neighSeqEquiv}
Suppose that $A\in M_{n}(\mathbb{Z})$ is an expanding matrix and suppose that $D\subset\mathbb{Z}^{n}$ is finite. A pair of $(A, D)$-representations, $x$ and $y$, are equivalent if and only if the integer sequence of $(x, y)$ is a neighbour sequence.
\end{theorem}

\begin{proof}
Suppose $(x_{j})_{j=1}^{\infty}$ and $(y_{j})_{j=1}^{\infty}$ are equivalent $(A, D)$-representations. Since $\pi_{A, D}(x_{j})_{j=1}^{\infty} = \pi_{A, D}(y_{j})_{j=1}^{\infty}$, the equation $\sum\limits_{j=1}^{k}A^{j-1}(x_{j}-y_{j}) + \pi_{A, D}(x_{j})_{j=k+1}^{\infty} = \pi_{A, D}(y_{j})_{j=k+1}^{\infty}$ holds for all positive integers $k\geq1$. This equation implies that $\sum\limits_{j=1}^{k}A^{j-1}(x_{j}-y_{j})$ is a neighbour of $T_{A, D}$ for each $k$. 

Conversely, if the integer sequence of $x$ and $y$ is a neighbour sequence, it means that the integer sequence is bounded. This is because only a finite number of neighbours of $T_{A, D}$ are elements of $\mathbb{Z}^{n}$. Observe that for every $k$, 
\begin{equation}\label{eq:boundedEq}
\pi_{A, D-D}(x_{j} - y_{j})_{j=1}^{\infty} = A^{-k}\zeta_{k} + \sum_{j=k+1}^{\infty}A^{-j}(x_{j} - y_{j})
\end{equation}
where $(\zeta_{k})_{k=1}^{\infty}$ is the integer sequence of $(x, y)$. Since the integer sequence is bounded and multiplication by a sufficiently large power of $A^{-1}$ is a contraction with respect to Euclidean distance, the first term of the addition in (\ref{eq:boundedEq}) tends to zero as $k$ tends to infinity. Meanwhile the infinite series $\sum\limits_{j=1}^{\infty}A^{-j}(x_{j}-y_{j})$ is a convergent. Therefore the second term of (\ref{eq:boundedEq}) tends to zero as $k$ tends to infinity. It must be that $\pi_{A, D-D}(x_{j}-y_{j})_{j=1}^{\infty} = 0$. 
\end{proof}

We can imagine that a neighbour sequence is an infinite path in a graph which treats the neighbours of $T_{A, D}$ in $\mathbb{Z}^{n}$ along with zero as vertices. We describe how to draw edges in the next paragraph. 

Observe that if $(\zeta_{k})_{k=0}^{\infty}$ is any integer sequence, then $\zeta_{k+1} = A\zeta_{k} + (x_{k+1}-y_{k+1})$. Suppose we do not know the value of $\zeta_{k+1}$ or that of $x_{k+1}$ and $y_{k+1}$. If $(\zeta_{k})_{k=0}^{\infty}$ is a neighbour sequence, then we know $\zeta_{k+1}\in\mathbb{Z}^{n}$ is an element of the set of neighbours of $T_{A, D}$ or zero. If we know the set of neighbours of $T_{A, D}$ in $\mathbb{Z}^{n}$ and the value of $\zeta_{k}$, then we can choose elements of $D-D$ for $x_{k+1}-y_{k+1}$ to yield a potential $\zeta_{k+1}$. Not every element of $D-D$ might yield zero or a neighbour for a given $\zeta_{k}$. If it does, then we can draw an edge from $\zeta_{k}$ to that value and label the edge $x_{k+1}-y_{k+1}$. Furthermore, we can choose $\zeta_{0} = 0$ to be the initial vertex of the directed graph. 

\begin{definition} 
Let $A\in M_{n}(\mathbb{Z})$ be an expanding matrix and let $D\subset\mathbb{Z}^{n}$ be finite. We call a directed graph $G = (V, E)$ a \textit{neighbour graph} if 
\begin{itemize}
\item[(i)] The set of vertices $V$ is the set $\{0\}$ unioned with the set of neighbours of $T_{A, D}$ which occur as elements in at least one neighbour sequence.  
\item[(ii)] The set of edges $E$ is the set of $e\in D-D$, for which there exists a neighbour sequence $(\zeta_{k})_{k=0}^{\infty}$ and nonnegative integer $k$ such that $\zeta_{k+1} = A\zeta_{k} + e$. 
\item[(iii)] There is a directed edge from $v_{1}$ to $v_{2}$ if there exists a neighbour sequence $(\zeta_{k})_{k=0}^{\infty}$, a nonnegative integer $k$, and $e\in E$, such that $v_{1} = \zeta_{k}$, $v_{2} = \zeta_{k+1}$ and $v_{2} = Av_{1} + e$.
\end{itemize} 
\end{definition}

By construction, a sequence of vertices of a neighbour graph corresponds to an infinite directed path through the graph if and only if the sequence of vertices is a neighbour sequence of some pair of $(A, D)$-representations. The graphs are an alternative formulation of those given in \cite{ST03}. We provide an example illustrating these concepts for base-$10$. 

\begin{example} 
We examine $(10, \{0, 1, \ldots, 9\})$-representations. The self-affine set generated by $(10, \{0, 1, \ldots, 9)$ is the closed unit interval. It is clear that its integer neighbours are $\pm 1$. Let $(\zeta_{k})$ be an arbitrary neighbour sequence. Since $\zeta_{0} = 0$, we have that $\zeta_{1} = x_{1} - y_{1}$. 
Since $\zeta_{1}$ is either $-1$, $1$, or $0$, we see that $x_{1}$ and $y_{1}$ must either be one apart or equal. If they are equal, we have $\zeta_{1} = 0$ and we are back to original case. If $\zeta_{1} = 1$, then know that $\zeta_{2} = 10 + (x_{2} - y_{2})$. The only way for this to be in $\{-1, 0, 1\}$ is for $x_{2} = 0$ and $y_{2} = 9$. This yields $\zeta_{2} = \zeta_{1} = 1$, causing cycle with period one. If $\zeta_{1} = -1$, a similar argument shows that $\zeta_{2} = -1$. 

This matches the rules we recalled in Example~\ref{ex:basemRules}. The case that $\zeta_{k} = 0$ for all $k$ corresponds to when $(x_{j})$ and $(y_{j})$ are the same representation. The case $\zeta_{k} = 1$ for some $k$ corresponds to $x_{k} = y_{k} + 1$ and $x_{j} = 0$, $y_{j} = 9$ for all $j > k$. From our analysis, this can only happen when $\zeta_{j} = 0$ for $j < k$ and so $x_{j} = y_{j}$ for $j < k$. The case $\zeta_{k} = -1$ corresponds to the same kind of equivalence of representations known for base-$10$, only that the labels have been swapped. 

We summarize this information in the neighbour graph shown in Figure~\ref{fig:decimal}. The set of vertices are the neighbours $\pm 1$ with zero, and the set of edges are the possible values of $x_{j} - y_{j}$ when the representations are equivalent : $0, \pm 1$, and $\pm 9$. Rather than draw multiple edges between vertices, we have placed a $+$-symbol to indicate that when choices of vectors $\begin{matrix}x_{k} \\ y_{k} \end{matrix}$ are viable. For example, the pair $(x = (1, \overline{0}), y = (0, \overline{9})$ and $(x = (2, \overline{0}), y = (1, \overline{9}))$ both take the path right from vertex $0$ to $1$ and then loop back to $1$ forever. There is no $+$-symbol in either of the terminal loops because no other pairs of digits except for $0$ and $9$ have a difference of $\pm9$. The path that never leaves the $0$ vertex corresponds to an indistinct pair of representations. 
\end{example}

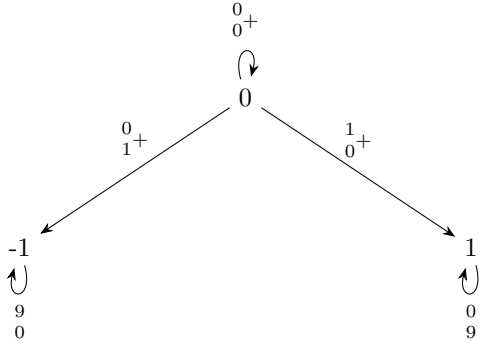
\begin{figure}
\centering
\vspace*{-3em}
\hspace*{-5.5em}
\begin{tikzpicture}
\begin{scope}[every node/.style]
     \node (A) at (0, 7) {0}; 
    \node (B) at (-3,5) {-1}; 
    \node (C) at (3,5) {1}; 
\end{scope}
\begin{scope}[>={Stealth[black]},
              every node/.style={fill=white},
              every edge/.style={draw=black}]
    \path [->] (A) edge[loop above]  node[above, fill=none]{\scriptsize$\begin{matrix}
0 \\
0 \\
\end{matrix}$+} (A);
    \path [->] (A) edge node[above, fill=none]{\scriptsize$\begin{matrix}
0 \\
1\\
\end{matrix}$+}(B);
    \path [->] (A) edge node[above, fill=none]{\scriptsize$\begin{matrix}
1 \\
0 \\
\end{matrix}$+}(C);
    \path [->] (B) edge[loop below] node[below, fill=none]{\scriptsize$\begin{matrix}
9 \\
0 \\
\end{matrix}$}(B);
    \path [->] (C) edge[loop below] node[below, fill=none]{\scriptsize$\begin{matrix}
0 \\
9 \\
\end{matrix}$}(C);
\end{scope}
\end{tikzpicture}
\caption{The neighbour graph for $(10, \{0, 1, \ldots, 9\})$-representations.}
\label{fig:decimal}
\end{figure}

The neighbour graph that governs the equivalence of $(-n+i, \{0, 1, \ldots, n^{2}\})$-representations is given in \cite{G82}. It carries additional information by keeping track of three states at each vertex that are associated with a triple of representations, instead of only a single state associated with a pair of representations. This kind of modification is possible when there is an upper bound on the number of representations an element of $T_{A, D}$ can have. In Example~\ref{ex:uncountAlphaRe}, we saw that some $(A, D-D)$-representations are equivalent to uncountably many other representations. Suppose $p := (p_{k})_{k=1}^{\infty}, q := (q_{k})_{k=1}^{\infty},$ and $r := (r_{k})_{k=1}^{\infty},$ are $(-n+i, \{0, 1, \ldots, n^{2}\})$-representations of the same complex number. They are not necessarily distinct. We define the $k$th state of $p, q$ and $r$ to be the triple 
\begin{equation} \label{eq:stateRel}
S(k) := (\zeta_{k}, \xi_{k}, -(\zeta_{k} + \xi_{k})),
\end{equation}
where $\zeta_{k}$ is the neighbour sequence of $(p, q)$ and $\xi_{k}$ is the neighbour sequence of $(q, r)$. It follows that $-(\zeta_{k} + \xi_{k})$ is the neighbour sequence of $(r, p)$. Notably, since the sum of these components is zero, one of the components is redundant. Nonetheless, it is useful to express all the differences explicitly in order to determine the digits at the $k$th place of the expansions $p, q,$ and $r$ (as demonstrated in Appendix~\ref{app:a}). It follows from the relationship $\zeta_{k+1} = A\zeta_{k} + (p_{k+1} - q_{k+1})$ that 
\begin{equation}\label{eq:states}
S(k+1) = AS(k) + (p_{k+1}-q_{k+1}, q_{k+1}-r_{k+1}, r_{k+1}-p_{k+1})
\end{equation}
where $A$ is applied component-wise to the tuple $S(k)$. Therefore the knowledge of the value of $S(k)$ can be used with Theorem~\ref{thm:neighbourset} to determine the possible values for the digits $p_{k+1}, q_{k+1}$, and $r_{k+1}$ and the state $S(k+1)$. 

If we treat allowable states as vertices, we can construct the graph. The directed edges indicate what states $S(k)$ can be achieved from a given state $S(k+1)$ (the node you are currently at). The graph in Figure~\ref{fig:radix} corresponds to the cases $n\geq3$ where $b=-n+i$. The case $n=2$ is more complicated and is presented in Appendix~\ref{app:b}. Both graphs feature a system of diagrams that communicate the value of a state. We describe the system for the case $n\geq3$ here. The additional states present in the case $n=2$ can be found in Appendix~\ref{app:b}. 

Let $(\zeta)_{k=1}^{\infty}$ be the neighbour sequence of $(p, q)$. We begin with a system of diagrams that communicate the value of $\zeta_{k}$. The system is as follows:

\begin{enumerate}

\item[(i)]  $\zeta_{k} = 0$ corresponds to \begin{tikzpicture}\draw (0,0) rectangle node{pq} (.75,.75); \end{tikzpicture}.
\item[(ii)] $\zeta_{k} = 1$ corresponds to  \begin{tikzpicture}\draw (0,0) rectangle node{q} (0.75,0.75); \draw (0.75, 0.75) rectangle  node{p} (1.5, 0); \end{tikzpicture}.

\item[(iii)] $\zeta_{k} = n-1+i$ corresponds to  \begin{tikzpicture}\draw (0,0) rectangle node{p} (0.75,0.75); \draw (0, 0) rectangle  node{q} (.75, -.75); \end{tikzpicture}.
\item[(iv)] $\zeta_{k} = n+i$ corresponds to  \begin{tikzpicture}\draw (0,0) rectangle node{q} (0.75,0.75); \draw (0.75, 0.75) rectangle  node{p} (1.5, 1.5); \end{tikzpicture}.
\end{enumerate}

\begin{figure} [p!]
\centering
\includegraphics[scale=0.63]{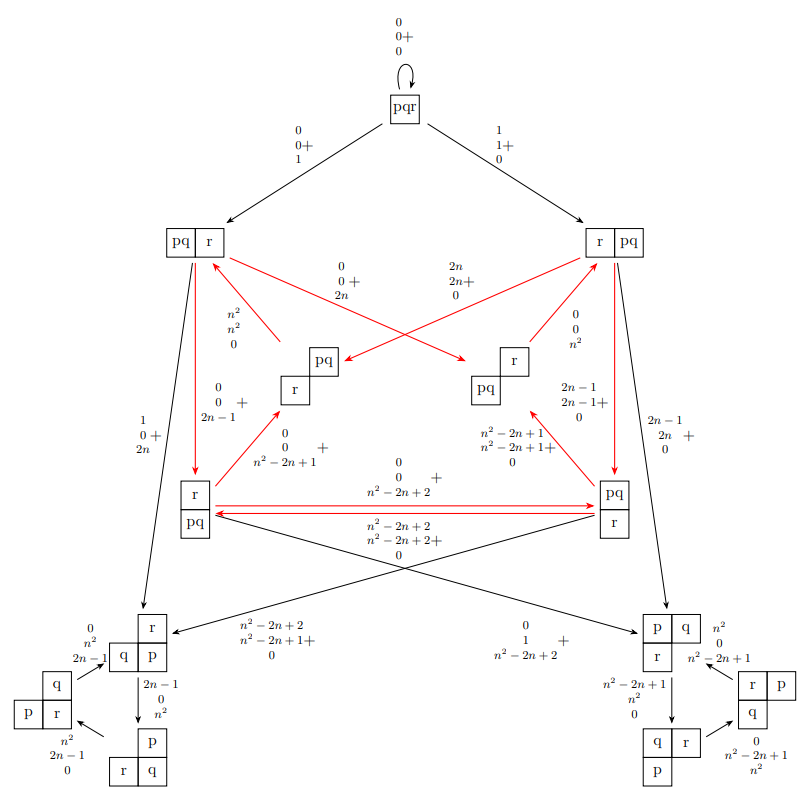}
\caption{The neighbour graph for $(-n+i, \{0, 1, \ldots, n^{2}\})$-representations, $n\geq3$.}
\label{fig:radix}
\end{figure}

Swapping the positions of $p$ and $q$ in any of these arrangements flips the sign on the value of $\zeta_{k}$. We can use this system to represent the values of $\zeta_{k}$, $\xi_{k}$, and $-(\zeta_{k} + \xi_{k})$ simultaneously. For example, the state $(1, -n-i, n-1+i)$ is communicated by 
$$\begin{tikzpicture}\draw (0,0) rectangle node{p} (0.75,0.75); \draw (0.75, 0.75) rectangle node{r}(0, 1.5);\draw (0, 0) rectangle node{q} (-0.75, 0.75); \end{tikzpicture}.$$
Each edge of the state graph is labelled with a triple of integers. These indicate a combination of digits, read from top to bottom, that communicate the allowable values of $(p_{k} - q_{k}, q_{k} - r_{k}, r_{k} - q_{k})$ in order for (\ref{eq:states}) to hold. This is a tuple of edges of the neighbour graph. The indication of a ``$+$" symbol means there are multiple tuples $(p_{k}, q_{k}, r_{k})$ which satisfy (\ref{eq:states}). Therefore the integers listed along the edges in the state graph communicate the distances between the digits at that index. 

\begin{theorem}\label{thm:rules} [W. J. Gilbert, \cite{G82}, theorem 5] 
Fix a positive integer $n\geq3$. The $(-n+i, \{0, 1, \ldots n^{2}\})$-representations $p, q$ and $r$ are pairwise equivalent if and only if they can be obtained from a single infinite path through the graph in Figure~\ref{fig:radix} starting at state $(0, 0, 0)$, if necessary, relabeling $p, q$ and $r$. 
\end{theorem}

\begin{remark}
In the approximation of $T_{3}$ in Figure~\ref{fig:3t}, it can be seen that at most three sets among $T_{3}$ and its translates by its neighbours intersect at a given point. The proof lies in directly checking that no quadruple / quintuple of distinct representations could satisfy the equivalent version of (\ref{eq:stateRel}). 
\end{remark}. 

\begin{corollary}
All $(-n+i, D)$-representations are unique if and only if $D$ is chosen such that no distinct pair of representations produces a neighbour sequence that corresponds to a path in Figure~\ref{fig:radix}. 
\end{corollary}

We include the derivation of figure~\ref{fig:radix} in the Appendix~\ref{app:a}. The descriptions that follow pertain to Figure~\ref{fig:radix}. 

If a complex number has a unique radix expansion in base $-n+i$, with $n\geq3$, then $p=q=r$ and this triple is perpetually in the state $(0, 0, 0)$. Complex numbers with precisely two distinct radix expansions correspond to paths that eventually exit the initial state $(0, 0, 0)$ but remain in the bold red subgraph that does not distinguish between $p$ and $q$. Complex numbers with three distinct radix expansions eventually exit the initial state $(0, 0, 0)$ and are ultimately trapped in one of the two loops of period three at the bottom of the diagram.

We provide an example to illustrate how to read the graph. 
\begin{example}
It can be verified by direct computation that the following $(-3+i, \{0, 1, \ldots, 9\})$-representations are equivalent:
\begin{equation*}
\begin{split}
p &=(0, 0, 0, \overline{4,0,9,}) , \\
q &= (0, 0, 1, \overline{9,4,0,}), \\
r & = (1,5,5, \overline{0,9,4,}). \\
\end{split}
\end{equation*}
The bar over the digits to the right of the radix point indicates a repetition of those digits with period three. The path that this number corresponds to in Figure~\ref{fig:radix} is the path that moves along the states

\hspace{-0.7cm}\begin{tikzpicture}
\begin{scope}[every node/.style]
    \node (A) at (0, 0) {\begin{tikzpicture}\draw (0,0) rectangle node{pqr} (.75,.75); \end{tikzpicture}};
    \node (B) at (2, 0) {\begin{tikzpicture}\draw (0,0) rectangle node{pq} (0.75,0.75); \draw (0.75, 0.75) rectangle  node{r} (1.5, 0); \end{tikzpicture}};
    \node (C) at (4, 0) {\begin{tikzpicture}\draw (0,0) rectangle node{pq} (0.75,0.75); \draw (0.75, 0.75) rectangle node{r} (0, 1.5); \end{tikzpicture}};
    \node (K) at (6, 0) {\begin{tikzpicture}\draw (0,0) rectangle node{r} (0.75,0.75);\draw (0.75, 0.75) rectangle node{p} (0, 1.5); \draw (0.75, 0.75) rectangle node{q} (1.5, 1.5); \end{tikzpicture}};
    \node (L) at (8.5, 0) {\begin{tikzpicture}\draw (0,0) rectangle node{p} (0.75,0.75); \draw (0.75, 0.75) rectangle node{q}(0, 1.5);  \draw (0.75, 0.75) rectangle node{r} (1.5, 1.5); \end{tikzpicture}};
    \node (M) at (11, 0) {\begin{tikzpicture}\draw (0,0) rectangle node{q} (0.75,0.75); \draw (0.75, 0.75) rectangle node{r} (0, 1.5);  \draw (0.75, 0.75) rectangle node{p} (1.5, 1.5); \end{tikzpicture}};
   
\end{scope}

\begin{scope}[>={Stealth[black]},
              every node/.style={fill=white},
              every edge/.style={draw=black}]
    \path [->] (A) edge (B);
    \path [->] (B) edge (C);
    \path [->] (C) edge (K);
    \path [->] (K)  edge (L);
    \path [->] (L)  edge (M);
    \path [->] (M) edge[bend right] (K); 
\end{scope}
\end{tikzpicture}.

The equivalent $(-3+i, \{0, 1, \ldots, 9\})$-representations 
$$(0, 2, 1,\overline{4, 0, 9,}),$$ $$(0, 2, 2, \overline{9, 0, 4,}),$$ $$(1, 7, 6, \overline{0, 9, 4,})$$ are also captured by this path. The distances between pairs of digits in the same coordinate are the same as those in the previous triple of expansions. 
\end{example}

\section{Bases Derived from the Quadratic Equation} 

We recall Lemma~\ref{lem:reimbound}: For positive integers $n \geq 3$, if a complex number $s$ is a neighbour of $T_{-n+i, \{0, 1, \ldots, n^{2}\}}$, then it satisfies the inequality $|\Re(s) - n\Im(s)| < 2$. We begin this section by establishing a version of this statement for a subset of complex numbers containing the family $-n+i$ for $n\geq 3$. 

\begin{lemma} \label{lem:reimbound2.0} 
Let $\rho$ be the root of a polynomial $x^{2} + Ax + B \in \mathbb{Z}[x]$ that satisfies $1 \leq A$, $2 \leq B$ and $A^{2} < 4B$. If $s\in\mathbb{C}$ is a neighbour of $T_{\rho, \{0, 1, \ldots, B-1\}}$, then 
\begin{equation}
\bigg|\Re(s) - \frac{A}{\sqrt{4B-A^{2}}}\Im(s)\bigg| < 5. 
\end{equation}
\end{lemma}

The Gaussian integer $-n+i$ is a root of the polynomial $x^{2} + 2nx + n^{2} + 1$. 

\begin{proof}
The complex number $\rho$ has the form $\frac{-A \pm i\sqrt{4B - A^{2}}}{2}$. Assume that $\rho$ is the number obtained from ``$+$". The claim for the other root will hold because it is the complex conjugate. 

As in Lemma~\ref{lem:reimbound}, we explicitly compute the first few terms of $s$
in terms of $A$ and $B$. Recall that 
\begin{equation}
s = \frac{e_{1}}{\rho} +  \frac{e_{2}}{\rho^{2}} + \frac{e_{3}}{\rho^{3}} + \delta_{4}
\end{equation}
where $\delta_{4}$ denotes the tail $\sum\limits_{j=4}^{\infty}e_{j}\rho^{-j}$ and $e_{j} \in \{0, \pm 1, \ldots, \pm (B-1)\}$. 
Let $\alpha := \Re(s)$ and $\beta := \Im(s)$.

Explicit computation yields
\begin{equation} \label{eq:neighbourreal2} 
\alpha = \frac{-A}{2B}e_{1} +  \frac{2A^{2}-4B}{4B^{2}}e_{2} + \frac{-4A^{3}+12AB}{8B^{3}}e_{3} + \Re(\delta_{4}) \\ 
\end{equation}
and 
\begin{equation} \label{eq:neighbourimag2} 
\beta = \frac{-\sqrt{4B-A^{2}}}{2B}e_{1} + \frac{2A\sqrt{4B-A^{2}}}{4B^{2}}e_{2} + \frac{(4B-4A^{2})\sqrt{4B-A^{2}}}{8B^{3}}e_{3} + \Im(\delta_{4}). \\
\end{equation}
The quantity $\alpha - \frac{A}{\sqrt{4B-A^{2}}}\beta$ is then 
\begin{equation} \label{eq:realimagdiff2} 
 \frac{-1}{B}e_{2} + \frac{A}{B^{2}}e_{3} + \Re(\delta_{4}) - \frac{A}{\sqrt{4B-A^{2}}}\Im(\delta_{4}). 
\end{equation}

Recall that $e_{j}$ range from $-B+1$ to $B-1$. In order to bound $|\alpha - \frac{A}{\sqrt{4B-A^{2}}}\beta|$ with an expression that is only in terms of $A$ and $B$, we maximize the sum of the first three terms of (\ref{eq:realimagdiff2}) by choosing $e_{2} = -B + 1$ and $e_{3} = B-1$ and estimate the absolute value of the last two terms using the bound $|Re(\delta_{4}) - \frac{A}{\sqrt{4B-A^{2}}}\Im(\delta_{4})| \leq (1+\frac{A}{\sqrt{4B-A^{2}}})\sum\limits_{j=4}^{\infty}(B-1)\rho^{-j}$. This results in the inequality
\begin{equation} \label{eq:boundnonly2} 
\bigg|\alpha - \frac{A}{\sqrt{4B-A^{2}}}\beta\bigg| \leq \frac{B-1}{B} + \frac{A(B-1)}{B^{2}} + \frac{(1+\frac{A}{\sqrt{4B-A^{2}}})(\sqrt{B}+1)}{B\sqrt{B}}. 
\end{equation}

We can respectively bound each term of (\ref{eq:boundnonly2}) in the following way:
\begin{align}
\frac{B-1}{B} &< 1, \label{eq:firstbound2}\\ 
\frac{A(B-1)}{B^{2}} &\leq 1, \label{eq:midbound2}\\
 \frac{(1+\frac{A}{\sqrt{4B-A^{2}}})(\sqrt{B}+1)}{B\sqrt{B}} &\leq \bigg(\frac{1}{B} + 1\bigg)\bigg(\frac{1}{\sqrt{B}}+1\bigg). \label{eq:lastbound2}
\end{align}
Consider the sum of all the sequences on the right hand side of a ``$<$" sign from (\ref{eq:firstbound2}) to (\ref{eq:lastbound2}). It is assumed that $B \geq 2$. Then the sum in question is bounded by $2 + (1.5)(1 + \frac{1}{\sqrt{2}}) = 3.5 + \frac{3}{2\sqrt{2}} < 5$. For (\ref{eq:midbound2}) and (\ref{eq:lastbound2}), observe that when $A = 1, 2$ or $3$, it can be checked directly that $A/B \leq 1$. After which, the assumption $4B > A^{2}$ is equivalent to $A/B < 4/A$. 
\end{proof}

Better bounds can be obtained for families of pairs $(A, B)$, as was demonstrated for $(2n, n^{2}+1)$, $n \geq3$, in Lemma~\ref{lem:reimbound}. 

\begin{example} 
Let $\rho = \frac{-9}{2} + i\frac{\sqrt{3}}{2}$. This is a root of the polynomial $x^{2} + 9x + 21$. Let $D = \{0, 10, 20\}$. Let $\alpha\in\mathbb{C}$ be the number that $(\rho, D)$-representation $(10, 0, \overline{10, 0})$. Since Lemma~\ref{lem:reimbound2.0} implies that all the real neighbours of $T_{\rho, D}$ have modulus less than $5$, we can see that all the real neighbours of $T_{\rho, D-D}$ have modulus less than $10$. By Lemma~\ref{lem:neighUnique}, since $D-D$ excludes all those candidates, the  $(\rho, D)$-representation of $\alpha$ is unique. The sequence $(D \cap (D+\alpha_{j}))_{j=1}^{\infty}$ is the SEP sequence $(\overline{\{10, 20\}, D})$. 

Lastly, the number $\rho$ corresponds to the $2$ by $2$ matrix $\begin{bmatrix} -9/2 & -\sqrt{3}/2 \\ \sqrt{3}/2 & -9/2 \end{bmatrix}$. Multiplication by the inverse of this matrix is a similarity with respect to Euclidean distance with contraction coefficient $1/|\rho|$. By Theorem~\ref{thm:attractSEP}, the set $T_{\rho, D} \cap (T_{\rho, D} + \alpha)$ is self-similar and by Theorem~\ref{thm:SepSimDim} has dimension $\log(4)/\log(21)$. 
\end{example}

The task of deriving a sufficient condition for the uniqueness of representations by excluding neighbours from the difference set $D-D$ is significantly easier when $D$ is the subset of the span of a single vector. It both limits our search and simplifies the arithmetic. In the most extreme case, however, choosing $D$ to have that property can cause the attractor $T_{A, D}$ to be one dimensional in the sense of vector spaces. For example, if $A = 3I$ where $I$ is the identity matrix and $D = \{0, 2v\}$ where $v$ is any of the standard basis vectors, then $T_{A, D}$ is essentially the middle third Cantor set in the unit interval. This is not the case for the pair $(-n+i, \{0, 1, \ldots, n^{2}\})$ because, by Theorem~\ref{thm:radixExistCanon}, complex numbers have $(-n+i, \{0, 1, \ldots, n^{2}\})$-representations. The set of pairs $(\rho, \{0, 1, \ldots, B-1\})$ where $\rho$ is a root of the polynomial $x^{2} + Ax + B \in \mathbb{Z}[x]$ satisfying $1 \leq A$, $2 \leq B$ and $A^{2} < 4B$ share this property. 

\begin{theorem}[W. Gilbert, \cite{G81}, theorem 1] \label{thm:quadReps} 
Suppose $A$ and $B$ be integers satisfying $1\leq A$, $2 \leq B$. If $\rho$ is a root of the polynomial $x^{2} + Ax + B$, then for every $y\in\mathbb{Z}[\rho]$ there exist $d_{0}, d_{1}, \ldots, d_{m}$ for some nonnegative integer $m$ such that
\begin{equation}
y = d_{0} + d_{1}\rho + \cdots + d_{m}\rho^{m}.
\end{equation}
\end{theorem}

It is natural to wonder if there are conditions on higher degree polynomials in $\mathbb{Z}[x]$ which ensure that all elements of $\mathbb{Z}[\rho]$ have expansions $\rho$ where $D$ is a finite subset of $\mathbb{Z}_{\geq0}$. While such questions are treated in \cite{ABP03} and \cite{BHP06}, an alternative approach may have to be employed to find the neighbours. Notably, there are no algebraic solutions to polynomials of degree greater than or equal to $5$ with which to perform our computations. A different direction, suggested by Theorem~\ref{thm:quadReps}, is that of extending the results in Chapter~\ref{chp:highDimSEP} and Chapter~\ref{chp:limFormula} that are stated over $\mathbb{Z}^{n}$ to other lattices.

\section{A Special Case of ``Self-similarity if and only if SEP"}\label{sec:sepSimSp} 

Suppose $n$ is an integer greater than or equal to $2$ and $m$ is an integer satisfying $2 \leq m \leq n^{2}$. We end this chapter by demonstrating that it is possible to prove that $T_{-n+i, \{0, m\}} \cap T_{-n+i, \{0, m\}} + \alpha$ is self-similar if and only if a certain sequence of \emph{integers} is SEP. This is a simpler property than the SEP property for sequences of sets. We are motivated to include the details of the proof because the argument does not assume that the linear factors of the IFS that generates the intersection is a power of $(-n+i)^{-1}$ in order to obtain the SEP condition. This is a fundamental difference from Theorem~\ref{thm:attractSEP}. Furthermore, we do not assume that uniqueness of all $(-n+i, \{0, \pm m\})$-representations, only the uniqueness of the representation of $\alpha$. Although we admit that all but finitely many $m$ satisfy the sufficient condition that all $(-n+i, \{0, \pm m\})$-representations, $n > \sqrt{m}$, are unique. 

\begin{definition} 
A sequence $(a_{j})_{j=1}^{\infty}$ of integers is \textit{strongly eventually periodic} (SEP) if there exists a finite sequence $(b_{\ell})_{\ell = 1}^{p}$ and a nonnegative sequence $(c_{\ell})_{\ell = 1}^{p}$, where $p$ is a positive integer, such that
\begin{equation}
(a_{j})_{j=1}^{\infty} = (b_{\ell})\overline{(b_{\ell} + c_{\ell})_{\ell = 1}^{p}}, 
\end{equation}
where $\overline{(d_{\ell})_{\ell = 1}^{p}}$ denotes the infinite repetition of the finite sequence $(d_{\ell})_{\ell = 1}^{p}$.   
\end{definition}

The following is a convenient sufficient condition for the SEP property. 

\begin{lemma}\label{lem:monoSEP}
Let $(a_{j})_{j = 1}^{\infty}$ be a bounded sequence of nonnegative integers. If there exists a positive integer $q$ such that $a_{j} \leq a_{j + q}$ for all $j\geq1$, then $(a_{j})_{j=1}^{\infty}$ is SEP. 
\end{lemma}

\begin{proof}
Since $a_{j}$ is bounded for all $j$, there exists a positive integer $q$ such that $a_{j + kp}$ is constant for all $k\geq q$ and $j = 1, 2, \ldots, q$. It follows that there exists a positive integer $m$ such that $(a_{j})_{j=1}^{\infty} = a_{1}a_{2}\ldots a_{mq}\overline{a_{1+mq}a_{2+mq}\ldots a_{q + mq}}$. 

For each $j = 1, 2, \ldots, q$, there exist non-negative integers $v_{j + (m-\ell)q}$ where $\ell = 1, 2, \ldots m$ such that
\begin{align}
a_{j + mq} &= a_{j+ (m-1)q} +v_{j + (m-1)q} \\
&= a_{j + (m-2)q} + v_{j + (m-1)q} + v_{j+(m-2)q} \\
&\;\;\vdots \\
&= a_{j} + \sum_{\ell = 1}^{m} v_{j + (m-\ell)q}. 
\end{align}
By choosing $u_{j + qk} = \sum\limits_{\ell = 1}^{m-k}v_{j + (m-\ell)q}$ for $j = 1, 2, \ldots, q$ and $k = 0, 1, \ldots, m - 1$, we obtain $(a_{j})_{j\geq1} = (a_{1}a_{2}\ldots a_{m})\overline{(a_{1} + u_{1})(a_{2}+u_{2})\ldots (a_{mq}+u_{mq})}$. 
\end{proof}

\begin{definition}
Fix an integer $n\geq2$, $D = \{0, m\}$, $2\leq m \leq n^{2}$ and let $\alpha$ have a unique $(-n+i, D-D)$-representation $(\alpha_{j})_{j=1}^{\infty}$. We call $\gamma := \pi_{-n+i, D}(\gamma_{j})_{j=1}^{\infty}$, where $\gamma_{j} := \min(D \cap (D + \alpha_{j}))$ for each $j$, the \textit{minimal element of $T_{-n+i, D} \cap (T_{-n+i, D} + \alpha)$}. 
\end{definition}

\begin{lemma}\label{lem:shiftFormSp} 
Suppose $n$ is a positive integer and that $D = \{0, m\}$ where $2\leq m \leq n^{2}$. If $\alpha$ has a unique $(-n+i, D-D)$-representation $(\alpha_{j})_{j=1}^{\infty}$ and $\gamma$ is the minimal element of $T(\alpha) := T_{-n+i, D} \cap (T_{-n+i, D} + \alpha)$. Then $T(\alpha) - \gamma = \{\pi_{-n+i, D}(z_{j})_{j=1}^{\infty} : z_{j}\in D, z_{j} \leq m - |\alpha_{j}|\}$ and $T(\alpha) - \gamma$ is a subset of $T_{-n+i, D}$. 
\end{lemma}

\begin{proof}
By Lemma~\ref{lem:seqExpress}, we can equate $T(\alpha)$ and $\pi_{-n+i, D}\left(\prod\limits_{j=1}^{\infty}D\cap(D + \alpha_{j})\right)$. It follows that $T(\alpha) - \gamma$ is equal to $\pi_{-n+i, D}\left(\prod\limits_{j=1}^{\infty}(D\cap(D + \alpha_{j})) - \gamma_{j}\right)$. 

It is sufficient to show that, for any $j\geq1$, the condition $z_{j} \leq m - |\alpha_{j}|$ and $z_{j}\in D$ holds if and only if $z_{j}$ is an element of $(D \cap (D + \alpha_{j})) - \gamma_{j}$. There are three possible values for $\alpha_{j}$ for each $j$. If $\alpha_{j} = m$, $(D \cap (D + \alpha_{j})) - \gamma_{j} = \{0\}$ and $m - |\alpha_{j}| = 0$. If $\alpha_{j} = -m$, we obtain the same equations. When $\alpha_{j} = 0$, $(D \cap (D + \alpha_{j})) - \gamma_{j} = \{0, m\}$ and $m - |\alpha_{j}| = m$. Since $\{0, m\} \cap \{z_{j} : z_{j} \leq m\}$ is equal to $\{0, m\}$, the condition holds in this case as well.  

Lastly, from these calculations we see that $(D\cap(D+\alpha_{j})) - \gamma_{j}$ is a subset of $\{0, m\} = D$ for all $j$, no matter the value of $\alpha_{j}$. 
\end{proof}

\begin{theorem}\label{thm:selfSimSEPSp} 
Fix an integer $n\geq 2$ and let $b:= -n+i$. Suppose $D = \{0, m\}$ where $2 \leq m \leq n^{2}$ and that $\alpha$ chosen such that $\alpha$ has a unique $(-n+i, D)$-representation. Let $\gamma$ be the minimal element of $T(\alpha) := T_{-n+i, D} \cap (T_{-n+i, D} + \alpha)$. 

If $T(\alpha)$ is the attractor of a collection of similarities containing a function of the form $f(x) = rx + (1 - r)\gamma$ for some $r\in\mathbb{C}$, then $(m -|\alpha_{j}|)_{j=1}^{\infty}$ is SEP. 

Conversely, if $(m -|\alpha_{j}|)_{j=1}^{\infty}$ is SEP and so by definition can be written as $(a_{\ell})_{\ell=1}^{p}\overline{(a_{\ell} + u_{\ell})_{\ell=1}^{p}}$, then $C(\alpha)$ is self-similar and is the attractor of the IFS containing precisely the functions of the form 
\begin{equation*}
f(x) = b^{-p}(x + \sum_{\ell=1}^{p}(y_{\ell}b^{p-\ell} + z_{\ell}b^{-\ell})-\gamma)+\gamma
\end{equation*}
where for each $\ell$, $y_{\ell}, z_{\ell}\in\{0, m\}$ such that $y_{\ell} \leq a_{\ell}$ and $z_{\ell} \leq u_{\ell}$.
\end{theorem}

We follow the proof strategy for theorem 1.2 in \cite{LYZ11}.

\begin{proof} 
Assume that $(m-|\alpha_{j}|)_{j=1}^{\infty}$ is an SEP sequence of integers. By definition there exist integers $a_{1}, a_{2}, \ldots, a_{p}$ and $u_{1}, u_{2}, \ldots, u_{p}$ such that $(m-|\alpha_{j}|)_{j\geq1} = a_{1}\ldots a_{p}\overline{(a_{\ell}+u_{\ell})_{\ell = 1}^{p}}$. By Lemma~\ref{lem:shiftFormSp}, $C(\alpha, \gamma) := C(\alpha) - \gamma$ is equal to $$\{\pi(z_{j})_{j=1}^{\infty} : z_{j}\in D, z_{j} \leq m - |\alpha_{j}|\}.$$ The ``SEP implies self-similarity" claim is a consequence of Lemma~\ref{lem:sepGen}.

We assume $T(\alpha)$ is self-similar and is generated by an IFS containing the map $f_{1}(x) = r_{1}x + (1-r_{1})\gamma$. It follows from Lemma~\ref{lem:attractorShift} that $T(\alpha)-\gamma$ is self- similar and is generated by an IFS containing the map $g_{1}(x) = r_{1}x$. Observe that the sequence $\overline{0}$ is trivially SEP. Let us then assume that at least one of the entries of $(m-|\alpha_{j}|)_{j=1}^{\infty}$ is nonzero. 

It follows from Lemma~\ref{lem:shiftFormSp} that $T(\alpha)-\gamma$ contains $mb^{-q}$ for some positive integer $q$. Therefore $mb^{-q}r_{1} = g_{1}(mb^{-q})$ is an element of $T(\alpha)-\gamma$ and, in particular, can be expressed as $\pi_{-n+i, D}(mr_{1, j})_{j=1}^{\infty}$ where $r_{1, j}\in\{0, 1\}$. Isolating for $r_{1}$ yields
\begin{equation}
r_{1} = \sum_{j=1}^{\infty}r_{1, j}b^{-j+q}. 
\end{equation}

For every $s\geq1$, we have that
\begin{align}
\tau &:= g_{1}((m-|\alpha_{s}|)b^{-s}) \\
 &= (m-|\alpha_{s}|)r_{1, 1}b^{q-1-s} + \cdots+(m-|\alpha_{s}|)r_{1, q-s} +\!\!\! \sum_{j=q - s +1}^{\infty}\!\!(m-|\alpha_{s}|)r_{1, j}b^{-(j+s)} \label{eq:gImage}
\end{align}
is an element of $T(\alpha)-\gamma$. By Lemma~\ref{lem:shiftFormSp}, it is also an element of $T_{-n+i, D}$. Therefore $\tau = \pi(d_{j})_{j=1}^{\infty}$ where $d_{j}\in\{0, m\}$. We now argue that the expansion in (\ref{eq:gImage}) is $\pi(d_{j})_{j=1}^{\infty}$. If $s > q - 1$, the desired result follows from the fact that $(-n+i, \{0, m\})$-expansions are unique for all $m > 1$ (Combine Lemma~\ref{lem:neighUnique} with Theorem~\ref{thm:realNeighbours}). Suppose $s \leq q - 1$.  The number $\tau/b^{q-1-s}$ is equal to $\sum\limits_{j=1}^{\infty}d_{j}b^{-j-(q-1-s)}$ and is also an element of $T_{-n+i, D}$. Again, by the uniqueness of the $(-n+i, \{0, m\})$-representations, it must be that $r_{1, 1} = r_{1, 2} = \cdots = r_{1, q-s} = 0$ and $(m-|\alpha_{s}|)r_{1, q-s+j} = d_{j}$ for all $j\geq1$. We obtain the desired result by multiplying back by $b^{q-1-s}$. 

Observe that $r_{1} \neq 0$ implies that $r_{1, t} = 1$ for some $t\geq1$. In particular, since $|r_{1}| < 1$, it must be that we can choose $t > q$. Otherwise, $r_{1}$ is a nonzero Gaussian integer and thus has magnitude greater than one. Fix such a $t$. It follows from the discussion above that $m - |\alpha_{s}| = d_{s-q+t}$ for all $s\geq 1$. Since $\tau$ is an element of $T(\alpha)-\gamma$, it follows from Lemma~\ref{lem:shiftFormSp} that $d_{s-q+t} \leq m - |\alpha_{s + (t - q)}|$ for all $s\geq1$. Therefore $m - |\alpha_{s}|  \leq m - |\alpha_{s + p}|$ for all $s\geq 1$ where $p:= t - q$ is a positive integer. We conclude that the sequence $(m - |\alpha_{j}|)_{j = 1}^{\infty}$ is SEP by Lemma~\ref{lem:monoSEP}. 
\end{proof}

\chapter{Multiplicative Invariance in $\mathbb{Z}^{n}$}\label{chp:discMultInv}

In this chapter, we present the multiplicative invariance in the non-negative integers introduced in \cite{GMR24} and highlight a result we replicate with respect to our definition of multiplicative invariance in $\mathbb{Z}^{n}$. We denote the nonnegative integers by $\mathbb{N}_{0}$ .

\section{Basic Definitions}

Let $r$ be an integer greater than or equal to $2$. For any integer for $n\geq1$, let $k:= \lfloor \log_{r}n \rfloor$. This is the greatest integer power of $r$ which is less than or equal to $n$. . Let $\phi_{r}: \mathbb{N}_{0} \rightarrow  \mathbb{N}_{0}$ and $\psi_{r}: \mathbb{N}_{0} \rightarrow \mathbb{N}_{0}$ be the functions given by
\begin{align}
&\phi_{r}(n) = \lfloor n/r \rfloor, \\
&\psi_{r}(n) =  n - r^{k}\lfloor n/r^{k} \rfloor.
\end{align}

The image of the functions is clearer if we replace $n$ by its base-$r$ representation. If $n = a_{k}r^{k} + \cdots a_{1}r + a_{0}$, $a_{k} \neq 0$, then 
\begin{align}
&\phi(n) = a_{k}r^{k-1} + \cdots + a_{2}r + a_{1}, \\
&\psi(n) = a_{k-1}r^{k-1} + \cdots + a_{1}r + a_{0}. 
\end{align}

\begin{definition}
A set $A \subset \mathbb{N}_{0}$ is called $\times r$ invariant if $\phi(A) \subset A$ and $\psi(A) \subset A$. We call $A\subset\mathbb{N}_{0}$ multiplicatively invariant if it is $\times r$ invariant for some $r\geq2$. 
\end{definition}

A $\times r$ invariant subset of the nonnegative integers can be mapped to a times-$r$ invariant subset of $[0, 1)$ using a limit process. We recall the definition of multiplicative invariance in $[0, 1)$. 

\begin{definition} 
Let $r$ be a positive integer. Define the map
\begin{equation}
\begin{split}&T_{r} : \mathbb{R} \rightarrow [0, 1) \\
&x \mapsto rx\mod{1}
\end{split}
\end{equation}
A nonempty closed subset $Y \subset [0, 1]$ is called $\times r-$\textit{invariant} if $T_{r}(Y) \subset Y$. A subset $Y$ is called \textit{multiplicatively invariant} if it is $\times r$-invariant for some $r\geq2$.
\end{definition}

\begin{proposition}[D. Glasscock, J. Moreira, F. Richter,\cite{GMR24}, proposition 3.15]\label{thm:discreteToContin} 
Let $A\subset\mathbb{N}_{0}$ be a $\times r$ invariant subset of the nonnegative integers. Define a sequence of sets $X_{k} := A\cap [0, r^{k}) / r^{k}$ where $k\geq1$. The sequence $(X_{k})_{k=1}^{\infty}$ converges in Hausdorff distance to a $\times r$ invariant set $X\subset [0, 1)$.  
\end{proposition}

\begin{remark}
In \cite{GMR24}, the fractal properties of $\times r$-invariant subsets of $\mathbb{N}_{0}$ using a discrete analogue to box-counting called the mass dimension. Theorem~\ref{thm:discreteToContin} as stated in \cite{GMR24} asserts that $\dim_{M}A = \dim_{B}X$. 
\end{remark}

\section{Multiplicative Invariance on $\mathbb{Z}^{n}$}

We present an analogous version of multiplicative invariance in $\mathbb{Z}^{n}$ using matrices as our base. We ultimately establish a version of Theorem~\ref{thm:discreteToContin} in this context. 

We first require a definition for multiplicative invariance in $\faktor{\mathbb{R}^{n}}{\mathbb{Z}^{n}}$. We directly generalize the definition used for $[0, 1)$ by viewing it as a subset of $\faktor{\mathbb{R}}{\mathbb{Z}}$. We require a topology on $\faktor{\mathbb{R}^{n}}{\mathbb{Z}^{n}}$ in order to state the definition. The natural metric is to map two cosets to the minimum distance realized by two of their respective representatives. Throughout this chapter we use $\mathbb{T}^{n}$ to denote $\faktor{\mathbb{R}^{n}}{\mathbb{Z}^{n}}$.

\begin{lemma}\label{lem:natMet} 
The function $d:\mathbb{T}^{n} \rightarrow [0, \infty)$ given by 
\begin{equation} 
d(x + \mathbb{Z}^{n}, y + \mathbb{Z}^{n}) := \min_{\zeta\in\mathbb{Z}^{n}}\norm{x - y + \zeta}
\end{equation}
is a metric. 
\end{lemma}

\begin{proof}
First, we show that the minimum exists. For any $x, y\in\mathbb{R}^{n}$, we can consider the closed ball of radius $\norm{x-y}$ centered at $x$. This set contains finitely many elements of the form $y + \zeta$ where $\zeta\in\mathbb{Z}^{n}$ since it is bounded. It also contains $y = y+0$. Therefore there must be a choice of $\zeta\in\mathbb{Z}^{n}$ that minimizes the norm of $x-y+\zeta$.  

Next we show that the functions does not depend on the choice of the representative of any coset. If $z$ and $w$ are equivalent to $x$ and $y$ respectively, then there exists $\xi\in\mathbb{Z}^{n}$ such that $z-w = x - y + \xi$. We then have 
\begin{equation}
\min_{\zeta\in\mathbb{Z}^{n}}\norm{z-w+\zeta} = \min_{\zeta\in\mathbb{Z}^{n}}\norm{x-y+\xi + \zeta} =  \min_{\zeta\in\mathbb{Z}^{n}}\norm{x-y+\xi + \zeta}
\end{equation}

Now we consider the norm conditions. 

The function $d$ inherits nonnegativity from the Euclidean norm. The function is symmetric because, if $\zeta$ that minimizes $\norm{x-y+\tau}$, the vector $-\zeta$ must minimize $\norm{y-x+\tau}$. Two elements of $\mathbb{R}^{n}$, $x$ and $y$, are representatives of the same coset if and only if $x = y + \zeta$ for some $\zeta\in\mathbb{Z}^{n}$. It follows that $d(x+\mathbb{Z}^{n}, y+\mathbb{Z}^{n}) = 0$ if and only if $x+\mathbb{Z}^{n} = y+\mathbb{Z}^{n}$. 

Lastly, we verify that the triangle inequality holds. Let $x, y, z\in \mathbb{R}^{n}$. There exist $\zeta_{1}, \zeta_{2}, \zeta_{3},$ such that
\begin{align}
d(x, y) &= \norm{x-y+\zeta_{1}}, \\
d(y, z) &= \norm{y-z+\zeta_{2}}, \\
d(x, z) &=  \norm{x-y+\zeta_{3}}.
\end{align}
Since $\zeta_{1} + \zeta_{2}\in\mathbb{Z}^{n}$, we have $\norm{x - z + \zeta_{3}|}\leq\norm{x-z+\zeta_{1}+\zeta_{2}}$. Then, since the Euclidean norm satisfies the triangle inequality we obtain $\norm{z - v + \zeta_{3}}\leq \norm{x-y+\zeta_{1}}+\norm{y-z+\zeta_{2}}$.
\end{proof}

For any $U\subset \mathbb{T}^{n}$, we use the notation $\dim_{H}U$ and $\dim_{B}U$ to respectively refer to the Hausdorff dimension and box-counting dimension derived using the metric in Lemma~\ref{lem:natMet}. 

\begin{definition} 
Let $A$ be an invertible $n$ by $n$ matrix with integer entries. A nonempty closed subset $X \subset \mathbb{T}^{n}$ is said to be \textit{$\times A$-invariant} if $AX :=\{Ax:x\in X\} \subset X$. We say that $X$ is \textit{multiplicatively invariant} if it is $\times A-invariant$ for some matrix $A\in M_{n}(\mathbb{Z})$. 
\end{definition}

The operation $Ax$ where $A$ is an $n$ by $n$ matrix with integer entries and $x\in\mathbb{T}^{n}$ is a shorthand for the function composition $x = y + \mathbb{Z}^{n} \mapsto Ay \mapsto \varphi(Ay)$. This procedure does not depend on the choice of coset representative because if $y$ and $z$ are elements of $\mathbb{R}^{n}$ that differ by an element of $\mathbb{Z}^{n}$, then $Ay$ and $Az$ also differ by an element of $\mathbb{Z}^{n}$. 

\begin{remark} 
Furstenberg and Berend give a more general definition called $\Sigma$-invariance. That is, invariance under a semigroup of endomorphisms on the torus. 
\end{remark}

We include some examples. 

\begin{example}
Let $A\in M_{n}(\mathbb{Z})$ be an expanding matrix and let $D\subset\mathbb{Z}^{n}$. Let $\varphi:\mathbb{R}^{n} \rightarrow \mathbb{T}^{n}$ be the canonical homomorphism given by $\varphi(v) = v + \mathbb{Z}^{n}$. The image $\varphi(T_{A, D})$ is $\times A$-invariant. For any $v\in T_{A, D}$, 
\begin{align}
Av+\mathbb{Z}^{n} &= A( \pi_{A, D}(v_{j})_{j=1}^{\infty}) + \mathbb{Z}^{n} \\
&= v_{1} + \pi_{A, D}(v_{j+1})_{j=1}^{\infty} + \mathbb{Z}^{n} \\
&= \pi_{A, D}(v_{j+1})_{j=1}^{\infty} + \mathbb{Z}^{n}. 
\end{align}
\end{example}

\begin{example}
Let $A\in M_{n}(\mathbb{Z})$ be an expanding matrix,  let $D = \{u, v\}$ where $u$ and $v$ are distinct elements of $\mathbb{R}^{n}$, and let $\varphi$ be the canonical homomorphism from $\mathbb{R}^{n} \rightarrow \mathbb{T}^{n}$. The set $\pi_{A ,D}(\{(d_{j})_{j=1}^{\infty}\in D^{\mathbb{N}}: d_{j} = v \;\;\text{implies}\;\; d_{j+1} = u\})$ under the image of $\varphi$ is $\times A$-invariant.
\end{example}

We now turn our attention to $\mathbb{Z}^{n}$. When working in the nonnegative integers, the maps $\phi_{r}$ and $\psi_{r}$ are defined for all nonnegative integers for any choice of nonnegative integer $r$ greater than or equal to $2$. This is because every nonnegative integer has a base-$r$ expansion. To replicate this with respect to an $n$ by $n$ matrix with integer entries, we need to choose a set of coefficients $D\subset\mathbb{Z}^{n}$. 

\begin{definition} 
Let $A\in M_{n}(\mathbb{Z})$ and let $D\subset\mathbb{Z}^{n}$ be finite. We call an expansion of the form 
\begin{equation}
 d_{0} + Ad_{1} + \cdots + A^{m-1}d_{m-1}, \;m\in\mathbb{Z}_{\geq1}
\end{equation}
a \textit{discrete $(A, D)$-expansion}. We call the elements of $D$ \textit{digits}. 
\end{definition}

We assume that $D$ is finite because we want to leverage the boundedness of $\mathcal{D}$ in a way that is similar to when we were examining $T_{A, D}$ when $A$ was an expanding matrix. However, for certain pairs $(A, D)$, it may not be true that every element of $\mathbb{Z}^{n}$ has a discrete $(A ,D)$-expansion. Considering the previous chapters, it is not unreasonable to guess that if $A$ is expanding and $D$ is a complete residue system mod $A$, then all elements of $\mathbb{Z}^{n}$ have discrete $(A, D)$-expansions. This is not the case. 

Let $(A, \mathcal{D})$ be the pair
\begin{equation}
\Bigg(\begin{bmatrix} 2 & 0 \\ 0 & 2 \end{bmatrix}, \Bigg\{\begin{bmatrix} 0 \\ 0 \end{bmatrix}, \begin{bmatrix}1 \\0 \end{bmatrix}, \begin{bmatrix}0\\1 \end{bmatrix}, \begin{bmatrix} 1 \\ 1 \end{bmatrix}\Bigg\}\Bigg).
\end{equation}
The vector $[-1, 1]^{t}$ does not have a discrete $(A, \mathcal{D})$-expansion because the image of any element of $D$ after multiplication by a power of $A$ has nonnegative entries. The component $-1$ is not reachable. Meanwhile, $A$ has eigenvalues $\pm 2$ and every element $v\in\mathbb{Z}^{n}$ is equal to $Au + d$ for some $u\in\mathbb{Z}^{n}$ and $d\in \mathcal{D}$. For example, 
\begin{equation}
\begin{bmatrix} -1 \\ 1 \end{bmatrix} = \begin{bmatrix} 2 & 0 \\ 0 & 2 \end{bmatrix} \begin{bmatrix} -1 \\ 0 \end{bmatrix} + \begin{bmatrix} 1 \\ 1 \end{bmatrix}.
\end{equation}

\begin{definition} 
Let $A\in\mathbb{Z}^{n}$ be an expanding matrix and let $\mathcal{D}$ be a complete residue system mod $A$ containing $0$. We call the pair $(A, \mathcal{D})$ a \textit{number system for $\mathbb{Z}^{n}$} if every element of $\mathbb{Z}^{n}$ has a discrete $(A, \mathcal{D})$-expansion. 
\end{definition}

We provide conditions for when an expanding matrix $A\in M_{n}(\mathbb{Z})$ is a complete residue system mod $A$ form a number system of $\mathbb{Z}^{n}$. 

\begin{theorem}\label{thm:radixExist} 
Suppose $A \in M_{n}(\mathbb{Z})$ is an expanding matrix and that $\mathcal{D} \subset \mathbb{Z}^{n}$ is a complete set of remainders mod $A$ containing the origin. The pair $(A, \mathcal{D})$ is a number system for $\mathbb{Z}^{n}$ if and only if the set $T_{A, \mathcal{D}} \cap \mathbb{Z}^{n}$ only contains $0$. 
\end{theorem}

To prove Theorem~\ref{thm:radixExist} we require two lemmas. The first is that a vector has a discrete $(A, D)$-expansion if and only if a version of Euclidean division terminates. 

\begin{definition} 
Let $A$ be an $n$ by $n$ matrix with integer entries and let $\mathcal{D}$ be a complete residue system mod $A$. For $v\in\mathbb{Z}^{n}$ we call the sequence $(v_{k})_{k=1}^{\infty}$ given by $v_{k-1} = Av_{k} + e_{k-1}$ where $e_{k-1} \in \mathcal{D}$ and $v_{0} = v$, the \textit{$(A, D)$-remainder sequence of $v$}. 
\end{definition}

The remainder sequence is well-defined. If $\mathcal{D}$ is a complete residue system mod $A$, then the canonical function from $\mathcal{D}$ to $\faktor{\mathbb{Z}^{n}}{A\mathbb{Z}^{n}}$ is bijective. This implies that the choice $e_{k-1}$ such that $v_{k-1}$ is an element of $e_{k-1} + A\mathbb{Z}^{n}$ is unique, and thus $v_{k}$ is uniquely determined.

\begin{lemma}\label{lem:euclideanAlgo} 
Suppose that $A \in M_{n}(\mathbb{Z})$ is an expanding matrix and that $\mathcal{D} \subset \mathbb{Z}^{n}$ is a complete set of residues mod $A$. An element $v\in\mathbb{Z}^{n}$ has a discrete $(A, D)$-expansion if and only if its $(A, D)$-remainder sequence is eventually zero. 
\end{lemma}

\begin{proof}
Suppose there exists some positive integer $K$ such that $v_{k} = 0$ for all $k\geq K$. Therefore $v_{K-1} = e_{K-1}\in\mathcal{D}$ and $v = v_{0} = A^{K-1}e_{K-1} + A^{K-2}e_{K-2} + \cdots + Ae_{1} + e_{0}$. Conversely, if $v =  d_{0} + Ad_{1} + \cdots + A^{m}d_{m}$, where $d_{j}\in\mathcal{D}$ for each $j$, then we have $v_{k} = 0$ for all $k\geq m+1$. 
\end{proof}

The next lemma is established as a means to show that the $(A, D)$-remainder sequence is eventually periodic. 

\begin{lemma}\label{lem:periodicAlgo} 
Suppose that $A \in M_{n}(\mathbb{Z})$ is an expanding matrix and that $\mathcal{D}$ is a complete set of remainders mod $A$. Let $g:\mathbb{Z}^{n}\rightarrow\mathbb{Z}^{n}$ be the function given by
\begin{equation}
g(x) = A^{-1}(x - d(x))
\end{equation}
where $d(x)$ is the unique element of $\mathcal{D}$ such that $x\in d(x)+\mathbb{Z}^{n}$. For every $v\in\mathbb{Z}^{n}$, the sequence $(g^{j}(v))_{j=1}^{\infty}$ is eventually periodic. 
\end{lemma}

\begin{proof}
Let $1< \rho < \min_{\lambda\in\sigma(A)}|\lambda|$ and $R := (\max\{\norm{d} : d\in\mathcal{D}\})/(\rho - 1)$. Recall that the function $\norm{x}_{A} := \sum\limits_{j=1}^{\infty} \rho^{j}\norm{A^{-j}x}$ on $\mathbb{R}^{n}$ is a norm where $\norm{\cdot}$ is the Euclidean norm. By Lemma~\ref{lem:contractingNorm}, multiplication by $A^{-1}$ is a contraction with coefficient $1/\rho$ with respect to $\norm{\cdot}_{A}$.  Observe that if $v\mathbb{Z}^{n}$ is in the complement of the closed ball $\overline{B_{R}(0)}$ (relative to $\norm{\cdot}_{A}$), then
\begin{align}
\norm{g(v)}_{A} &= \norm{A^{-1}(v-d(v))}_{A} \\
&\leq \rho^{-1}(\norm{v}_{A} + \norm{d(v)}_{A}) \\
&\leq \rho^{-1}(\norm{v}_{A} + (\rho - 1)R) \\
&< \norm{v}_{A}
\end{align}

Therefore the sequence $(g^{j}(v))_{j=1}^{\infty}$, visits $\overline{B_{R}(0)}$ infinitely often. Since there are only finitely many elements of $\mathbb{Z}^{n}$ in $\overline{B_{R}(0)}$, we conclude that the sequence $(g^{j}(v))_{j=1}^{\infty}$ is eventually periodic. 
\end{proof}

We are now in a position to prove Theorem~\ref{thm:radixExist}. 

\begin{proof}[Proof of Theorem~\ref{thm:radixExist}]
By Lemma~\ref{lem:euclideanAlgo}, it suffices to prove that every sequence of remainders is eventually zero if and only if $0$ is the only element of $T_{A, \mathcal{D}} \cap \mathbb{Z}^{n}$. 

Suppose for some $v_{0}\in\mathbb{Z}^{n}$, the sequence $(v_{k})_{k=1}^{\infty}$ takes on non-zero values for infinitely many indices $k$. This is the same as $v_{k} \neq 0$ for all $k$, because if a sequence of remainders is zero for some index $K$, it is zero for all $k>K$.   For each $d\in\mathcal{D}$, let $f_{d}:\mathbb{R}^{n}\rightarrow\mathbb{R}^{n}$ be given by $f_{d}(x) = B^{-1}(x-d)$. We can express $v_{k}$ as the image of $v_{0}$ under a composition of these functions in the following way
\begin{equation}
v_{k} = (f_{d_{k-1}}\circ  f_{d_{k-1}} \circ \cdots \circ f_{d_{0}})(v_{0}). 
\end{equation}
Therefore $v_{k} = g^{k}(v_{0})$ where $g$ is the function in Lemma~\ref{lem:periodicAlgo}. Therefore $(v_{k})_{k=0}^{\infty}$ is eventually periodic. Let $(e_{1}, e_{2}, \ldots, e_{p})$ be the sequence of digits in $\mathcal{D}$ that repeat (in the order of repetition). Let $(v_{k_{i}})_{i=1}^{\infty}$ be a subsequence with the property that for all sufficiently large $i$, $v_{k_{i+1}} = (f_{e_{p}} \circ f_{e_{p-2}} \circ \cdots \circ f_{e_{1}})(v_{k_{i}})$. 

Let $w$ denote the value of  $-\sum\limits_{j=1}^{\infty}A^{-jp}(\sum\limits_{i=1}^{p}A^{i-1}e_{i})$. Let $(g_{i})_{i=1}^{\infty}$ be a subsequence of the sequence of function compositions $(f_{d_{k-1}} \circ \cdots \circ f_{d_{0}})_{k=1}^{\infty}$ that corresponds to the functions used to produce $(v_{k_{i}})_{i=1}^{\infty}$. The sequence $(g_{i}(0))_{i=1}^{\infty}$ converges to $w$.There exists a constant $C > 0$ such that
\begin{equation}
\norm{g^{k}(v_{0}) - (f_{d_{k-1}}\circ  f_{d_{k-1}} \circ \cdots \circ f_{d_{0}})(v_{0})} \leq C\rho^{-k}\norm{v_{0}}
\end{equation} 
for \emph{any} $k$ because $A^{-1}$ is a contraction with respect to an equivalent norm. This implies that $(v_{k_{i}})_{i=1}^{\infty}$ converges to $w$. 

Therefore $w = v_{k_{i}}$ for sufficiently large $i$ and is an element of $\mathbb{Z}^{n}$. The argument is complete since $(-w)\in T_{A, \mathcal{D}}$. 

Conversely, suppose every element of $\mathbb{Z}^{n}$ has a discrete $(A, \mathcal{D})$-expansion. By way contradiction, suppose $\zeta$ is an element of $T_{A, \mathcal{D}} \cap \mathbb{Z}^{n}$ and is not equal to zero. By definition $\zeta$ has the $(A, \mathcal{D})$-representation (not to be confused with discrete $(A, \mathcal{D})$-expansions) $(t_{j})_{j=1}^{\infty}$. The vector $-\zeta$ has a discrete $(A, \mathcal{D})$-expansion, which we denote by
$-\zeta = A^{m}d_{m} + A^{m-1}d_{m-1} \ldots + d_{0}$. By multiplying $-\zeta + \pi_{A, \mathcal{D}}(t_{j})_{j=1}^{\infty}$ by successive powers of $A$, we have that $0$ is an element of 
\begin{equation}
(A^{m+k}d_{m} +\cdots A^{k}d_{0} A^{k-1}t_{1} + \cdots + t_{k}) + T_{A, \mathcal{D}}.
\end{equation}
Since $\mathcal{D}$ is a complete residue system mod $A$, each $A^{m+k}d_{m} +\cdots A^{k}d_{0} A^{k-1}t_{1} + \cdots + t_{k}$ is a distinct element of $\mathbb{Z}^{n}$. This is a contradiction because $T_{A, \mathcal{D}}$ only has finitely many neighbours. 
\end{proof}

We now define $\times (A, \mathcal{D})$-invariance when $(A, \mathcal{D})$ is a number system for $\mathbb{Z}^{n}$. 

\begin{definition} 
Let $(A, \mathcal{D})$ be a number system for $\mathbb{Z}^{n}$. Let $\phi_{A, \mathcal{D}}$ denote the function from $\mathbb{Z}^{n}$ to $\mathbb{Z}^{n}$ given by 
\begin{equation}
g = d_{0} + Ad_{1} + \cdots + A^{\ell}d_{\ell} \mapsto d_{1} + Ad_{2} + \cdots + A^{\ell-1}d_{\ell}.
\end{equation}
Similarly, let $\psi_{A, \mathcal{D}}$ denote the function from $\mathbb{Z}^{n}$ to $\mathbb{Z}^{n}$ given by
\begin{equation}
g = d_{0} + Ad_{1} + \cdots + A^{\ell}d_{\ell} \mapsto d_{0} + Ad_{1} + \cdots + A^{\ell-1}d_{\ell-1}.
\end{equation}
A nonempty subset $E\subset\mathbb{Z}^{n}$ is \textit{$\times (A, \mathcal{D})$-invariant} if both $\phi_{A, \mathcal{D}}(A)$ and $\psi_{A, \mathcal{D}}(E)$ are subsets of $E$. 
\end{definition}

This notion of invariance is dependent on the choice of $\mathcal{D}$. Let us consider $\times r$ invariance in $\mathbb{N}\cup\{0\}$. The defining maps are given in \cite{GMR24} as $\phi_{r}(n) = \lfloor n/r \rfloor$ and $\psi_{r}(n) = n - r^{k}\lfloor n/r^{k} \rfloor$. These look like maps that do not make a choice of digits but implicitly they do. The properties of these two maps that are leveraged in the proofs of GMR are precisely the way they transform a base-$(r, \{0, 1, \ldots, r-1\})$ expansion into another base-$(r, \{0, 1, \ldots, r-1\})$ expansion. We suspect the ``right" perspective is that they \emph{happen} to be expressible using the floor function. We now bridge multiplicative invariance on $\mathbb{Z}^{n}$ to multiplicative invariance on $\faktor{\mathbb{R}^{n} }{\mathbb{Z}^{n}}$ in the same way.

\begin{theorem}\label{thm:invTransfer} 
Assume that $(A, \mathcal{D})$ is a number system. If $E$ is $\times (A, \mathcal{D})$-invariant, then the sequence of sets $(A^{-k}(E \cap A^{k}T_{A, \mathcal{D}}))_{k=1}^{\infty}$ converges to $X \in \mathbb{R}^{n}$ in Hausdorff distance. Moreover, the image $\varphi(X)$ is $\times A$-invariant where $\varphi$ is the canonical homomorphism from $\mathbb{R}^{n}$ to $\mathbb{T}^{n}$. 
\end{theorem}

Our goal for the remainder of this section is to prove this theorem. We first introduce some terminology and establish a sequence of technical lemmas. We follow the strategy employed in \cite{GMR24} for the nonnegative integer case. 

\begin{definition} 
Let $(X, d)$ be a metric space. For any $C\subset X$ and $\delta > 0$, we call the set $\{x\in X : d(x, y) < \delta\;\;\textit{for some}\;\;y\in C\}$ the \textit{$\delta$-neighbourhood of $C$} and denote it by $[C]_{\delta}$. 
\end{definition}

\begin{definition} \label{def:hausDist}
Let $(X, d)$ be a metric space. For $A, B\subset X$, we call the quantity
\begin{equation}
d_{H}(A, B) := \inf\{\delta: A\subset [B]_{\delta}, B\subset [A]_{\delta}\}
\end{equation}
the \textit{Hausdorff distance between $A$ and $B$}.
\end{definition}

\begin{remark}
Typically, the Hausdorff distance between $A$ and $B$ is defined as $\max(\sup\limits_{a\in A} \dist(a, B), \sup\limits_{b\in B} \dist(A, b))$. This definition and the one given in Definition~\ref{def:hausDist} are equivalent. 
\end{remark}

\begin{lemma}\label{lem:phiBound} 
Suppose that $(A, \mathcal{D})$ is a number system. For any $\zeta\in\mathbb{Z}^{n}$, positive integers $\ell, k$ with $\ell \geq k$, we have
\begin{equation}
\norm{A^{-\ell}\zeta - A^{-k}\phi_{A, \mathcal{D}}^{\ell-k}(\zeta)} \leq C\rho^{-k}/(\rho-1),
\end{equation}
where $\rho$ is a fixed number in $(1, \min_{\lambda\in\sigma(A)}|\lambda|)$ and $C$ is a constant that depends on $(A, \mathcal{D})$. 
\end{lemma}
\begin{proof}
The element $\zeta$ has a discrete $(A, \mathcal{D})$-expansion: $$d_{0} + Bd_{1} + \cdots + A^{m-1}d_{m-1}.$$ The image of $\zeta$ under $\phi_{A, \mathcal{D}}$ after $\ell-k$ applications is $\phi_{A, \mathcal{D}}^{\ell-k}(\zeta) = d_{\ell-k}+Ad_{\ell-k+1} + \cdots + A^{m-1-(\ell-k)}d_{m-1}$. It follows that 
\begin{equation}
A^{-k}\phi_{A, \mathcal{D}}^{\ell-k}(\zeta) = A^{-k}d_{\ell-k} + A^{-k+1}d_{\ell-k+1} + \cdots + A^{m-1-\ell}d_{m-1}.
\end{equation}
Meanwhile,
\begin{equation}
A^{-\ell}\zeta = A^{-\ell}d_{0} + A^{-\ell + 1}d_{1} + \cdots + A^{m-1-\ell}d_{m-1}. 
\end{equation}
Therefore by Lemma~\ref{lem:contractingNorm} and the equivalence of norms on $\mathbb{R}^{n}$,
\begin{align}
\norm{A^{-\ell}\zeta - A^{-k}\phi_{A, \mathcal{D}}^{\ell-k}(\zeta)}&= \norm{A^{-\ell} d_{0} + A^{-\ell + 1}d_{1} + \cdots + A^{-k-1}d_{\ell-k-1}} \\
&\leq \max_{d\in\mathcal{D}}\norm{d} \sum_{j=k+1}^{\ell}\rho^{-j} \\
&\leq C\rho^{-(k+1)}\max_{d\in\mathcal{D}}\norm{d} \sum_{j=0}^{\infty}\rho^{-j}\\
&= \frac{\max_{d\in\mathcal{D}}\norm{d}}{\rho-1}\rho^{-k}. 
\end{align}
\end{proof}

\begin{lemma}\label{lem:phiNeighbourhood} 
Suppose that $(A, \mathcal{D})$ is a number system and that $E$ is a subset of $\mathbb{Z}^{n}$. For each integer $k\geq1$, define $X_{k} := A^{-k}(E\cap A^{k}T_{A, \mathcal{D}})$. Then the following implications hold where $\rho \in (1, \min_{\lambda\in\sigma(A)}|\lambda|)$ and $C$ is a constant that depends on $A$ and $\mathcal{D}$. 
\begin{itemize}
\item[(i)] If $\phi_{A, \mathcal{D}}(E)\subset E$, then for all $k, \ell\in\mathbb{Z}_{>0}, \ell\geq k, X_{\ell}\subset [X_{k}]_{C\rho^{-k}/(\rho-1)}$. 
\item[(ii)] If $E\subset \phi_{A, \mathcal{D}}(E)$, then for all $k, \ell\in\mathbb{Z}_{>0}, \ell\geq k, X_{k}\subset [X_{\ell}]_{C\rho^{-k}/(\rho-1)}$. 
\end{itemize}
Thus, if $\phi_{A, \mathcal{D}}(E) = E$ then for all $\ell\geq k$, we have $d_{H}(X_{\ell}, X_{k})\leq C\rho^{-k}/(\rho-1)$. 
\end{lemma}

\begin{proof}
In both parts we assume that $\ell$ and $k$ are positive integers with $k \leq \ell$.

We begin with (i). Let $y$ be an element of $X_{\ell}$. By definition, $y = A^{-1}v$, with $v\in A\cap B^{\ell}T_{B, \mathcal{D}}$. If $\phi_{A, \mathcal{D}}(E) \subset E$, then $\phi_{A, \mathcal{D}}^{i}(E\cap B^{\ell}T_{B, \mathcal{D}}) = E\cap A^{\ell-i}T_{A, \mathcal{D}}$ for some positive integer $i$. It follows that $\tilde{v} := \phi_{A, \mathcal{D}}^{\ell-k}(v)$ is an element of $E \cap A^{\ell - (\ell-k)}T_{A, \mathcal{D}} = E \cap A^{k}T_{A, \mathcal{D}}$. Therefore $\tilde{y} := A^{-k}\tilde{v}$ is an element of $X_{k}$. A direct application of Lemma~\ref{lem:phiBound} yields $\norm{y - \tilde{y}} \leq C\rho^{-k}$. 

We now prove (ii). Suppose $x\in X_{k}$. By definition, $x = A^{-\ell}w$, with $w\in E\cap A^{k}T_{A, \mathcal{D}}$. If $A \subset \phi_{A, \mathcal{D}}(A)$, then there exists $\tilde{w}\in A$ such that $\phi_{A, \mathcal{D}}^{\ell-k}(\tilde{w}) = w$. The highest power of $A$ in the discrete $(A, \mathcal{D})$-expansion of $w$ is less than or equal to $k-1$. Since the expansion is unique, it must be that $\tilde{w}$ is an element of $A^{\ell}T_{A, \mathcal{D}}$. Therefore $A^{-\ell}\tilde{w}$ is an element of $\ell$. Now we apply Lemma~\ref{lem:phiBound} to obtain the desired inequality. 

If $\phi_{A, \mathcal{D}}(E) = E$ then the conclusions of both (i) and (ii) hold. The desired result now follows from the definition of Hausdorff distance. 
\end{proof}

\begin{lemma}\label{lem:cauchyTrick} 
Suppose $(A, D)$ is a number system for $\mathbb{Z}^{n}$. If $E\subset\mathbb{Z}^{n}$ satisfies $\phi_{A, \mathcal{D}}(E)\subset E$, then $\lim_{k\rightarrow\infty}d_{H}(X_{k}, X_{k}^{'}) = 0$ where $X_{k} = A^{-k}(E\cap A^{k}T_{A, \mathcal{D}})$ and $X_{k}^{'} = A^{-k}(E^{'}\cap A^{k}T_{A, \mathcal{D}})$ with $E^{'} = \bigcap\limits_{k=1}^{\infty}\phi_{A, \mathcal{D}}(E)$. 
\end{lemma}

\begin{proof}
Let $C$ be the constant appearing in Lemma~\ref{lem:phiNeighbourhood}. Let $\varepsilon>0$ and choose a positive integer $m$ large enough that $2\rho^{-m}/(\rho-1) < \varepsilon/C$. 

The property $\phi_{r}(E)\subset E$ implies that the sequence of sets $(\phi_{A, \mathcal{D}}^{j}(E) \cap A^{m}T_{A, D})_{j=1}^{\infty})$ is a decreasing sequence of sets (in terms of set inclusion). This sequence must stabilize because the sets are finite. That is, for all sufficiently large $k$, $\phi_{A, \mathcal{D}}^{k}(E) \cap A^{m}T_{A, \mathcal{D}} = E^{'}\cap A^{m}T_{A, \mathcal{D}}$. In fact, for any $m$, we can choose $k\geq m$ such that 
\begin{equation} \label{eq:stabilize}
\phi_{A, \mathcal{D}}(E)^{k-m}\cap = E^{'}\cap A^{m}T_{A, \mathcal{D}}. 
\end{equation}

We have from Lemma~\ref{lem:phiBound} that 
\begin{equation}
X_{k} \subset \bigg[A^{-m}(\phi_{A, \mathcal{D}}^{k-m}(E)\cap A^{m}T_{A, \mathcal{D}})\bigg]_{C\rho^{-m}/(\rho-1)}. 
\end{equation}
By (\ref{eq:stabilize}), this is precisely $X_{k}\subset [X_{m}^{'}]_{C\rho^{-m}/(\rho-1)}$. We now make use of the fact that $\phi_{A, \mathcal{D}}(E^{'}) = E^{'}$. By Lemma \ref{lem:phiBound}, we have $d_{H}(X_{k}^{'}, X_{m}^{'}) < C\rho^{-m}/(\rho-1)$ for $k\geq m$. Therefore $[X_{m}^{'}]_{C\rho^{-m}/(\rho-1)}\subset [X_{k}^{'}]_{2C\rho^{-m}/(\rho-1)}$. On the other hand, $X_{k}^{'} \subset X_{k}$ because $E^{'}\subset E$. Therefore $X_{k}^{'}\subset [X_{k}]_{2C\rho^{-m}/(\rho-1)}$ and $X_{k}\subset [X_{k}^{'}]_{2C\rho^{-m}/(\rho-1)}$. Therefore $d_{H}(X_{k}, X_{k}^{'}) < \varepsilon$. 
\end{proof}

We now prove Theorem~\ref{thm:invTransfer}. 

\begin{proof}[Proof of Theorem~\ref{thm:invTransfer}]
By Lemma~\ref{lem:cauchyTrick}, the sequence $(X_{k})_{k=1}^{\infty}$ converges if and only if $(X_{k}^{'})_{k=1}^{\infty}$ converges. We then appeal to Lemma~\ref{lem:phiNeighbourhood} to observe that $(X_{k}^{'})_{k=1}^{\infty}$ is a cauchy sequence of nonempty compact sets. This is sufficient for convergence because the collection of nonempty compact subsets of Euclidean space equipped with the Hausdorff distance is a complete metric space \cite{DD02}. Therefore $(X_{k})_{k=1}^{\infty}$ converges to a nonempty compact subset of $\mathbb{R}^{n}$ in Hausdorff distance. 

To finish the proof, we verify that $\varphi(X)$ is $\times A$-invariant. Every point in $X_{k}$ has an $(A, \mathcal{D})$-representation of the form $(d_{1}, \ldots, d_{k}, \overline{0})$. Therefore the elements of $A\varphi(X_{k})$ have representatives whose $(A, \mathcal{D})$-representations have the form  $(d_{2}, \ldots, d_{k}, \overline{0})$. Therefore $A\varphi(X_{k})$ is a subset of $\varphi(A^{-k}(\mathbb{Z}^{n} \cap A^{k-1}T_{A, \mathcal{D}}))$. Recall that part of the assumption that $E$ is $\times (A, \mathcal{D})$-invariant is that $\psi_{A, \mathcal{D}}(E)$ is a subset of  $E$. This means that whenever $A^{k-1}d_{1} + A^{k-2}d_{2} +\cdots+d_{k}$ is in $E$, $A^{k-2}d_{2} + \cdots + d_{k}$ is also in $E$. It follows that $A\varphi(X_{k})$ is contained in $\varphi(A^{-(k-1)}E)$. We conclude that $A\varphi(X_{k}) \subset \varphi(X_{k-1})$. Both multiplication by $A$ on $\mathbb{T}^{n}$ and $\varphi$ are uniformly continuous functions. We conclude that $A\varphi(X) = \lim_{k\rightarrow\infty}A\varphi(X_{k})$ is a subset of $\lim_{k\rightarrow\infty}\varphi(X_{k-1}) = \varphi(X)$. 
\end{proof}

\chapter{Conclusion and Future Work}\label{chp:mirai}

This final chapter is organized into four sections discussing potential future work. The first two parts involve questions about the intersection of attractors and their translates, and the latter two are about multiplicative invariance. The first of each pair present opportunities that look to be within an arm's reach. The remaining two contain problems that may require a ladder.

\section{Digit Sets with Uniform Gaps} 

Suppose that $A\in M_{n}(\mathbb{R})$ is an expanding matrix and suppose $D\subset\mathbb{R}^{n}$ is finite. The self-similarity of an intersection $T(\alpha) := T_{A, D} \cap (T_{A, D} + \alpha)$ was shown to imply that the sequence of sets $(D\cap(D+\alpha_{j}) - \beta_{j})_{j=1}^{\infty}$ is SEP for some sequence $(\beta_{j})_{j=1}^{\infty} \in D^{\mathbb{N}}$ under two assumptions. These were the conditions that all $(A, D-D)$-representations are unique and that the IFS generating the intersection exclusively contains functions of the form $f(x) = A^{-p}x + u$ where $p$ is a positive integer and $u$ is some element of $\mathbb{R}^{n}$ (Theorem~\ref{thm:attractSEP}). Suppose we restrict to the special case that $A$ corresponds to multiplication by the Gaussian integer $-n+i$ and $D$ is equal to $\{0, m\}$ where $2\leq m \leq n^{2}$. We showed that if $T(\alpha)$ is the attractor of a collection of similarities containing $g(x) = rx + (1-r)\gamma$, where $\gamma = \pi_{A, D}(\gamma_{j})_{j=1}^{\infty}, \gamma_{j} = \min{D\cap(D+\alpha_{j})}$ for each $j$, then $(m - |\alpha_{j}|)_{j=1}^{\infty}$ is an SEP sequence of integers (Theorem~\ref{thm:selfSimSEPSp}). This raises the following question: can we remove either of the assumptions that all $(A, D-D)$-representations are unique or that the elements of the IFS have the form $A^{-p}x + u$?

In Example~\ref{ex:nonSEP} it was shown that there exist self-similar sets of the form $\pi_{A, D}\left(\prod\limits_{j=1}^{\infty}D_{j}\right)$ for which $(D_{j} - \beta_{j})$ is not SEP. It is suggested that the uniqueness of $(A, D-D)$-expansions cannot be removed without restricting to a subclass of the attractors $T_{A, D}$. On the other hand, theorem 1.2 in \cite{LYZ11} states that the self-similarity of $T_{r, D} + (T_{r, D} + \alpha)$ where $r$ is a positive integer and $D$ is a $k$-length arithmetic progression with step $m$ holds exactly when $(k-1-\alpha_{j}/m)$ is an SEP sequence of integers. The proof does not use the uniqueness of $(r, D-D)$-expansions (although it may hold for all sufficiently large $m$). The following conjecture is a generalization Theorem~\ref{thm:selfSimSEPSp} inspired by theorem 1.2 in \cite{LYZ11}.

\begin{conjecture}\label{thm:selfSimSEPSpExt} 
Suppose that $A\in M_{n}(\mathbb{R})$ is invertible and multiplication by its inverse is a similarity with respect to Euclidean distance. Suppose $D$ is a $k$-length arithmetic progression with step $m$ and that $\alpha$ has a unique $(A, D-D)$-representation. Let $\gamma$ be the minimal element of $T(\alpha) := T_{A, D} \cap (T_{A, D} + \alpha)$. If $T(\alpha)$ is the attractor of a collection of similarities containing a function of the form $f(x) = rx + (1 - r)\gamma$ for some $r\in\mathbb{C}$, then $((k-1)-\alpha_{j}/m)_{j=1}^{\infty}$ is SEP. 
\end{conjecture}

The challenge of moving from, say, $D = \{0, m\}$ to $D = \{0, m, 2m, \ldots, (k-1)m\}$ lies in surmounting the lack of an ordering on $\mathbb{R}^{n}$ and the behaviour of the $(A, D)$-representations. 

\section{Relaxing Homogeneity} 

The paper \cite{ZLY11} stands out from the existing literature on the self-similarity of $T\cap(T+\alpha)$ because it treats a self-similar set generated by a nonhomogeneous collection of similarities. An IFS is homogeneous if all the contraction coefficients of its elements are equal. Suppose $A$ and $B$ are distinct expanding matrices and that $T$ is the attractor of $\{f(x) = A^{-1}x + u, g(x) = B^{-1}x + v\}$. For which $\alpha \in T-T$ is the intersection $T\cap (T+\alpha)$ self-affine? What can be said about the function $\alpha \mapsto \dim_{H}(T\cap(T+\alpha))$?

\section{Equating Discrete and ``Continuous" Fractal Dimensions} 

Theorem~\ref{thm:invTransfer} states that if $E\subset\mathbb{Z}^{n}$ is a $\times (A, \mathcal{D})$-invariant set, then the sequence of sets of the form $A^{-k}(E\cap A^{k}T_{A, \mathcal{D}})$ converges to $X\subset\mathbb{R}^{n}$ such that $\varphi(X) := \bigcup\limits_{x\in X}(x + \mathbb{Z}^{n})$ is a $\times A$-invariant subset of $\mathbb{T}^{n}$. In \cite{GMR24}, it is shown that the $r^{-k}(F \cap [0, r^{k})$ converges to $Y \subset [0, 1]$ if $F$ is a $\times r$-invariant subset of the nonnegative integers, $\mathbb{N}_{0}$. The set $Y$, as a subset of $\mathbb{T}$ is $\times r$-invariant and its box-counting dimension agrees with a discrete version of box-counting dimension first introduced in \cite{BT89}. We present this version of dimension, called mass dimension:

\begin{definition}
Let $F$ be a subset of $\mathbb{N}_{0}$. We call the quantity
\begin{equation}\label{eq:massDim}
\lim_{N\rightarrow\infty}\frac{\log(|F\cap[0, N)|)}{\log(N)}
\end{equation}
the \textit{mass dimension of $F$} provided the limit exists and denote it by $\dim_{M}F$. 
\end{definition}

A reasonable generalization to subsets of $\mathbb{Z}^{n}$ is to use either squares or balls centered at the origin. For example, we could call the limit of 
\begin{equation}
\lim_{N\rightarrow\infty}\frac{\log(|E\cap (-N, N)^{n}|)}{\log(N)}
\end{equation}
the mass dimension of $E\subset\mathbb{Z}^{n}$. This is the version given in \cite{BT89}. In the particular case of a multiplicatively invariant set $E$, which yields a multiplicative invariant set $\varphi(X) \subset \mathbb{T}^{n}$, we want to tailor this definition to multiplicatively invariant sets in the same way that we measured the box-counting dimension of subsets of $T_{A, \mathbb{D}}$ using $k$-tiles (Corollary~\ref{lem:boxtile2}). In \cite{GMR24}, the main idea is to swap $(N)_{N=1}^{\infty}$ in (\ref{eq:massDim}) with the subsequence $(r^{k})_{k=1}^{\infty}$ when $F\subset\mathbb{N}_{0}$ is $\times r$-invariant. Observe that $\mathbb{N}_{0}\cap [0, r^{k})$ is equal to the set of nonnegative integers with discrete $(r, \{0, 1, \ldots, r-1)\})$-expansions $d_{0} + d_{1}r + \cdots + d_{k-1}r^{k-1}$. This is the set of nonnegative integers contained in $r^{k}(T_{r, \{0, 1, \ldots, r-1\}}\setminus\{1\})$. While this notation is superfluous in one dimension (for all $r$, $T_{r, \{0, 1, \ldots, r-1\}} = [0, 1])$, it is not when we consider number systems on $\mathbb{Z}^{n}$. Given $(A, \mathcal{D})$, the suggested approach is to measure the mass dimension using the ``neighbourhoods" $\mathbb{Z}^{n} \cap A^{-k}T_{A, \mathcal{D}}$. 

There is a snag. Our definition of $\times (A, D)$-invariance allows for matrices that are not scalings of orthogonal matrices (norm preserving). The following theorem of in \cite{J23} shows that there exist subsets of $\mathbb{T}^{2}$ which are invariant under multiplication by matrices of this kind, but do not have a box-counting dimension to speak of. We do not think it unreasonable that examples could be extrapolated to $\mathbb{T}^{n}$, $n\geq2$. 

\begin{theorem} 
There exists integers $n > m\geq2$ and a compact subset $F \subset \mathbb{T}^{2}$ such that $T(F) = F$ where $T(x, y) = (mx, ny)$ and $\underline{\dim}_{B}F < \overline{\dim}_{B}F$. 
\end{theorem}

This a case when $T_{A, \mathcal{D}}$ is strictly self-affine since the transformation $T$ corresponds to the matrix $\begin{bmatrix} m & 0 \\ 0 & n \end{bmatrix}$ which is used in the definition of a Bedford-McMullen Carpet. To show that $\dim_{M}E = \dim_{B}\varphi(X)$, we suggest restricting ourselves to the case when multiplication by $A^{-1}$ is a similarity with respect to some norm on $\mathbb{R}^{n}$. 

\section{Multiplicative Independence in $\mathbb{Z}^{n}$}\label{sec:multConj}

A pair of integers $r, s > 1$ are said to be \textit{multiplicatively independent} if the ratio $\log(r)/\log(s)$ is an irrational number. In other words, no positive integer power of $r$ is equal to some positive integer power of $s$. For example, take any pair of positive integers that are coprime. It has been proved that if a pair of integers $r, s > 1$ are multiplicatively independent and an infinite subset of $\mathbb{N}_{0}$ is simultaneously $\times r$-invariant and $\times s$-invariant, then the subset is equal to $\mathbb{N}_{0}$ \cite{GMR24}. This suggests an investigation into a version for multiplicatively invariant subsets of $\mathbb{Z}^{n}$. For example, recall that by Theorem~\ref{thm:radixExistCanon}, the pair $(b, \mathcal{D}) := (-n+i, \{0, 1, \ldots, n^{2})$ is a number system on the Gaussian integers for $n\geq1$. If an infinite subset of the Gaussian integers is simultaneously $\times (-2+i, \mathcal{D})$ and $\times (-5+i, \mathcal{D})$ invariant, then must it contain every Gaussian integer?

The result that inspired the version about subset of $\mathbb{N}_{0}$ states that if a pair of integers $r, s > 1$ are multiplicatively independent, then the only closed infinite subset of $\mathbb{T} = \faktor{\mathbb{R}}{\mathbb{Z}}$ that is $\times r$-invariant and $\times s$-invariant is itself. This result was generalized to $\mathbb{T}_{n}$ in \cite{B83}. Suppose that $A$ and $B$ are $n$ by $n$ integer matrices which commute and satisfy the following properties:
\begin{itemize}
\item[(i)] For either $A$ or $B$, the characteristic polynomial of the $k$th power of the matrix is irreducible over $\mathbb{Z}$ for every positive integer $k$,
\item[(ii)] For any common eigenvector between $A$ and $B$, at least one of the corresponding eigenvalues lies outside of the unit disc in the complex plane,
\item[(iii)] There exists no pair of non-zero integers, $m$ and $\ell$, such that $A^{m} = B^{\ell}$. 
\end{itemize}
The theorem in \cite{B83} asserts that if a closed infinite subset of $\mathbb{T}^{n}$ is invariant under the action of both matrices, then it must be $\mathbb{T}^{n}$. 

\begin{remark}
We mention that the original theorem is stated for commutative semigroups of endomorphisms on $\mathbb{T}^{n}$. 
\end{remark}

The conditions (i), (ii), and (iii) along with commutativity may be sufficient for a version to hold in $\mathbb{Z}^{n}$. 

\section{Conclusion}

I like to believe at the end of all this formalism, I've earned the right to be candid. I am fairly pleased that this thesis has something of a theme and, from Chapters 3-5, a logical progression. I say this because the truth is that the learning that went into this thesis is actually in reverse. In the fall of 2020, Florian Richter of \cite{GMR24} gave a virtual talk at Queen's University on the topic of multiplicative invariance. Given his grounding in ergodic theory with a view towards number theory, it seemed like the topic would be a good choice. I had both experience doing research in topological dynamics from my master's, had worked through Furstenberg's paper on Szemeredi's theorem, and have an advisor who specializes in the application of ergodic theory to number theory. Little did I know there would be so much fractal geometry. 

The original plan, after consultation with Florian, was to generalize the theorems about multiplicative invariance from expansions of the form $d_{1}m^{-1} + d_{2}m^{-2} + \cdots$ where $d_{j}\in\{0, 1, \ldots, m-1\}$ for each $j$ to expansions $d_{1}m_{1}^{-1} + d_{2}(m_{1}m_{2})^{-1} + \cdots$ where $d_{j} \in\{0, 1, \ldots, m_{j}-1\}$ for each $j$. This is fairly difficult because expansions of the latter type do not have a correspondence to a natural sequence space that is invariant under the left-shift operator. I remember reaching out to Pablo Schmerkin, arguably \emph{the} expert on the topic of multiplicative invariance/independence, and hearing back that while it was interesting to think about, he had no idea what the statements should even be. It is likely a problem worth somebody's time and effort, but perhaps not a graduate student who has limited funding and is on a tight schedule. 

At the suggestion of my advisors, I pivoted to the case of Furstenberg's and Florian and his co-author's theorems when the base of the expansions were not merely integers, but Gaussian integers. This is what started my journey down the rabbit hole of canonical representations of complex numbers with respect to a Gaussian integer base. This part of the project was fairly enjoyable. I remember my three-part indoctrination: I was initially vexed by Professor Gilbert's presentation of the equivalence of $(-n+i, \{0, 1, \ldots, n^{2}\})$-representations, then neutral, and then finally believing it was the only sensible way to do it. Today, I still have some fondness for it, but its bewildering complexity is best compartmentalized as a special case of the neighbour graphs that have been treated by Shigeki Akiyama and J\"{o}rg Thuswaldner, among others, for more general number systems. It was also around this time that I came across Steen Pedersen and Vincent Shaw's paper on the intersections of Cantor sets derived from $-n+i$ with their translates. I realized that Gilbert's work improved their results after staring at Gilbert's graphs for an embarrassingly long amount of time. It was neat, but I still felt multiplicative invariance was the core of my research. I had painstakingly worked through and presented proofs about the multiplicative independence that has thematic connections to the $\times 2, \times 3$ conjecture and felt that it had to lead somewhere. I then attended my first fractal geometry conference. 

Under blue skies, on sandy beaches, surrounded by friendly attendees, I realized how much of the subject I was missing. There are of course worse places and situations to be in when such a thing happens, but it was somewhat shocking to be hearing the words self-similar and self-affine for the first time and for them to be used practically every other sentence after putting considerable effort into understanding multiplicative invariance. Being shaken up like that made me reflect on multiplicative invariance and the talks that week which featured it. I realized that despite being able to follow and even fill in the technical details of the proofs, I didn't understand why any of it was happening. It was then months after the fact, months filled with the distractions of writing up what I had, TAing, and running a class, that I came to the conclusion that it was a bad idea to force myself to defend a thesis I did not yet have the mathematical maturity to comprehend what even motivated it. Meanwhile, I had a much better grasp of the results about intersections of Cantor sets with their translates. Even if it was not as sophisticated as my original project, I was making progress. Not to mention that the second specialized conference I attended on number systems was a far better experience. 

Hyperfocusing on improving Pedersen and Shaw's theorem led me to the question of the self-similarity of $T\cap(T+\alpha)$. It was there that I eventually noticed that the proof techniques used to treat the intersection really only relied on the group structure of the vector space $\mathbb{R}^{n}$. That is how I wrote my thesis in reverse. Hopefully it explains why $-n+i$ received such an extensive treatment. Given more time I would have liked to showcase the neighbour graphs of Sierpinski-like carpets. For now, the Bedford-McMullen carpet example in Chapter $3$ will have to do. Despite my change of focus, I'd be very interested to know the truth of my musings in Section~\ref{sec:multConj}. 

Thank you for reading.


\appendix
\addtocontents{toc}{\protect\setlength{\protect\cftchapnumwidth}{2.5cm}}
\addtocontents{toc}{\protect\renewcommand{\protect\cftchappresnum}{Appendix }}

\chapter{The Neighbour Graphs for $(-n+i, \{0, 1, \ldots, n^{2}\})$, $n\geq2$}

\section{Derivation of the Neighbour Graph ($n\geq3$)}\label{app:a}
This appendix is a supplement to the discussion of Figure~\ref{fig:radix} in Section~\ref{sec:neighGraphsss}. The goal of this appendix is to demonstrate how Theorem~\ref{thm:neighboursetrepeat} is used to derive the augmented neighbour graph in Figure~\ref{fig:radix}. For convenience, the graph can be found in Figure~\ref{fig:radixrepeat2} below and Theorem~\ref{thm:neighboursetrepeat} is simply a repetition of Theorem~\ref{thm:neighbourset}. 

Recall that the claim is that any triple of $(-n+i, \{0, 1, \ldots, n^{2}\})$ represent the same complex number if and only if they can be obtained from an infinite path through the neighbour graph starting from the top node (state). The diagrams for the states and the labeling system for the edges is the same as it is in Section~\ref{sec:neighGraphsss}. 

\begin{figure} 
\centering
\includegraphics[scale=0.63]{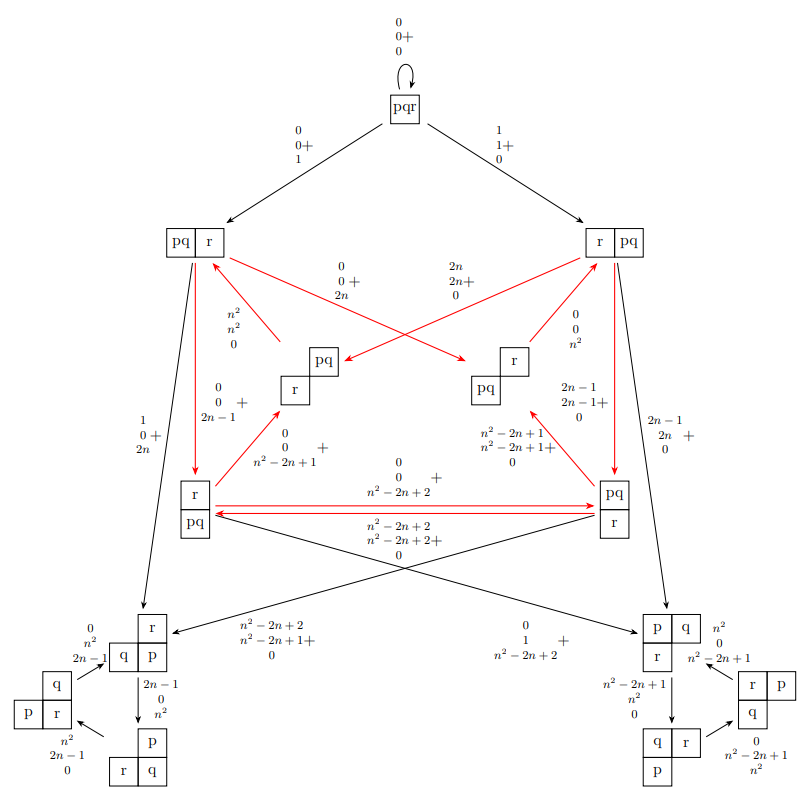}
\caption{The neighbour graph for $(-n+i, \{0, 1, \ldots, n^{2}\})$-representations, $n\geq3$.}
\label{fig:radixrepeat2}
\end{figure}

\begin{theorem}[W. Gilbert, \cite{G82}, proposition 1]\label{thm:neighboursetrepeat}
A Gaussian integer $s$ is a neighbour of $T_{-n+i, \{0, 1, \ldots, n^{2}\}}$ if and only if
\begin{itemize}
\item[(i)] $s\in\{0, \pm1, \pm (n-1+i), \pm (n+i)\}$ and $n\geq3$.
\item[(ii)] $s\in\{0, \pm1, \pm (1+i), \pm (2+i), \pm i, \pm (2+2i)\}$ and $n=2$. 
\end{itemize}
\end{theorem}

We proceed under the assumption that $n\geq3$. We discuss the special case of $n=2$ in Appendix~\ref{app:b}. In \cite{G82}, Gilbert gives some of the calculations pertaining to the $n=1$ neighbour graph. The derivation of that graph does not exhibit all the reasoning featured in the derivation of the graph governing the cases $n\geq3$. 

Let $p, q,$ and $r$ be $(-n+i, \{0, 1, \ldots, n^{2}\})$. The $k$th state is defined to be $S(k) := (\zeta_{k}, \xi_{k}, -\zeta_{k} - \xi_{k}, )_{k=1}^{\infty}$ where $(\zeta_{k})_{k=1}^{\infty}$ and $(\xi_{k})_{k=1}^{\infty}$ are the integer sequences of $(p, q)$ and $(q, r)$ respectively. 

Although the sum of the components of $S(k)$ is zero, our notation lists them all. This is because we wish to explicitly compute the digits of all three representations in the $k$th place. We recall (\ref{eq:states}),
\begin{equation} \label{eq:graphkey}
S(k+1) = (p_{k+1}-q_{k+1}, q_{k+1}-r_{k+1}, r_{k+1}-p_{k+1}) + bS(k).
\end{equation}

It says that the $k$th state can be used to find the possible values of $S(k+1)$. From our discussion in Section~\ref{sec:neighGraphsss}, recall that $S(0) = (0, 0, 0)$. This state corresponds with the top node of Figure \ref{fig:radixrepeat2} with the diagram
$$\begin{tikzpicture}\draw (0,0) rectangle node{pqr} (.75,.75); \end{tikzpicture}.$$
We compute, using Theorem~\ref{thm:neighboursetrepeat}, the possible values of $S(k+1)$. Each value will correspond to a node in the state graph that is a successor of the node corresponding to $S(k) = (0, 0, 0)$. 

Observe that by (\ref{eq:graphkey}) the $k$th state must satisfy $S(k+1) = (p_{k+1}-q_{k+1}, q_{k+1}-r_{k+1}, r_{k}-p_{k+1})$. This forces the components of $S(k+1)$ to be integers since each digit is an integer. In accordance with Theorem~\ref{thm:neighboursetrepeat}, the components must be $0$ or $\pm1$. This splits into cases. It is that either all three digits are the same ($S(k+1) = (0, 0, 0)$) or at least one digit differs from the other two. 

The case $S(k+1) = (0, 0, 0)$ implies the existence of an arrow from the state $(0, 0, 0)$ back to itself. The triple of digits $(p_{k+1}, q_{k+1}, r_{k+1})$ could be any $(a, a, a)$ where $a\in\{0, 1, \ldots, n^{2}\}$. This is indicated by the label on the corresponding edge in the state graph given by 
$$\begin{matrix} 0\\0\\0\end{matrix}+.$$
We proceed with the case of differing digits. The digits cannot all be distinct because this would mean one of the pairs would necessarily have a difference of magnitude greater than or equal to $2$. Without loss of generality, let us say that $r$ is the expansion that differs in the $k+1$st digit and $p_{k+1}=q_{k+1}$. Either $r_{k+1}$ is one more than $p_{k+1}$ or one less. We either have $S(k+1) = (0, -1, 1)$ or $S(k+1) = (0, 1, -1)$. These states correspond to the diagrams 
$$\begin{tikzpicture}\draw (0,0) rectangle node{pq} (0.75,0.75); \draw (0.75, 0.75) rectangle  node{r} (1.5, 0); \end{tikzpicture}\;\;\text{and}\;\;\begin{tikzpicture}\draw (0,0) rectangle node{r} (0.75,0.75); \draw (0.75, 0.75) rectangle  node{pq} (1.5, 0); \end{tikzpicture}$$ respectively and result in the remaining two edges from the top node in Figure~\ref{fig:radixrepeat2}. 

The triples $(p_{k+1}, q_{k+1}, r_{k+1})$ are either of the form $(a, a, a + 1)$ or $(a + 1, a + 1, a)$ where $a\in\{0, 1, \ldots, n^{2}-1\}$. This is indicated by the respective labels 
$$\begin{matrix} 0\\0\\1\end{matrix}+\;\;\text{and}\;\;\begin{matrix} 1\\1\\0\end{matrix}+$$
 on the corresponding edges.

This first step provides the flavour of the calculations that appear in the full derivation of the graph. We compute a second step which will include the possibility that all three of the digits $p_{k+1}, q_{k+1}$, and $r_{k+1}$ are distinct. Let us re-index such that $S(k) = (0, 1, -1)$. Again, we refer to (\ref{eq:graphkey}) to direct our calculations. We have
\begin{equation}\label{eq:potentialtriple}
S(k+1) = (p_{k+1}-q_{k+1}, q_{k+1}-r_{k+1}, r_{k+1}-p_{k+1}) + (0, -n+i, n-i).
\end{equation}
It is clear that, at least one of the digits must differ from the other two. Let us investigate the case of exactly one distinct digit. Without loss of generality we assume $p_{k+1} = q_{k+1}$ and $r_{k+1}\neq p_{k+1}$. Consider the second component of $S(k+1)$: $q_{k_1} - r_{k+1} - n + i$. 

The digits are integers and thus there is no way of changing the positive imaginary part. According to Theorem~\ref{thm:neighboursetrepeat}, we can choose digits $q_{k+1}$ and $r_{k+1}$ such that $q_{k+1} - r_{k+1} = 2n$ or $2n-1$. The choice of a difference of $2n$ implies that the third component is $-n-i$. The resulting state is $S(k+1) = (0, n+i, -n-i)$. Its corresponding diagram in Figure~\ref{fig:radixrepeat2} is 
$$\begin{tikzpicture}\draw (0,0) rectangle node{r} (0.75,0.75); \draw (0.75, 0.75) rectangle  node{pq} (1.5, 1.5); \end{tikzpicture}.$$

The triple of digits $(p_{k+1}, q_{k+1}, r_{k+1})$ is of the form $(2n + a, 2n + a, a)$ where $a\in\{0, 1, \ldots, n^{2}-2n\}$. This is indicated by the label on the corresponding edge given by 
$$\begin{matrix} 2n\\2n\\0\end{matrix}+.$$

If we made the other choice, the resulting state is $S(k+1) = (0, n-1+i, -n+1-i)$ whose diagram is given by
$$\begin{tikzpicture}\draw (0,0) rectangle node{pq} (0.75,0.75); \draw (0, 0) rectangle  node{r} (.75, -.75); \end{tikzpicture}\;\;\text{and has the label}\;\;\begin{matrix} 2n-1\\2n-1\\0\end{matrix}+$$ on the incoming edge. 

Now we consider the case where all three are different and, in particular, $p_{k+1}\neq q_{k+1}$. We can see in (\ref{eq:potentialtriple}) that the first component of $S(k+1)$ is precisely $p_{k+1}-q_{k+1}$. It follows from Theorem~\ref{thm:neighboursetrepeat} that either $p_{k+1}$ is one more than $q_{k+1}$ or one less. The expansions $p$ and $q$ have the same digits for all places $k+j$ for all $j\geq2$. We are distinguishing them for the first time. Without loss of generality we may assume $p_{k+1} = q_{k+1} - 1$. 

In order for the remaining components of $S(k+1)$ to obey Theorem~\ref{thm:neighboursetrepeat}, we must have $q_{k+1} - r_{k+1} = 2n$ and thus $r_{k+1} - p_{k+1} = -2n+1$. The resulting state is $S(k+1) = (1, n-1+i, -n-i)$ and its corresponding diagram is $$\begin{tikzpicture}\draw (0,0) rectangle node{r} (0.75,0.75);\draw (0.75, 0.75) rectangle node{p} (0, 1.5); \draw (0.75, 0.75) rectangle node{q} (1.5, 1.5); \end{tikzpicture}.$$

The remaining structure of the neighbour graph can be deduced by iterating this procedure until all the successive states $S(k+1)$ are found. We leave this task to the interested reader. 

\section{The Remaining Graph ($n=2$)}\label{app:b}

Here we present the state graph governing equivalent radix expansions in base $-2+i$. 

The neighbour sequence of two representations may take on a larger number of values when $n=2$. This increases the number of realizable states and thus complicates the corresponding neighbour graph. The method used to derive the state graph for $n\geq3$ applies in the case $n=2$. We do not include the details. We include the notation required to parse the diagrams for the new states in the state graph, the primary claim from \cite{G82} about the graph (Theorem~\ref{thm:rules2}), and the graph itself (Figure~\ref{fig:radix2}). The new edges particular to $n=2$ are highlighted in blue and any successor of a blue edge is also a new state particular to the $n=2$ case. 

We make special mention that we only label the edges that correspond to the first distinction between a pair of expansions. The interested reader can 
derive any edge label using the value of the source and successor states of the edge and (\ref{eq:graphkey}). 

Let $p$ and $q$ be two equivalent $(-2+i, \{0, 1, 2, 3, 4, 5\})$-representations. We extend the list of diagrams that communicate the entries of the neighbour sequence of $p$ and $q$. The additions are as follows:
\begin{enumerate}
\item[(v)]  $p(k)-q(k) = i$ corresponds to \begin{tikzpicture}\draw (0,0) rectangle node{q} (.75,.75); \draw (0, .75) rectangle node{p} (-0.75, 1.5); \end{tikzpicture}.
\item[(vi)] $p(k)-q(k) = 2+2i$ corresponds to  \begin{tikzpicture}\draw (0,0) rectangle node{q} (0.75,0.75); \draw (0, 1.5) rectangle  node{p} (0.75, 2.25); \end{tikzpicture}. 
\end{enumerate}
We can communicate the value of additional states using these diagrams. For example, the state $(-1-i, 1+i, -2-2i)$ is communicated by the diagram
$$\begin{tikzpicture}\draw (0,0) rectangle node{r} (0.75,0.75);\draw (0.75,0.75) rectangle node{q} (0,1.5);\draw (0,1.5) rectangle node{p} (0.75,2.25);\end{tikzpicture}.$$

\begin{theorem}[W. J. Gilbert, \cite{G82}, theorem 8]\label{thm:rules2}
Let $p, q$ and $r$ be three $(-2+i, \{0, 1, 2, 3, 4\})$-representations. These representations are equivalent if and only if they can be obtained from an infinite path through the state graph in Figure~\ref{fig:radix2} starting at state $(0, 0, 0)$, if necessary relabeling $p, q$ and $r$ and in some cases, when $p=q$, replacing $q$ with another expansion. 
\end{theorem}

\begin{figure} [h]
\centering
\includegraphics[scale=0.63]{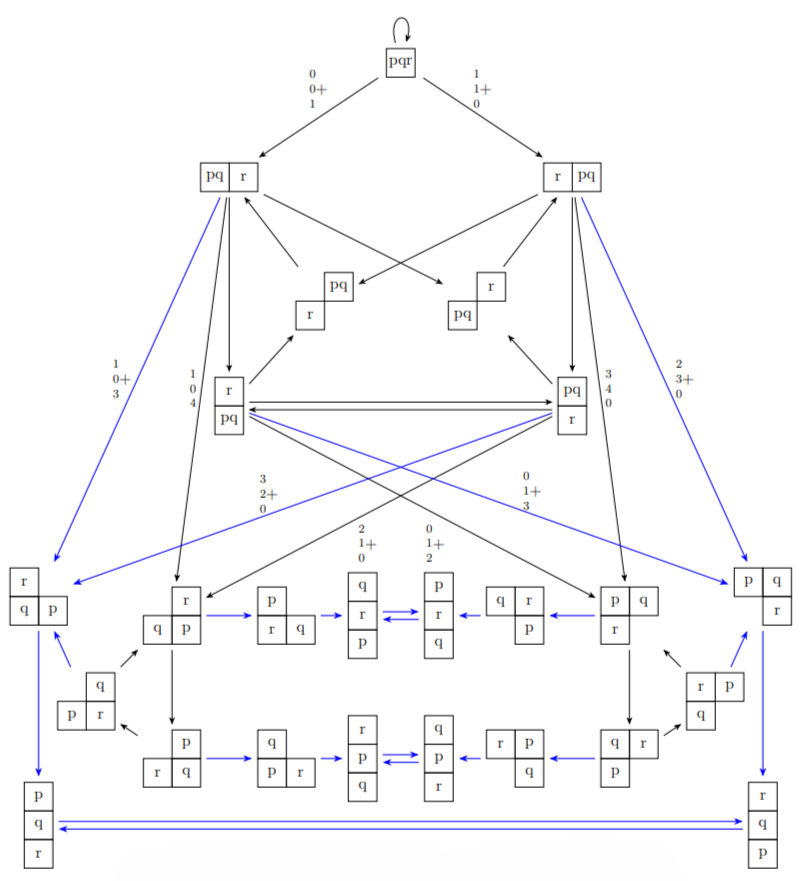}
\caption{The neighbour graph for $(-2+i, \{0, 1, 2, 3, 4\})$-representations.}
\label{fig:radix2}
\end{figure}

\end{document}